\pdfoutput=1
\documentclass[hidelinks, reqno]{amsart}

\usepackage[numbers, square]{natbib}
\setcitestyle{citesep={,}}

\usepackage[english,french]{babel}
\usepackage{graphicx}
\usepackage{subcaption}
\usepackage{tabularx}
\usepackage{booktabs}
\usepackage{array}
\usepackage{amsmath}
\usepackage{amsfonts}
\usepackage{amssymb}
\usepackage{amsthm}
\usepackage[scr]{rsfso}
\usepackage{enumitem}
\usepackage{permute}
\usepackage{slashed}
\usepackage[usenames,dvipsnames]{xcolor}
\usepackage[pagebackref = true, colorlinks, linkcolor = Red, citecolor = Green, bookmarksdepth=2, linktocpage=true]{hyperref}
\usepackage[capitalize]{cleveref}
\usepackage{comment}

\usepackage[T1]{fontenc}
\usepackage{makecell}

\usepackage{tikz}
\usetikzlibrary{calc}

\allowdisplaybreaks

\DeclareMathOperator{\interior}{int}
\DeclareMathOperator{\exterior}{ext}
\DeclareMathOperator{\Hom}{Hom}

\DeclareMathOperator{\Id}{Id}

\DeclareMathOperator{\Stab}{Stab}

\DeclareMathOperator{\supp}{supp}

\DeclareMathOperator{\SL}{SL}

\DeclareMathOperator{\U}{U}

\DeclareMathOperator{\Lip}{Lip}

\DeclareMathOperator{\proj}{proj}

\DeclareMathOperator{\diam}{diam}

\DeclareMathOperator{\len}{len}
\let\Pr\relax
\DeclareMathOperator{\Pr}{Pr}

\DeclareMathOperator{\Ad}{Ad}

\DeclareMathOperator{\Span}{span}

\DeclareMathOperator{\Fix}{Fix}

\newcommand{\C}{\mathbb{C}}
\newcommand{\N}{\mathbb{N}}

\newcommand{\R}{\mathbb{R}}
\newcommand{\Z}{\mathbb{Z}}

\newcommand{\rank}{\mathsf{r}}

\newcommand{\LieG}{\mathfrak{g}}
\newcommand{\LieA}{\mathfrak{a}}
\newcommand{\LieN}{\mathfrak{n}}
\newcommand{\LieK}{\mathfrak{k}}
\newcommand{\LieM}{\mathfrak{m}}
\newcommand{\LieP}{\mathfrak{p}}

\newcommand{\Fboundary}{\mathcal{F}}
\newcommand{\Gboundary}{\partial\Gamma}
\newcommand{\involution}{\mathsf{i}}
\newcommand{\growthindicator}{\psi_\Gamma}

\newcommand{\limitset}{\Lambda_\Gamma}
\newcommand{\limitcone}{\mathcal{L}_\Gamma}
\newcommand{\BMS}{m_\mathsf{v}^\mathrm{BMS}}

\DeclareFontFamily{U}{mathb}{\hyphenchar\font45}
\DeclareFontShape{U}{mathb}{m}{n}{
	<5> <6> <7> <8> <9> <10> gen * mathb
	<10.95> mathb10 <12> <14.4> <17.28> <20.74> <24.88> mathb12
}{}
\DeclareSymbolFont{mathb}{U}{mathb}{m}{n}
\DeclareMathSymbol{\bigast}{2}{mathb}{"06}

\def\XXint#1#2#3{{\setbox0=\hbox{$#1{#2#3}{\int}$}
		\vcenter{\hbox{$#2#3$}}\kern-.5\wd0}}

\theoremstyle{plain}
\newtheorem{theorem}{Theorem}[section]
\newtheorem{proposition}[theorem]{Proposition}
\newtheorem{lemma}[theorem]{Lemma}
\newtheorem{corollary}[theorem]{Corollary}

\theoremstyle{definition}
\newtheorem{definition}[theorem]{Definition}

\theoremstyle{remark}
\newtheorem{remark}[theorem]{Remark}

\setlist[enumerate,1]{ref=(\arabic*)}
\setlist[enumerate,2]{ref=(\theenumi)(\alph*)}
\setlist[enumerate,3]{ref=(\theenumi)(\theenumii)(\roman*)}
\setlist[enumerate,4]{ref=(\theenumi)(\theenumii)(\theenumiii)(\Alph*)}

\newlist{alternative}{enumerate}{4}     % this creates a dedicated counter named 'subtaski'
\setlist[alternative,1]{label=(\arabic*), ref=(\arabic*)}
\setlist[alternative,2]{label=(\alph*), ref=(\thealternativei)(\alph*)}
\setlist[alternative,3]{label=(\roman*), ref=(\thealternativei)(\thealternativeii)(\roman*)}
\setlist[alternative,4]{label=(\Alph*), ref=(\thealternativei)(\thealternativeii)(\thealternativeiii)(\Alph*)}

\Crefname{enumi}{Property}{Properties}
\Crefname{alternativei}{Alternative}{Alternatives}
\Crefname{subsection}{Subsection}{Subsections}

\begin{document}
	\selectlanguage{english}
	
	%\setlength{\parindent}{0 in}
	%\setlength{\mathindent}{0.5 in}
	%\numberwithin{equation}{section}
	%\renewcommand*{\abstractname}{Introduction}
	
	\title[Local mixing of one-parameter diagonal flows]{Local mixing of one-parameter diagonal flows on Anosov homogeneous spaces}
	
	\author{Michael Chow}
	\address{Department of Mathematics, Yale University, New Haven, Connecticut 06511}
	\email{mikey.chow@yale.edu}
	
	\author{Pratyush Sarkar}
	\address{Department of Mathematics, Yale University, New Haven, Connecticut 06511}
	\email{pratyush.sarkar@yale.edu}
	
	\date{\today}
	
	\begin{abstract}
		Let $G$ be a connected semisimple real algebraic group and $\Gamma < G$ be a Zariski dense Anosov subgroup with respect to a minimal parabolic subgroup. We prove local mixing of the one-parameter diagonal flow $\{\exp(t\mathsf{v}) : t \in \R\}$ on $\Gamma \backslash G$ for any interior direction $\mathsf{v}$ of the limit cone of $\Gamma$ with respect to the Bowen--Margulis--Sullivan measure associated to $\mathsf{v}$. More generally, we allow a class of deviations to this flow along a direction $\mathsf{u}$ in some fixed subspace transverse to $\mathsf{v}$. We also obtain a uniform bound for the correlation function which decays exponentially in $\|\mathsf{u}\|^2$. The precise form of the result is required for several applications such as the asymptotic formula for the decay of matrix coefficients in $L^2(\Gamma \backslash G)$ proved by Edwards--Lee--Oh.
	\end{abstract}
	
	\maketitle
	
	\selectlanguage{english}
	
	\setcounter{tocdepth}{1}
	\tableofcontents

	\section{Introduction}
	Let $G$ be a connected semisimple real algebraic group of real rank $\rank$. Fix a minimal parabolic subgroup $P < G$ and a Langlands decomposition $P = MAN^-$ where $A$ is a maximal real split torus, $N^-$ is a maximal horospherical subgroup, and $M$ is a compact subgroup commuting with $A$. Let $N^+ < G$ be the opposite horospherical subgroup. Let $\Fboundary = G/P$ be the corresponding Furstenberg boundary and $\Fboundary^{(2)}$ be the unique open $G$-orbit in $\Fboundary \times \Fboundary$. Let $\Gamma < G$ be a Zariski dense \emph{Anosov subgroup} with respect to $P$. The Anosov property (with respect to any parabolic subgroup of $G$) was first introduced by Labourie \cite{Lab06} for surface groups and later generalized by Guichard--Wienhard \cite{GW12} for Gromov hyperbolic groups. They also showed that a Zariski dense discrete subgroup of $G$ which is Gromov hyperbolic is Anosov with respect to $P$ if it admits a continuous $\Gamma$-equivariant map $\zeta$ from the Gromov boundary $\Gboundary$ to $\Fboundary$ such that $(\zeta(x), \zeta(y)) \in \Fboundary^{(2)}$ for all $x \neq y \in \Gboundary$ \cite[Corollary 4.16]{GW12}. The quotient space $\Gamma \backslash G$ is called an \emph{Anosov homogeneous space}. Let $\LieA$ be the Lie algebra of $A$ and $\LieA^+ \subset \LieA$ be a choice of a closed positive Weyl chamber. For any $v \in \LieA$, we denote $a_v = \exp(v) \in A$. For any unit vector $\mathsf{v} \in \LieA^+$, consider the one-parameter diagonal flow on $\Gamma \backslash G$ given by the right translation action of $\{a_{t\mathsf{v}}: t \in \mathbb R\}$. We are interested in obtaining a local mixing theorem for this flow in the sense of the following definition originally due to Oh--Pan \cite[Definition 1.4]{OP19}.
	
	\begin{definition}[Local mixing]
		We say that a continuous map $a: \Gamma \backslash G \times \mathbb R \to \Gamma \backslash G$ has the \emph{local mixing property} with respect to a Borel measure $\mathsf{m}$ on $\Gamma \backslash G$ if there exists a positive continuous \emph{normalizing function} $\alpha: \R \to \R$ such that for all $\phi_1, \phi_2 \in C_{\mathrm{c}}(\Gamma \backslash G)$, we have
		\begin{align*}
			\lim_{t \to +\infty} \alpha(t) \int_{\Gamma \backslash G} \phi_1(a(x, t)) \phi_2(x) \, d\mathsf{m}(x) = \mathsf{m}(\phi_1) \mathsf{m}(\phi_2).
		\end{align*}
	\end{definition}
	
	For each $\mathsf{v}$ in the interior of the limit cone $\limitcone \subset \LieA^+$ introduced by Benoist \cite{Ben97}, there exists a corresponding Bowen--Margulis--Sullivan (BMS) measure $\BMS$ on $\Gamma \backslash G$ \cite{Qui02a} (see \cref{subsec:BMSmeasures}). Let $\mathfrak{Y}_\Gamma$ be the set of $P^\circ$-minimal subsets of $\Gamma \backslash G$, where $P^\circ$ is the identity component of $P$. Fix any $\mathcal{E}_0 \in \mathfrak{Y}_\Gamma$ and define the normal subgroup $M_\Gamma = \Stab_M(\mathcal{E}_0) < M$ with respect to the right translation action \cite{GR07,LO20a}. Note that $\#\mathfrak{Y}_\Gamma \leq [M : M^\circ]$. Let $\Omega = \supp(\BMS) = \{[g] \in \Gamma \backslash G: g^\pm \in \limitset\}$. Fix $\Omega_0 = \Omega \cap \mathcal{E}_0$ and define $\Omega_{[m]} = \Omega_0m$ for all $[m] \in M_\Gamma \backslash M$. Then, the restrictions $\BMS|_{\Omega_{[m]}}$ for all $[m] \in M_\Gamma \backslash M$ are precisely the $A$-ergodic components of $\BMS$ as shown by Lee--Oh \cite{LO20a}. We now present the main theorem of this paper regarding local mixing of the one-parameter diagonal flow on $\Omega_{[m]}$ with respect to $\BMS$ for all $[m] \in M_\Gamma \backslash M$.
	
	\begin{theorem}
		\label{thm:LocalMixingNoPsi}
		Let $\mathsf{v} \in \interior(\limitcone)$. There exists $\kappa_\mathsf{v} > 0$ such that for all $\phi_1, \phi_2 \in C_{\mathrm{c}}(\Gamma \backslash G)$, we have 
		\begin{multline*}
			\lim_{t \to +\infty} t^{\frac{\rank - 1}{2}} \int_{\Gamma \backslash G} \phi_1(x a_{t\mathsf{v}}) \phi_2(x) \, d\BMS(x) \\
			= \kappa_\mathsf{v}\sum_{[m] \in M_\Gamma \backslash M} \BMS\bigr|_{\Omega_{[m]}}(\phi_1) \cdot \BMS\bigr|_{\Omega_{[m]}}(\phi_2).
		\end{multline*}
	\end{theorem}
	
	Such local mixing results are of interest not only from a dynamical point of view but also due to several applications such as orbit counting, equidistribution, and measure classification results \cite[etc.]{DRS93,EM93,Oh14,MMO14,Win15,Sam15,OP19,ELO20,ELO22a}. We also mention a recent new application of \cref{thm:LocalMixingNoPsi} to the Hopf--Tsuji--Sullivan dichotomy \cite{BLLO21}. Indeed, for applications to orbit counting and equidistribution results as in \cite{Sam15,ELO20,ELO22a}, the following more refined version of \cref{thm:LocalMixingNoPsi} is required. For any $\mathsf{v} \in \interior(\limitcone)$, we denote by $\psi_\mathsf{v} \in \LieA^*$ the unique linear form such that $\psi_\mathsf{v} \geq \growthindicator$ and $\psi_\mathsf{v}(\mathsf{v}) = \growthindicator(\mathsf{v})$, where $\growthindicator$ is the growth indicator function introduced by Quint \cite{Qui02a} (see \cref{subsec:LimitSetAndLimitCone}).
	
	\begin{theorem}
		\label{thm:LocalMixing}
		Let $\mathsf{v} \in \interior(\limitcone)$. There exist $\kappa_\mathsf{v} > 0$ and an inner product $\langle \cdot, \cdot \rangle_*$ on $\LieA$ such that for all positive continuous functions $r: \R \to \R$ with $\lim\limits_{t \to +\infty}\frac{r(t)}{t} \in [0, +\infty]$ and $\lim\limits_{t \to +\infty}\frac{r(t)^2}{t} =: \ell \in [0, +\infty]$, $\mathsf{u} \in \ker\psi_\mathsf{v}$, and $\phi_1, \phi_2 \in C_{\mathrm{c}}(\Gamma \backslash G)$, we have
		\begin{multline*}
			\lim_{t \to +\infty} t^{\frac{\rank - 1}{2}} \int_{\Gamma \backslash G} \phi_1(x a_{t\mathsf{v} + r(t)\mathsf{u}}) \phi_2(x) \, d\BMS(x) \\
			= \kappa_\mathsf{v} e^{-\ell I(\mathsf{u})} \sum_{[m] \in M_\Gamma \backslash M} \BMS\bigr|_{\Omega_{[m]}}(\phi_1) \cdot \BMS\bigr|_{\Omega_{[m]}}(\phi_2)
		\end{multline*}
		with the convention that $+\infty\cdot 0 = 0$ in the exponent, where $I: \ker\psi_\mathsf{v} \to \R_{\geq 0}$ is defined by $I(\mathsf{u}) = \langle\mathsf{u}, \mathsf{u}\rangle_* - \frac{\langle \mathsf{u}, \mathsf{v} \rangle_*^2}{\langle \mathsf{v}, \mathsf{v}\rangle_*}$ for all $\mathsf{u} \in \ker\psi_\mathsf{v}$. Moreover:
		\begin{enumerate}
			\item\label{itm:UniformBoundOnCones} there exists $\kappa_\mathsf{v}(\phi_1, \phi_2) > 0$ such that if $\ell \in (0, +\infty]$, then there exist $T_r > 0$ and $\ell' \in (0, \ell)$ such that the left hand side in absolute value is bounded above by $\kappa_\mathsf{v}(\phi_1, \phi_2) e^{-\ell' I(\mathsf{u})}$ for all $(t, \mathsf{u}) \in [T_r, +\infty) \times \ker\psi_\mathsf{v}$ with $t\mathsf{v} + r(t)\mathsf{u} \in \LieA^+$;
			\item\label{itm:UniformBoundOnCompacts} for all compact subsets $\mathcal{K} \subset \ker\psi_\mathsf{v}$, the convergence of the left hand side is uniform in $\mathsf{u} \in \mathcal{K}$.
		\end{enumerate}
	\end{theorem}
	
	\begin{remark}
		See \cref{thm:LocalMixingForErgodicComponent,rem:ExponentialDecayAlongRays,rem:OnProofsOfProperties,rem:ConstantKappa} for further details.
	\end{remark}
	
	In particular, the following is a theorem of Edwards--Lee--Oh regarding an asymptotic formula for the decay of matrix coefficients in $L^2(\Gamma \backslash G)$ for compactly supported continuous functions (see \cite[Theorem 7.12]{ELO20} and \cite[Theorems 3.1 and 3.4]{ELO22b}). They prove it by developing a higher rank version of Roblin's transverse intersection argument \cite{Rob03} (see also \cite{OS13,MO15}) and using \cref{thm:LocalMixing} as input. We fix some Haar measure on $G$ which induce a $G$-invariant measure on $\Gamma \backslash G$ and Burger--Roblin measures $m_{\involution(\mathsf{v})}^\mathrm{BR}$ and $m_{\mathsf{v}}^{\mathrm{BR}_*}$ corresponding to each $\mathsf{v} \in \interior(\limitcone)$ as defined in \cite[Section 3]{ELO20} and \cite[Section 2]{ELO22b}. Denote by $\rho$ half the sum of the positive roots with respect to $\mathfrak{a}^+$.
	
	\begin{theorem}[{\citealp[Theorem 3.4]{ELO22b}}]
		\label{thm:decayofmatrixcoefficients}
		Using the same notation as in \cref{thm:LocalMixing}, we have
		\begin{multline*}
			\lim_{t \to +\infty} t^{\frac{\rank - 1}{2}}e^{(2\rho - \psi_\mathsf{v})(t\mathsf{v} + r(t)\mathsf{u})} \int_{\Gamma \backslash G} \phi_1(x a_{t\mathsf{v} + r(t)\mathsf{u}}) \phi_2(x) \, dx \\
			= \kappa_\mathsf{v} e^{-\ell I(\mathsf{u})} \sum_{[m] \in M_\Gamma \backslash M} m_{\involution(\mathsf{v})}^\mathrm{BR}\bigr|_{\Omega_{[m]}N^+}(\phi_1)\cdot m_{\mathsf{v}}^{\mathrm{BR}_*}\bigr|_{\Omega_{[m]}N^-}(\phi_2).
		\end{multline*}
		Moreover, \cref{itm:UniformBoundOnCones,itm:UniformBoundOnCompacts} in \cref{thm:LocalMixing} also hold.
	\end{theorem}
	
	\begin{remark}
		The uniformity statement in \cite[Theorem 3.4]{ELO22b} was written before \cref{itm:UniformBoundOnCones,itm:UniformBoundOnCompacts} in \cref{thm:LocalMixing} appeared, but with this new input, their proof gives \cref{thm:decayofmatrixcoefficients}.
	\end{remark}
	
	In the rank one case, Zariski dense Anosov subgroups coincide with Zariski dense convex cocompact subgroups and \cref{thm:LocalMixingNoPsi} was obtained by Babillot \cite{Bab02} for $M$-invariant functions and by Winter \cite{Win15} in general. Their results hold more generally whenever the BMS measure is finite. For coabelian subgroups of convex cocompact subgroups, local mixing was proved by Oh--Pan \cite{OP19}.
	
	When $\Gamma$ is a Schottky subgroup of $\SL_n(\mathbb R)$, and $\mathsf{v}$ is along a special direction called the maximal growth direction, \cref{thm:LocalMixing} was established by Thirion \cite{Thi07,Thi09} for $M$-invariant functions using renewal theory \cite{Bab88}. His result holds slightly more generally for Zariski dense Ping-Pong subgroups which are not necessarily Anosov as parabolic elements are allowed. Extending Thirion's argument for $M$-invariant functions, Sambarino \cite{Sam15} generalized the result  (without any mention of \cref{itm:UniformBoundOnCones,itm:UniformBoundOnCompacts})  when $\Gamma$ is the fundamental group of a compact negatively curved manifold (however, see \cref{rem:Modulus1EigenvalueProblem} about the proof).
	
	\subsection*{On the proofs}
	The one-parameter diagonal flow $\{a_{t\mathsf{v}}: t \in \mathbb R\}$ on $\Gamma \backslash G$ can be better understood by studying a translation flow induced by the linear form $\psi_\mathsf{v}$. Our starting point is the work of Bridgeman--Canary--Labourie--Sambarino \cite{BCLS15} on the construction of a metric Anosov flow associated to a projective Anosov representation of $\Gamma$ which is also conjugate to a reparametrization of the Gromov geodesic flow associate to $\Gamma$. We use techniques from \cite{BCLS15,BCLS18} and properties specific to Zariski dense Anosov subgroups, such as the Morse property \cite{KLP17} and an analogue of Sullivan's shadow lemma \cite{Thi07, LO20b}, to provide a detailed proof that the translation flow is conjugate to a reparametrization of the same Gromov geodesic flow. This metric Anosov flow has a Markov section which can be transferred to a suitable Markov section for the translation flow with respect to the projection of the horospherical foliations on $\Gamma \backslash G$. This allows us to use the techniques of symbolic dynamics and thermodynamic formalism.
	
	We introduce the \emph{first return vector map} $\mathsf{K}: \Sigma \to \LieA$ which is a generalization of the first return time map for the Markov section. It is cohomologous to the map $K$ in \cite[Section 3]{Sam15}, but $\mathsf{K}$ factors through $\Sigma^+ \to \LieA$. We also introduce the holonomy $\vartheta: \Sigma \to M_\Gamma$ which factors through $\Sigma^+ \to M_\Gamma$, where $M_\Gamma$ is a normal subgroup of $M$ of finite index such that $P_\Gamma := M_\Gamma AN$ is the stabilizer of the $P^\circ$-minimal subsets of $\Gamma \backslash G$. Both $\mathsf{K}$ and $\vartheta$ are defined using a section $F: R \to \Omega_0 := Y \cap \supp\bigl(\BMS\bigr)$ corresponding to $\Gamma \backslash G \to \Gamma \backslash G/M \supset \limitset^{(2)} \times \LieA \to \limitset^{(2)} \times \R$ over the Markov section, where $Y$ is a $P^\circ$-minimal subset of $\Gamma \backslash G$. The section $F$ requires a careful construction in order to ensure that the holonomy is a $M_\Gamma$-valued map. We use $\mathsf{K}$ and $\vartheta$ to define the transfer operators with holonomy.
	
	We then prove the key technical \cref{thm:SpectralBound} which provides bounds and constraints on the spectra of the transfer operators with holonomy. Together with perturbation theory, this allows us to analytically extend its Neumann series. The proofs of the spectral properties of the transfer operators with holonomy rely on a few density results regarding $\mathsf{K}$ and $\vartheta$ which take into account the fact that $M_\Gamma$ may be different from $M$. The density results are proved using the density of the subgroup generated by the generalized length spectrum \cite{GR07, LO20a} and the $A$-ergodic decomposition of BMS measures \cite{LO20a}.
	
	\begin{remark}
		\label{rem:Modulus1EigenvalueProblem}
		The density results are not sufficient to eliminate all the modulus $1$ eigenvalues of the transfer operators with holonomy in \cref{thm:SpectralBound}. This problem also exists without holonomy (when proving \cref{thm:LocalMixing} for $M$-invariant functions as in \cite{Sam15}, including for arbitrary cocycles). For the Iwasawa cocycle however, \cite[Proposition 3.5]{Qui05}, which is stronger than the density of the subgroup generated by the length spectrum, is available to eliminate all modulus $1$ eigenvalues \cite[Pages 157--158]{Thi07} (but not when including holonomy). Thus, a complete analysis is needed in this paper to ensure that the exceptional eigenvalues do not cause problems.
	\end{remark}
	
	Another result on the transfer operators with holonomy that we need is the second order expansion of their maximal simple eigenvalue. These spectral properties of the transfer operators with holonomy are then used to prove the main theorem by carrying out the necessary Fourier analysis in a similar fashion as in \cite[Appendix A]{Thi07} and \cite{OP19}.
	
	\subsection*{Organization of the paper}
	In \cref{sec:Preliminaries}, we establish basic notations, introduce the one-parameter diagonal flow, and recall the construction of BMS measures. 
	In \cref{sec:AnosovSubgroups}, we recall the definition of a Zariski dense Anosov subgroup and properties of BMS measures that are specific to Zariski dense Anosov subgroups, including the $A$-ergodic decomposition of BMS measures. We also project the one-parameter diagonal flow to a translation flow which is the primary object of study in \cref{sec:CodingTheTranslationFlow}. 
	
	In \cref{sec:CodingTheTranslationFlow}, we prove that the translation flow is conjugate to a reparametrization of the Gromov geodesic flow associated to $\Gamma$ and construct a Markov section which provides a suitable coding for the flow. In \cref{sec:FirstReturnVectorMapAndHolonomy}, we introduce the first return vector map and the holonomy and prove some essential properties. This is used in \cref{sec:SpectraOfTheTransferOperatorsWithHolonomy} to define the transfer operators with holonomy and study their spectral properties. In \cref{sec:LocalMixing}, we use these properties to prove the main theorem. In \cref{sec:TopologicalMixing}, we give an application to topological mixing of the one-parameter diagonal flow.
	
	\subsection*{Acknowledgements}
	We are grateful to our advisor, Hee Oh, for suggesting this problem and being a source of inspiration for us throughout the work. We also thank her for the numerous helpful discussions about her recent works and the surrounding literature and many invaluable suggestions that greatly improved the quality of our writing. In particular, we thank her for emphasizing the importance of obtaining the second bound in \cref{thm:LocalMixing} which is crucial for certain applications to counting problems and pointing out an application of our result which is included in \cref{sec:TopologicalMixing}. We also thank Minju Lee, Andr\'{e}s Sambarino, and Le\'{o}n Carvajales for useful correspondences. Finally, we thank the referees for their thorough review and many helpful comments.
	
	\section{Preliminaries}
	\label{sec:Preliminaries}
	Let $G$ be a connected semisimple real algebraic group, i.e., a Lie group which is the identity component of the group of real points of a semisimple linear algebraic group defined over $\mathbb R$. We denote its Lie algebra by $\LieG$ and its real rank by $\rank$. Let $\Theta: \LieG \to \LieG$ be a Cartan involution and $\LieG = \LieK \oplus \LieP$ be the associated eigenspace decomposition corresponding to the eigenvalues $+1$ and $-1$ respectively. Let $K < G$ be the maximal compact subgroup whose Lie algebra is $\mathfrak{k}$. Let $\LieA \subset \LieP$ be a maximal abelian subalgebra and $\Phi \subset \LieA^*$ be its restricted root system. Note that $\dim(\LieA) = \rank$. Identifying $\LieA \cong \LieA^*$ via its Killing form, let $\Phi^\pm \subset \Phi$ be a choice of sets of positive and negative roots and $\LieA^+ \subset \LieA$ be the corresponding closed positive Weyl chamber. We have the associated restricted root space decomposition
	\begin{align*}
		\LieG = \LieA \oplus \LieM \oplus \LieN^+ \oplus \LieN^- = \LieA \oplus \LieM \oplus \bigoplus_{\alpha \in \Phi} \LieG_\alpha
	\end{align*}
	where $\LieM = Z_{\LieK}(\LieA) \subset \LieK$ and $\LieN^\pm = \bigoplus_{\alpha \in \Phi^\mp} \LieG_\alpha$. Define the Lie subgroups of $G$ by
	\begin{align*}
		A &= \exp(\LieA), & N^\pm &= \exp(\LieN^\pm).
	\end{align*}
	Let $M = Z_K(A)$ which need not be connected. Define the closed subset $A^+ = \exp(\LieA^+) \subset A$. Denote
	\begin{align*}
		a_v = \exp(v) \in A \qquad \text{for all $v \in \LieA$}.
	\end{align*}
	We fix a reference point $o = [K]$ in the associated symmetric space $G/K$. We fix a left $G$-invariant and right $K$-invariant Riemannian metric on $G$ and denote the corresponding inner product and norm on any of its tangent spaces by $\langle \cdot, \cdot \rangle$ and $\|\cdot\|$ respectively. This induces a left $G$-invariant and right $K$-invariant metric $d$ on $G$. We use the same notations for inner products, norms, and metrics induced on any other quotient spaces. The Riemannian metric on $G$ also induces an inner product and norm on $\LieA$, which we denote by $\langle \cdot, \cdot \rangle$ and $\|\cdot\|$ respectively, which is invariant under the Weyl group $N_K(A)/M$.
	
	The \emph{one-parameter diagonal flow} $\mathcal{W}_v: G \times \mathbb R \to G$ in any nontrivial direction $v \in \LieA^+$ is defined by the right translation action $\mathcal{W}_{v, t}(g) = ga_{tv}$ for all $g \in G$ and $t \in \mathbb R$. It is said to be \emph{regular} if $v \in \interior(\LieA^+)$. Whenever this action descends to some quotient of $G$, we use the same terminology and notation.
	
	Note that $N^\pm$ are the expanding and contracting \emph{horospherical subgroups}, i.e.,
	\begin{align}
		\label{eqn:HorosphericalSubgroups}
		N^\pm = \left\{n^\pm \in G: \lim_{t \to \pm \infty} a_{tv} n^\pm a_{-tv} = e\right\}
	\end{align}
	for any $v \in \interior(\LieA^+)$.
	
	Let $P^\pm = MAN^\pm$ which are opposing minimal parabolic subgroups of $G$. With respect to the choice $P = P^-$, we have the \emph{Furstenberg boundary} of $G$ denoted by
	\begin{align*}
		\Fboundary = G/P \cong K/M.
	\end{align*}
	Let $w_0 \in K$ be a representative of the element in the Weyl group $N_K(A)/M$ such that $\Ad_{w_0}(\LieA^+) = -\LieA^+$. For all $g \in G$, denote
	\begin{align*}
		g^+ &= g[P] \in \Fboundary, & g^- &= gw_0[P] \in \Fboundary.
	\end{align*}
	Then $\Fboundary^{(2)} = G\cdot(e^+,e^-) \subset \Fboundary \times \Fboundary$ is the unique open $G$-orbit.
	
	\begin{remark}
		Since $P^+ = w_0Pw_0^{-1}$, given $(g^+, g^-) \in \Fboundary^{(2)}$ for some $g \in G$ and any other $(x, g^-), (g^+, y) \in \Fboundary^{(2)}$, there are unique elements $h \in N^+$ and $n \in N^-$ such that
		\begin{align}
			\label{eqn:UniqueElementsInN^+AndN^-}
			((gh)^+, (gh)^-) &= (x, g^-), & ((gn)^+, (gn)^-) &= (g^+, y).
		\end{align}
	\end{remark}
	
	\begin{definition}[Iwasawa cocycle]
		The Iwasawa decomposition is given by the $C^\infty$-diffeomorphism $K \times \LieA \times N^- \rightarrow G$ defined by $(k,v,n) \mapsto ka_vn$. The \emph{Iwasawa cocycle} $\sigma: G \times \Fboundary \rightarrow \LieA$ gives the unique element $\sigma(g,\xi)$ such that $gk \in Ka_{\sigma(g,\xi)}N^-$ and satisfies the cocycle relation $\sigma(gh,\xi) = \sigma(g,h\xi) + \sigma(h,\xi)$ for all $g, h \in G$ and $\xi = kM \in \Fboundary \cong K/M$.
	\end{definition}
	
	\begin{definition}[Busemann function]
		The \emph{Busemann function} $\beta: \Fboundary \times G/K \times G/K \to \LieA$ is defined by
		\begin{align*}
			\beta_\xi(x, y) = \sigma(g^{-1}, \xi) - \sigma(h^{-1}, \xi)
		\end{align*}
		for all $\xi \in \Fboundary$, $x = go \in G/K$, and $y = ho \in G/K$.
	\end{definition}
	
	We often write $\beta_\xi(g, h) = \beta_\xi(go, ho)$ for all $g, h \in G$ to ease notation. For all $\xi \in \Fboundary$ and $x, y, z \in G/K$ and $g \in G$, the Busemann function satisfies the properties
	\begin{enumerate}
		\item $\beta_\xi(e, g) = -\sigma(g^{-1}, \xi)$,
		\item $\beta_{g\xi}(gx, gy) = \beta_\xi(x, y)$,
		\item $\beta_\xi(x, z) = \beta_\xi(x, y) + \beta_\xi(y, z)$
	\end{enumerate}
	which are derived from the properties of the Iwasawa cocycle.
	
	Define a left $G$-action on $\Fboundary^{(2)} \times \LieA$ by
	\begin{align}
		\label{eqn:G-ActionOnHopfParametrization}
		g \cdot (x,y,v) = (gx,gy,v + \beta_x(g^{-1}, e)) = (gx,gy,v + \sigma(g,x))
	\end{align}
	for all $g \in G$ and $(x, y, v) \in \Fboundary^{(2)} \times \LieA$. Then $\Stab_G(e^+,e^-,0) = M$ and we obtain the following diffeomorphism.
	
	\begin{definition}[Hopf parametrization]
		The \emph{Hopf parametrization} is a left $G$-equivariant diffeomorphism $G/M \to \Fboundary^{(2)} \times \LieA$ defined by
		\begin{align*}
			gM \mapsto (g^+, g^-, \beta_{g^+}(e, g)) = (g^+, g^-, \sigma(g, e^+)).
		\end{align*}
	\end{definition}
	
	For all $v \in \interior(\LieA^+)$, the one-parameter diagonal flow on $G/M$ in the Hopf parametrization is then $\mathcal{W}_v: \Fboundary^{(2)} \times \LieA \times \mathbb R \to \Fboundary^{(2)} \times \LieA$ defined by $\mathcal{W}_{v, t}(x, y, w) = (x, y, w + tv)$ for all $(x, y, w) \in \Fboundary^{(2)} \times \LieA$ and $t \in \mathbb R$.
	
	\subsection{Limit set and limit cone}
	\label{subsec:LimitSetAndLimitCone}
	Let $\Gamma < G$ be a Zariski dense discrete subgroup henceforth.
	
	\begin{definition}[Limit set]
		The \emph{limit set} $\limitset \subset \Fboundary$ of $\Gamma$ is defined by
		\begin{align*}
			\limitset = \{\xi \in \Fboundary : \exists \{\gamma_n\}_{n \in \N} \subset \Gamma, (\gamma_n)_*\nu \to \delta_\xi\}
		\end{align*}
		where $\nu$ denotes the unique $K$-invariant probability measure on $\Fboundary$ and $\delta_\xi$ denotes the Dirac measure at $\xi$. It is the unique minimal nonempty closed $\Gamma$-invariant subset of $\Fboundary$ \cite{Ben97}.
	\end{definition}
	
	We say that $g \in G$ is \emph{hyperbolic} if $g$ is semisimple and conjugates into $A^+$. Any element $g\in G$ can be written uniquely as a commuting product $g = g_{\mathrm{h}}g_{\mathrm{e}}g_{\mathrm{u}}$, where $g_{\mathrm{h}}$ is hyperbolic, $g_{\mathrm{e}}$ is elliptic, and $g_{\mathrm{u}}$ is unipotent. We say that $g \in G$ is \emph{loxodromic} if $g_{\mathrm{h}}$ conjugates into $\interior(A^+)$, in which case $g_{\mathrm{e}}$ simultaneously conjugates into $M$ and $g_{\mathrm{u}} = e$. Let $\mu: G \to \LieA^+$ denote the \emph{Cartan projection}, i.e., $\mu(g) \in \LieA^+$ is the unique element such that $g \in Ka_{\mu(g)}K$ for all $g \in G$. Let $\lambda: G \to \LieA^+$ denote the \emph{Jordan projection}, i.e., $\lambda(g) \in \LieA^+$ is such that $g_{\mathrm{h}}$ is conjugate to $a_{\lambda(g)}$ for all $g \in G$. We record a basic but useful lemma which relates the Jordan projection with the Busemann function (see \cite[Lemma 3.5]{LO20b} for a proof).
	
	\begin{lemma}
		Let $g \in G$ be a loxodromic element whose attracting fixed point is $\xi \in \Fboundary$. Then $\beta_\xi(x, gx) = \lambda(g)$ for all $x \in G/K$.
	\end{lemma}
	
	\begin{definition}[Limit cone]
		The \emph{limit cone} $\limitcone \subset \LieA^+$ of $\Gamma$ is the unique minimal closed cone containing $\lambda(\Gamma)$. It is a convex subset with nonempty interior \cite{Ben97}.
	\end{definition}
	
	The \emph{growth indicator function} $\growthindicator : \LieA^+ \to \mathbb R \cup \{-\infty\}$ of $\Gamma$ is defined by
	\begin{align*}
		\growthindicator(v) = \|v\| \inf_{\text{open cones }\mathcal{C}\subset \LieA^+} \tau_\mathcal{C} \qquad \text{for all $v \in \LieA^+$}
	\end{align*}
	where $\tau_\mathcal{C}$ is the abscissa of convergence of $t \mapsto \sum_{\gamma \in \Gamma, \mu(\gamma) \in \mathcal{C}} e^{-t\|\mu(\gamma)\|}$. Identifying $\mathbb R \cong \mathbb R^*$, the growth indicator function is simply the critical exponent $\delta_\Gamma$ in the $\rank = 1$ case. Quint \cite{Qui02b} showed that $\growthindicator$ is homogeneous of degree $1$, concave, upper semicontinuous, and satisfies $\growthindicator|_{\exterior(\limitcone)} = -\infty$, $\growthindicator|_{\limitcone} \ge 0$, and $\growthindicator|_{\interior(\limitcone)} > 0$.
	
	\subsection{Patterson--Sullivan measures}
	We say that $\psi \in \LieA^*$ is \emph{tangent to $\growthindicator$ at $\mathsf{v} \in \limitcone$} if $\psi \geq \growthindicator$ and $\psi(\mathsf{v}) = \growthindicator(\mathsf{v})$. Define
	\begin{align*}
		D_\Gamma^\star = \{\psi \in \LieA^*: \psi \text{ is tangent to }\growthindicator \text{ at some } \mathsf{v} \in \interior(\LieA^+) \cap \limitcone\}.
	\end{align*}
	By the works of Quint \cite[Theorems 8.1 and 8.4]{Qui02a} building on \cite{Pat76,Sul79,Alb99}, we can make the following definition.
	
	\begin{definition}[Patterson--Sullivan measure]
		Let $\psi \in D_\Gamma^\star$. Any Borel probability measure $\nu_\psi$ on $\limitset$ is called a $(\Gamma,\psi)$-Patterson-Sullivan (PS) measure if
		\begin{equation*}
			\frac{d\gamma_*\nu_\psi}{d\nu_\psi}(\xi) = e^{\psi(\beta_\xi(e, \gamma))} \qquad \text{for all $\gamma \in \Gamma$ and $\xi \in \mathcal{F}$}.
		\end{equation*}
	\end{definition}
	
	\subsection{Bowen--Margulis--Sullivan measures}
	\label{subsec:BMSmeasures}
	We introduce the \emph{opposition involution} defined by $\involution = -\Ad_{w_0}: \LieA \to \LieA$. It is an involution with the properties
	\begin{enumerate}
		\item $\involution(\LieA^+)=\LieA^+$,
		\item $\involution(\lambda(g)) = \lambda(g^{-1})$ and $\involution(\mu(g)) = \mu(g^{-1})$ for all $g \in G$,
		\item $\growthindicator \circ \involution = \growthindicator$ and $\involution(\limitcone) = \limitcone$.
	\end{enumerate}
	An immediate consequence is that $\psi \circ \involution \in D_\Gamma^\star$ for all $\psi \in D_\Gamma^\star$.
	
	Choose a $(\Gamma, \psi)$-PS measure $\nu_\psi$ and a $(\Gamma, \psi \circ \involution)$-PS measure $\nu_{\psi \circ \involution}$. Via the Hopf parametrization, a \emph{Bowen--Margulis--Sullivan (BMS) measure} $m_{\nu_\psi,\nu_{\psi\circ\involution}}^{\mathrm{BMS}}$ is a locally finite Borel measure on $G/M$ defined by
	\begin{align*}
		dm_{\nu_\psi,\nu_{\psi\circ\involution}}^{\mathrm{BMS}}(gM) = e^{\psi\left(\beta_{g^+}(e,g)\right)} e^{(\psi \circ \involution)\left(\beta_{g^-}(e,g)\right)} \, d\nu_\psi(g^+) \, d\nu_{\psi \circ\involution}(g^-) \, dv
	\end{align*}
	where $dv$ denotes any choice of Lebesgue measure on $\LieA$. Then $m_{\nu_\psi,\nu_{\psi\circ\involution}}^{\mathrm{BMS}}$ is left $\Gamma$-invariant and right $A$-invariant \cite[Lemma 3.6]{ELO20}. We define induced measures all of which we call BMS measures and also denote by $m_{\nu_\psi,\nu_{\psi\circ\involution}}^{\mathrm{BMS}}$ by abuse of notation. We can lift $m_{\nu_\psi,\nu_{\psi\circ\involution}}^{\mathrm{BMS}}$ to a right $M$-invariant measure on $G$ by using the probability Haar measure on $M$. By left $\Gamma$-invariance, $m_{\nu_\psi,\nu_{\psi\circ\involution}}^{\mathrm{BMS}}$ descends to a measure on $\Gamma \backslash G$. By right $M$-invariance, $m_{\nu_\psi,\nu_{\psi\circ\involution}}^{\mathrm{BMS}}$ descends to a measure on $\Gamma \backslash G/M$. When $\psi$ is tangent to $\growthindicator$ at $\mathsf{v} \in \interior(\LieA^+)$, we use $\BMS$ to denote $m_{\nu_\psi,\nu_{\psi\circ\involution}}^{\mathrm{BMS}}$ for some choice of $\nu_\psi$ and $\nu_{\psi \circ \involution}$.
	
	\section{Anosov subgroups}
	\label{sec:AnosovSubgroups}
	The Anosov property (with respect to any parabolic subgroup of $G$) was first introduced by Labourie \cite{Lab06} for surface groups and later generalized by Guichard--Wienhard \cite{GW12} for Gromov hyperbolic groups (cf. \cite{KLP17,GGKW17,Wie18}). Although \cite[Definition 2.10]{GW12} is a general one, we prefer to make the following succinct definition which is more practical for our purposes. It is possible due to \cite[Corollary 4.16]{GW12} since $\Gamma < G$ is a Zariski dense discrete subgroup and $P < G$ is a minimal parabolic subgroup throughout the paper. Let $\Gboundary$ denote the Gromov boundary of $\Gamma$.
	
	\begin{definition}[Anosov subgroup]
		Let $\Gamma < G$ be a Zariski dense discrete subgroup. We say that $\Gamma$ is an \emph{Anosov subgroup} if it is a Gromov hyperbolic group and admits a continuous $\Gamma$-equivariant boundary map $\zeta: \Gboundary \to \Fboundary$ such that $(\zeta(x), \zeta(y)) \in \Fboundary^{(2)}$ for all $x, y \in \Gboundary$ with $x \neq y$.
	\end{definition}
	
	\begin{remark}
		The boundary map $\zeta$ is H\"{o}lder continuous by \cite[Proposition 3.2]{Lab06} and \cite[Theorem 6.1]{BCLS15}.
	\end{remark}
	
	Henceforth, we assume that $\Gamma$ is a Zariski dense Anosov subgroup and denote by $\zeta: \Gboundary \to \Fboundary$ its boundary map. Note that $\zeta$ is a homeomorphism onto $\zeta(\Gboundary) = \limitset$. The quotient $\Gamma \backslash G$ is called an \emph{Anosov homogeneous space}.
	
	\subsection{PS and BMS measures for Anosov subgroups}
	\label{subsec:PSandBMSmeasure}
	The following \cref{thm:PSDensityForAnosovSubgroups} was proved by Quint \cite[Proposition 3.2 and Theorem 4.7]{Qui03} when $\Gamma$ is a Schottky subgroup. In general, \cref{thm:PSDensityForAnosovSubgroups} follows from \cite[Corollaries 3.12, 3.13, and 4.9]{Sam14a} in light of \cite{BCLS15} using the Pl\"{u}cker representation (see also \cite[Propositions 4.6 and 4.11]{PS17}).
	
	\begin{theorem}
		\label{thm:PSDensityForAnosovSubgroups}
		We have
		\begin{enumerate}
			\item\label{itm:LimitConeInInteriorLieA+} $\limitcone \subset \interior(\LieA^+) \cup \{0\}$;
			\item\label{itm:NonTrivialElementsInGammaLoxodromic} $\gamma$ is loxodromic for all nontrivial $\gamma \in \Gamma$;
			\item\label{itm:ConcaveAnalyticOnLimitCone} $\growthindicator$ is an analytic strictly concave function on $\interior(\limitcone)$ and in particular, for all $v \in \mathbb \interior(\limitcone)$, there exists a unique $\psi_v \in D_\Gamma^\star$ tangent to $\growthindicator$ along $v$;
			\item\label{itm:TagentAtAInteriorVector} $D_\Gamma^\star = \{\psi \in \LieA^*: \psi \text{ is tangent to }\growthindicator \text{ at some } v \in \interior(\limitcone)\}$;
			\item\label{itm:PositiveOnLimitCone} $\psi|_{\limitcone\setminus\{0\}} > 0$ for all $\psi \in D_\Gamma^\star$.
		\end{enumerate}
	\end{theorem}
	
	\medskip
	
	For the rest of the paper, we fix $\mathsf{v} \in \interior(\limitcone)$ and the unique associated $\psi := \psi_\mathsf{v} \in D_\Gamma^\star$. The regular one-parameter diagonal flow $\mathcal{W}_{\mathsf{v}}$ is the primary object of study in this paper. Also fix a $(\Gamma, \psi)$-PS measure $\nu_\mathsf{v} := \nu_\psi$, a $(\Gamma, \psi \circ \involution)$-PS measure $\nu_\mathsf{\involution(\mathsf{v})} := \nu_{\psi \circ \involution}$, and the corresponding BMS measure $\BMS = m_{\nu_\psi,\nu_{\psi\circ\involution}}^{\mathrm{BMS}}$. Let $\Omega:=\supp\bigl(\BMS\bigr) \subset \Gamma \backslash G$. For convenience, we normalize $\mathsf{v}$ so that $\psi(\mathsf{v}) = 1$.
	
	\begin{remark}
		Although we do not use this, it was shown in \cite[Theorem 7.9]{ELO20} that a consequence of \cref{thm:LocalMixing} is the uniqueness of $\nu_\mathsf{v}$ for all $\mathsf{v} \in \interior(\limitcone)$.
	\end{remark}
	
	\subsection{\texorpdfstring{The vector bundle $\pi_\psi: \check{\Omega} \to \mathcal{X}$ and the translation flow}{The vector bundle and the translation flow}}
	\label{subsec:TheVectorBundleAndTheTranslationFlow}
	Define $\limitset^{(2)} = (\limitset \times \limitset) \cap \Fboundary^{(2)}$ and $\Gboundary^{(2)} = \{(x, y) \in \Gboundary \times \Gboundary: x \neq y\}$. Note that the Anosov property immediately implies $\limitset^{(2)} = \{(x, y) \in \limitset \times \limitset: x \neq y\}$. Observe that $\limitset^{(2)} \times \LieA$ is a vector bundle with typical fiber $\ker\psi$ via the projection map $\pi_\psi: \limitset^{(2)} \times \LieA \to \limitset^{(2)} \times \mathbb R$ defined by
	\begin{align*}
		\pi_\psi(x, y, v) = (x, y, \psi(v)) \qquad \text{for all $(x, y, v) \in \limitset^{(2)} \times \LieA$}.
	\end{align*}
	We define a left $\Gamma$-action on $\limitset^{(2)} \times \mathbb R$ by
	\begin{align}
		\label{eqn:Gamma-ActionOnLimitSet^2xR}
		\gamma \cdot (x, y, t) &= (\gamma x, \gamma y, t + \psi(\beta_x(\gamma^{-1}, e))) \qquad \text{for all $\gamma \in \Gamma$, $(x, y, t) \in \limitset^{(2)} \times \mathbb R$}
	\end{align}
	so that $\pi_\psi$ is $\Gamma$-equivariant. We can then define
	\begin{align*}
		\check{\Omega} &= \Gamma \backslash \bigl(\limitset^{(2)} \times \LieA\bigr), & \mathcal{X} &= \Gamma \backslash \bigl(\limitset^{(2)} \times \mathbb R\bigr).
	\end{align*}
	
	The following is an immediate consequence of \cref{itm:TranslationFlowConjugateToGromovGeodesicFlow} in \cref{thm:TranslationFlowConjugateToGromovGeodesicFlow} whose statement and proof is deferred to \cref{sec:CodingTheTranslationFlow}.
	
	\begin{theorem}
		The $\Gamma$-action on $\limitset^{(2)} \times \mathbb R$ is properly discontinuous and cocompact. Consequently, $\mathcal{X}$ is a compact Hausdorff topological space.
	\end{theorem}
	
	By $\Gamma$-equivariance, $\pi_\psi$ descends to a projection map $\pi_\psi: \check{\Omega} \to \mathcal{X}$ with which $\check{\Omega}$ is a priori an \emph{affine} bundle with typical fiber $\ker\psi$. In fact, it is a trivial vector bundle $\check{\Omega} \cong \mathcal{X} \times \ker\psi$ due to the following argument. It is well-known that global sections can be constructed on affine bundles using a partition of unity. The fiberwise translation action of $\ker\psi$ on $\limitset^{(2)} \times \LieA$ commutes with the left $\Gamma$-action on $\limitset^{(2)} \times \LieA$ and so it descends to a $\ker\psi$-action on $\check{\Omega}$. Hence, $\check{\Omega}$ is a principal $\ker\psi$-bundle with a global section which implies triviality.
	
	Using the Hopf parametrization, $\check{\Omega} \subset \Gamma \backslash (\Fboundary^{(2)} \times \LieA)$ coincides with $\supp\bigl(\BMS\bigr) \subset \Gamma \backslash G/M$ which is invariant under $\mathcal{W}_{\mathsf{v}}$. Moreover, the left $G$-invariant metric $d$ on $G/M \cong \Fboundary^{(2)} \times \LieA$ restricts to a left $\Gamma$-invariant metric $d$ on $\limitset^{(2)} \times \LieA$ and descends to a metric $d$ on $\check{\Omega}$. We endow $\mathcal{X}$ with a metric $d$ using an embedding $\mathcal{X} \cong \mathcal{X} \times \{0\} \subset \mathcal{X} \times \ker\psi \cong \check{\Omega}$ which lifts to a left $\Gamma$-invariant metric $d$ on $\limitset^{(2)} \times \mathbb R$.
	
	We also define the locally finite Borel measure $m_\mathsf{v}$ on $\limitset^{(2)} \times \mathbb R$ by
	\begin{align*}
		dm_\mathsf{v}(\xi, \eta, t) = e^{\psi([\xi, \eta])} \, d\nu_\mathsf{v}(\xi) \, d\nu_{\involution(\mathsf{v})}(\eta) \, dt
	\end{align*}
	where $[\xi, \eta] = \beta_{g^+}(e, g) + \involution(\beta_{g^-}(e, g))$ and $g \in G$ is any element with $g^+ = \xi$ and $g^- = \eta$ (see \cite[Definition 3.8]{LO20b}) and $dt$ denotes the Lebesgue measure on $\mathbb R$. By left $\Gamma$-invariance, it descends to a measure $m_\mathsf{v}$ on $\mathcal{X}$. For convenience, we normalize all the BMS measures and $m_\mathsf{v}$ so that $m_\mathsf{v}(\mathcal{X}) = 1$ and
	\begin{align}
		\label{eqn:ProductOfBMSAndLebesgue}
		d\BMS\bigr|_{\check{\Omega}} = dm_\mathsf{v} \, dv
	\end{align}
	(see the proof of \cite[Corollary 4.9]{LO20b} and \cite[Proposition 3.5]{Sam15}) where $dv$ denotes the Lebesgue measure on $\ker\psi$ which is compatible with $\langle \cdot, \cdot \rangle_\psi$ (see \cref{sec:FirstReturnVectorMapAndHolonomy} for the definition).
	
	The one-parameter diagonal flow $\mathcal{W}_{\mathsf{v}}$ projects down via $\pi_\psi$ to the \emph{translation flow} $\mathcal{W}: \bigl(\limitset^{(2)} \times \mathbb R\bigr) \times \mathbb R \to \limitset^{(2)} \times \mathbb R$ defined by $\mathcal{W}_t(x, y, s) = (x, y, s + t)$ for all $(x, y, s) \in \limitset^{(2)} \times \mathbb R$ and $t \in \mathbb R$. We use the same terminology for the induced flow $\mathcal{W}: \mathcal{X} \times \mathbb R \to \mathcal{X}$.
	
	\subsection{Ergodic decomposition of the BMS measure}
	\label{subsec:ErgodicDecomposition}
	We discuss the ergodic properties of BMS measure in the Anosov setting. For fundamental groups of compact negatively curved manifolds, \cite[Theorem 3.2]{Sam14b} implies that the BMS measure on $\Gamma \backslash G$ is $AM$-ergodic \cite[Corollary 4.9]{LO20b}. For general Zariski dense Anosov subgroups, \cite[Corollary 4.9]{LO20b} also holds using \cref{cor:TranslationFlowHasMarkovSection} and \cite[Appendix A.2]{Car21}. Unlike when $\rank = 1$ \cite{Sul84,Che02,CI99,Win15}, when $\rank \geq 2$, the BMS measure on $\Gamma \backslash G$ is not necessarily $A$-ergodic. Recently, Lee--Oh \cite[Theorems 1.1(2) and 4.4(2)]{LO20a} characterized the $A$-ergodic decomposition of the BMS measure which we recall in \cref{thm:ErgodicDecomposition}.
	
	Let $M^\circ$ and $P^\circ$ be the identity components of $M$ and $P$ respectively. Note that $P^\circ = M^\circ AN^-$. Let $\Fboundary^\circ = G/P^\circ$. Each $g \in G$ such that $g^+ \in \limitset$ generates a $\Gamma$-minimal subset $\overline{\Gamma gP^\circ} \subset \Fboundary^\circ$. Let $\mathcal{Y}_\Gamma$ denote the set of $\Gamma$-minimal subsets of $\Fboundary^\circ$ and fix any $\Lambda_0 \in \mathcal{Y}_\Gamma$. Define
	\begin{align*}
		M_\Gamma = \{m \in M : \Lambda_0m = \Lambda_0\} < M.
	\end{align*}
	Then, $M_\Gamma$ is a normal closed subgroup of $M$ containing both $M^\circ$ and $[M,M]$ and is independent of the choice of $\Lambda_0$ \cite{GR07}.
	Let $P_\Gamma = M_\Gamma AN^-$, $\tilde{\mathcal{E}}_{[m]} \subset G$ be the left $\Gamma$-invariant and right $P_\Gamma$-invariant closed subset such that $\tilde{\mathcal{E}}_{[m]}/P^\circ = \Lambda_0m$, and $\Omega_{[m]} = \Gamma \backslash \tilde{\mathcal{E}}_{[m]} \cap \Omega$ for all $[m] \in M_\Gamma \backslash M$. Denote $\Omega_0 = \Omega_{[e]}$. Note that $\Omega_{[m]} = \Omega_0m$ and
	\begin{align}
		\label{eqn:Omega_mProjectsToE}
		\pi\bigl(\Omega_{[m]}\bigr) = \check{\Omega} \qquad \text{for all $[m] \in M_\Gamma \backslash M$}
	\end{align}
	where $\pi: \Gamma \backslash G \to \Gamma \backslash G/M$ is the quotient map.
	
	\begin{theorem}[{\citealp[Theorem 1.1(2)]{LO20a}}]
		\label{thm:ErgodicDecomposition}
		We have the $A$-ergodic decomposition $\BMS = \sum_{[m] \in M_\Gamma \backslash M} \BMS\bigr|_{\Omega_{[m]}}$.
	\end{theorem}
	
	For all $m \in M$, we have $\BMS\bigr|_{\Omega_{[m]}} = m_*\BMS\bigr|_{\Omega_0}$ where $m_*$ denotes the pushforward by $m$ using the right translation action of $M$. Hence, it suffices to prove \cref{thm:LocalMixing} for $\BMS\bigr|_{\Omega_0}$. We denote $\mathsf{m} = \BMS\bigr|_{\Omega_0}$ henceforth for brevity.
	
	We recall some facts from \cite[Section 2]{LO20a} (cf. \cite[Section 3]{Dan20}) for details. The Hopf parametrization of $G/M$ can be lifted to a $G$-equivariant Borel isomorphism $G \to \Fboundary^{(2)} \times A \times M$. Restricting to $\supp\bigl(\BMS\bigr) \subset G$ and using $\psi$, we get a $\Gamma$-equivariant Borel map $\Psi: \supp\bigl(\BMS\bigr) \to \limitset^{(2)} \times \R \times M$. The $\Gamma$-action on $\limitset^{(2)} \times \R \times M$ is such that the projection map $\limitset^{(2)} \times \R \times M \to \limitset^{(2)} \times \R$ is $\Gamma$-equivariant. Moreover, we have the identity
	\begin{align}
		\label{eqn:PsiConjugation}
		\Psi(\gamma ga_vm) = \mathcal{W}_{\psi(v)}(\gamma \Psi(g)m) \qquad \text{for all $\gamma \in \Gamma$, $v \in \LieA$, $m \in M$, $g \in \supp\bigl(\BMS\bigr)$}
	\end{align}
	whenever $g$ and $\gamma g$ are contained in the open dense subset $N^+P \subset G$, where we use the right translation action of $M$ and by abuse of notation $\mathcal{W}$ denotes the translation action of $\mathbb R$.
	
	On $\limitset^{(2)} \times \R \times M$, we put the $\Gamma$-invariant measure
	\begin{equation*}
		d\hat{m}_\mathsf{v} = e^{\psi\left(\beta_{g^+}(e,g)\right)} e^{(\psi \circ \involution)\left(\beta_{g^-}(e,g)\right)} \, d\nu_\mathsf{v}(g^+) \, d\nu_{\involution(\mathsf{v})}(g^-) \, dt \, dm,
	\end{equation*}
	which descends to a finite measure $\hat{m}_\mathsf{v}$ on $\mathcal{Z} := \Gamma \backslash \bigl(\limitset^{(2)} \times \R \times M\bigr)$. By left $\Gamma$-invariance, $\Psi$ descends to a Borel map also denoted by $\Psi: \Omega \to \mathcal{Z}$. Setting $\mathcal{Z}_{[m]} := \Psi(\Omega_{[m]})$ gives a measurable partition $\{\mathcal{Z}_{[m]} : [m] \in M_\Gamma \backslash M\}$ for $\mathcal{Z}$. We also use the same notation for the $\R$- and $M$-actions induced on $\mathcal{Z}$. 
	
	\begin{proposition}[{\citealp[Proposition 4.8]{LO20a}}]
		\label{pro:R-ergodicity}
		We have the $\R$-ergodic decomposition $\hat{m}_\mathsf{v} = \sum_{[m] \in M_\Gamma \backslash M}\hat{m}_\mathsf{v}|_{\mathcal{Z}_{[m]}}$. Consequently, there exists a dense $\R_{> 0}$-orbit in $\mathcal{Z}_{[e]}$.
	\end{proposition}
	
	\section{Coding the translation flow}
	\label{sec:CodingTheTranslationFlow}
	In this section, we will review metric Anosov flows and construct an appropriate Markov section on $\mathcal{X}$ which will provide the symbolic coding used throughout this paper. We will also review symbolic dynamics and thermodynamic formalism.
	
	\subsection{Metric Anosov flows and Markov sections}
	\label{subsec:MetricAnosovFlowsAndMarkovSections}
	We first recall metric Anosov flows from \cite[Subsection 3.2]{BCLS15} which is based on Smale flows from \cite{Pol87}. We will use similar definitions and notations. Let $\phi: \mathcal{Y} \times \mathbb R \to \mathcal{Y}$ be an arbitrary (continuous) flow on an arbitrary metric space $(\mathcal{Y}, d)$. Let $W$ be an equivalence relation on $\mathcal{Y}$. The \emph{leaf} through a point $x \in \mathcal{Y}$ is the equivalence class of $x$ which we denote by $W(x)$. This provides a partition $\mathcal{Y} = \bigsqcup_{x \in Y} W(x)$ where $Y$ is some set of representatives for each equivalence class. For all $x \in \mathcal{Y}$ and $\epsilon > 0$, we also denote $W_\epsilon(x) = \{y \in W(x): d(x, y) < \epsilon\}$.
	
	\begin{definition}[Foliation]
		An equivalence relation $W$ is called a \emph{foliation} of $\mathcal{Y}$ if for all $x \in \mathcal{Y}$, there exists a homeomorphism $\varphi_x = (\varphi_x^1, \varphi_x^2): O_x \to V_x^1 \times V_x^2$ called a \emph{chart} from an open neighborhood $O_x$ of $x$ to the product of two topological spaces $V_x^1$ and $V_x^2$ such that
		\begin{enumerate}
			\item for all $w, z \in O_x \cap O_y$, we have $\varphi_x^2(w) = \varphi_x^2(z)$ if and only if $\varphi_y^2(w) = \varphi_y^2(z)$;
			\item $z \in W(w)$ if and only if there exists a sequence $\{w_j\}_{j = 1}^n$ for some $n \in \mathbb N$ with $w_1 = w$ and $w_n = z$ such that for all $j \in \{1, 2, \dotsc, n - 1\}$, we have $w_{j + 1} \in O_{w_j}$ and $\varphi_{w_j}^2(w_{j + 1}) = \varphi_{w_j}^2(w_j)$.
		\end{enumerate}
	\end{definition}
	
	\begin{remark}
		There is a natural foliation $W^0$ provided by the orbits of the flow $\phi$, i.e., for all $x, y \in \mathcal{Y}$, we have $y \in W^0(x)$ if and only if $y = \phi_t(x)$ for some $t \in \mathbb R$.
	\end{remark}
	
	Let $W$ be a foliation. Let $x \in \mathcal{Y}$. Both subsets of the form $V_x^1 \times \{v\} \subset V_x^1 \times V_x^2$ and their pullbacks $\varphi_x^{-1}(V_x^1 \times \{v\}) \subset O_x$ for some $v \in V_x^2$ are called \emph{plaques} of the chart $\varphi_x$. Both the open subsets $V \times \{v\} \subset V_x^1 \times V_x^2$ of the plaques and their pullbacks $\varphi_x^{-1}(V \times \{v\}) \subset O_x$ for some open set $V \subset V_x^1$ and $v \in V_x^2$ are called \emph{plaque open sets}. The \emph{plaque topology} of $W(x)$ is the topology generated by the plaque open sets in $W(x)$.
	
	\begin{definition}[Local product structure]
		\label{def:LocalProductStructure}
		A pair of foliations $(W, W')$ has a \emph{local product structure} if there exists $\epsilon > 0$ such that for all $x \in \mathcal{Y}$, there exists a homeomorphism $[\cdot, \cdot]: W_\epsilon(x) \times W'_\epsilon(x) \to O_x$ where $O_x \subset \mathcal{Y}$ is a neighborhood of $x$ such that $[\cdot, \cdot]^{-1}$ is a chart for both $W$ and $W'$.
	\end{definition}
	
	\begin{remark}
		From the above definition, it follows that for all $x \in \mathcal{Y}$, there is a unique local intersection $[y, W'_\epsilon(x)] \cap [W_\epsilon(x), z] = \{[y, z]\}$ of plaque open sets for all $y \in W_\epsilon(x)$ and $z \in W'_\epsilon(x)$.
	\end{remark}
	
	A foliation $W$ is \emph{equivariant} under $\phi$ if $\phi_t(W(x)) = W(\phi_t(x))$ for all $x \in \mathcal{Y}$ and $t \in \mathbb R$. We say that $W$ is \emph{transverse} to $\phi$ if it is equivariant under $\phi$ and there exists $\epsilon > 0$ such that for all $x \in \mathcal{Y}$, there exists a chart $\varphi_x: O_x \to V_x^1 \times V_x^2 = V_x^1 \times (-\epsilon, \epsilon) \times \tilde{V}_x^2$ for $W$ such that $\phi_t((\varphi_x)^{-1}(y, s, v)) = (\varphi_x)^{-1}(y, s + t, v)$ for all $(y, s, v) \in V_x^1 \times (-\epsilon, \epsilon) \times \tilde{V}_x^2$ and $t \in \mathbb R$ with $s + t \in (-\epsilon, \epsilon)$. Note that in this case there exists $\epsilon > 0$ such that for all $x \in \mathcal{Y}$, we have $W_\epsilon(x) \cap \bigcup_{t \in (-\epsilon, \epsilon)} \phi_t(x) = \{x\}$. If $\mathcal{Y}$ is a topological manifold, it can be deduced from definitions that the second condition is equivalent to the given definition of transversality.
	
	Given a foliation $W$ transverse to $\phi$, we define the corresponding \emph{central foliation} $W^{\mathrm{c}}$ such that for all $x, y \in \mathcal{Y}$, we have $y \in W^{\mathrm{c}}(x)$ if and only if $\phi_t(y) \in W(x)$ for some $t \in \mathbb R$.
	
	\begin{definition}[Metric Anosov flow]
		\label{def:MetricAnosovFlow}
		The flow $\phi: \mathcal{Y} \times \mathbb R \to \mathcal{Y}$ is called a \emph{metric Anosov flow (with respect to $(W^{\mathrm{su}}, W^{\mathrm{ss}})$)} if there exist $\epsilon > 0$, $C > 0$, $\eta > 0$, and a pair of foliations $(W^{\mathrm{su}}, W^{\mathrm{ss}})$ transverse to $\phi$ such that denoting $W^{\mathrm{wu}} = (W^{\mathrm{su}})^{\mathrm{c}}$ and $W^{\mathrm{ws}} = (W^{\mathrm{ss}})^{\mathrm{c}}$, we have
		\begin{enumerate}
			\item $(W^{\mathrm{wu}}, W^{\mathrm{ss}})$ and $(W^{\mathrm{ws}}, W^{\mathrm{su}})$ have local product structures;
			\item for all $x \in \mathcal{Y}$, $y \in W_\epsilon^{\mathrm{su}}(x)$, and $z \in W_\epsilon^{\mathrm{ss}}(x)$, we have
			\begin{align*}
				d(\phi_{-t}(x), \phi_{-t}(y)) &\leq Ce^{-\eta t}d(x, y), & d(\phi_t(x), \phi_t(z)) &\leq Ce^{-\eta t}d(x, z)
			\end{align*}
			for all $t \geq 0$, called the \emph{Anosov property}.
		\end{enumerate}
	\end{definition}
	
	Let $(W^{\mathrm{su}}, W^{\mathrm{ss}})$ be a pair of foliations transverse to $\phi$. Denote $W^{\mathrm{wu}} = (W^{\mathrm{su}})^{\mathrm{c}}$ and $W^{\mathrm{ws}} = (W^{\mathrm{ss}})^{\mathrm{c}}$ and suppose $(W^{\mathrm{wu}}, W^{\mathrm{ss}})$ and $(W^{\mathrm{ws}}, W^{\mathrm{su}})$ have local product structures with some constant $\epsilon_0 > 0$ in \cref{def:LocalProductStructure}. Denote by $[\cdot, \cdot]$ only the map provided by the local product structure of $(W^{\mathrm{wu}}, W^{\mathrm{ss}})$. Subsets $U \subset W_{\epsilon_0}^{\mathrm{su}}(x)$ and $S \subset W_{\epsilon_0}^{\mathrm{ss}}(x)$ for some $x \in \mathcal{Y}$ are called \emph{proper} if $U = \overline{\interior(U)}$ and $S = \overline{\interior(S)}$, where the interiors and closures are taken in the respective plaque topologies. For any proper subsets $U \subset W_{\epsilon_0}^{\mathrm{su}}(x)$ and $S \subset W_{\epsilon_0}^{\mathrm{ss}}(x)$ containing some $x \in \check{\Omega}$, we call
	\begin{align*}
		R = [U, S] = \{[u, s] \in \mathcal{Y}: u \in U, s \in S\} \subset \mathcal{Y}
	\end{align*}
	a \emph{rectangle of size $\hat{\delta}$ (with respect to $(W^{\mathrm{su}}, W^{\mathrm{ss}})$)} if $\diam(R) \leq \hat{\delta}$ for some $\hat{\delta} > 0$, and $x$ the \emph{center} of $R$. For any rectangle $R = [U, S]$, we can extend the map $[\cdot, \cdot]$ to $[\cdot, \cdot]: R \to R$ defined by $[v_1, v_2] = [u_1, s_2]$ for all $v_1 = [u_1, s_1] \in [U, S]$ and $v_2 = [u_2, s_2] \in [U, S]$.
	
	\begin{remark}
		The map $[\cdot, \cdot]$ then satisfies the properties
		\begin{enumerate}
			\item $[x, x] = x$,
			\item $[[x, y], z] = [x, z]$,
			\item $[x, [y, z]] = [x, z]$
		\end{enumerate}
		for all $x, y, z \in R$ for any rectangle $R$ as in \cite{Pol87}.
	\end{remark}
	
	\begin{definition}[Complete set of rectangles]
		\label{def:CompleteSetOfRectangles}
		A set $\mathcal{R} = \{R_1, R_2, \dotsc, R_N\} = \{[U_1, S_1], [U_2, S_2], \dotsc, [U_N, S_N]\}$ for some $N \in \mathbb N$ consisting of rectangles of size $\hat{\delta}$ with respect to $(W^{\mathrm{su}}, W^{\mathrm{ss}})$ in $\mathcal{Y}$ is called a \emph{complete set of rectangles of size $\hat{\delta}$ (with respect to $(W^{\mathrm{su}}, W^{\mathrm{ss}})$)} if:
		\begin{enumerate}
			\item \label{itm:MarkovProperty1} $R_j \cap R_k = \varnothing$ for all $1 \leq j, k \leq N$ with $j \neq k$;
			\item \label{itm:MarkovProperty2} $\mathcal{Y} = \bigcup_{j = 1}^N \bigcup_{t \in [0, \hat{\delta}]} \phi_t(R_j)$.
		\end{enumerate}
	\end{definition}
	
	We fix $\mathcal{R} = \{R_1, R_2, \dotsc, R_N\} = \{[U_1, S_1], [U_2, S_2], \dotsc, [U_N, S_N]\}$ to be a complete set of rectangles of some size $\hat{\delta} \in (0, \epsilon_0)$ in $\mathcal{Y}$. We define
	\begin{align*}
		R &= \bigsqcup_{j = 1}^N R_j, & U &= \bigsqcup_{j = 1}^N U_j.
	\end{align*}
	Define the first return time map $\tau: R \to \mathbb R$ by
	\begin{align*}
		\tau(u) = \inf\{t \in \mathbb R_{>0}: \phi_t(u) \in R\} \qquad \text{for all $u \in R$}.
	\end{align*}
	We denote the constants $\overline{\tau} = \sup_{u \in R} \tau(u)$ and $\underline{\tau} = \inf_{u \in R} \tau(u)$ for convenience. Define the Poincar\'{e} first return map $\mathcal{P}: R \to R$ by
	\begin{align*}
		\mathcal{P}(u) = \phi_{\tau(u)}(u) \qquad \text{for all $u \in R$}.
	\end{align*}
	Let $\sigma = (\proj_U \circ \mathcal{P})|_U: U \to U$ be its projection where $\proj_U: R \to U$ is the projection defined by $\proj_U([u, s]) = u$ for all $[u, s] \in R$. We define the \emph{cores}
	\begin{align*}
		\hat{R} &= \{u \in R: \mathcal{P}^k(u) \in \interior(R) \text{ for all } k \in \mathbb Z\}, \\
		\hat{U} &= \{u \in U: \sigma^k(u) \in \interior(U) \text{ for all } k \in \mathbb Z_{\geq 0}\}.
	\end{align*}
	We note that the cores are both residual subsets (complements of meager subsets) of $R$ and $U$ respectively.
	
	\begin{definition}[Markov section]
		\label{def:MarkovSection}
		Let $\hat{\delta} > 0$ and $N \in \mathbb N$. We call a complete set of rectangles $\mathcal{R}$ of size $\hat{\delta}$ with respect to $(W^{\mathrm{su}}, W^{\mathrm{ss}})$ in $\mathcal{Y}$ a \emph{Markov section (with respect to $(W^{\mathrm{su}}, W^{\mathrm{ss}})$)} if in addition to \cref{itm:MarkovProperty1,itm:MarkovProperty2} in \cref{def:CompleteSetOfRectangles}, the following property
		\begin{enumerate}
			\setcounter{enumi}{2}
			\item $[\interior(U_k), \mathcal{P}(u)] \subset \mathcal{P}([\interior(U_j), u])$ and $\mathcal{P}([u, \interior(S_j)]) \subset [\mathcal{P}(u), \interior(S_k)]$ for all $u \in R$ such that $u \in \interior(R_j) \cap \mathcal{P}^{-1}(\interior(R_k)) \neq \varnothing$, for all $1 \leq j, k \leq N$
		\end{enumerate}
		called the \emph{Markov property}, is satisfied. This can be understood pictorially in \cref{fig:MarkovProperty}.
	\end{definition}
	
	Observe that for a Markov section $\mathcal{R}$ in $\mathcal{Y}$, the map $\tau$ is constant on $[u, S_j]$ for all $u \in U_j$ and $1 \leq j \leq N$.
	
	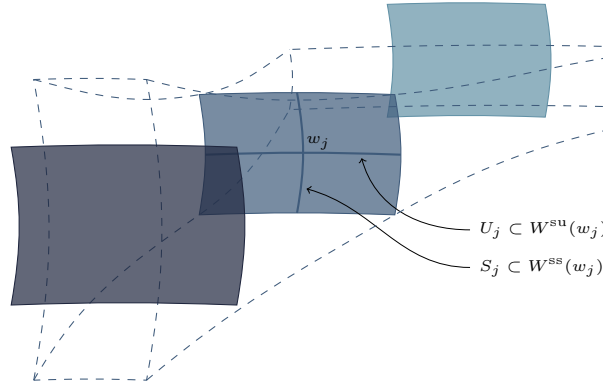
\begin{figure}
		\definecolor{front}{RGB}{31, 38, 62}
		\definecolor{middle}{RGB}{63,91,123}
		\definecolor{back}{RGB}{98,145,166}
		\begin{tikzpicture}
			\coordinate (E) at (2.7,1.5);
			\coordinate (F) at (2.7,3);
			\coordinate (G) at (4.8,3);
			\coordinate (H) at (4.8,1.5);
			
			\draw[back, fill = back, fill opacity=0.7] (E) to[out=80,in=-80] (F) to[out=2,in=178] (G) to[out=-80,in=80] (H) to[out=178,in=2] (E) -- cycle;
			
			\coordinate (A) at (0.2,0.2);
			\coordinate (B) at (0.2,1.8);
			\coordinate (C) at (2.8,1.8);
			\coordinate (D) at (2.8,0.2);
			
			\draw[middle, fill = middle, fill opacity=0.7] (A) to[out=80,in=-80] (B) to[out=2,in=178] (C) to[out=-80,in=80] (D) to[out=178,in=2] (A) -- cycle;
			
			\coordinate (AB_Mid) at (0.28,1);
			\coordinate (CD_Mid) at (2.88,1);
			
			\coordinate (Uj_Label) at (3.8, 0);
			\coordinate (Sj_Label) at (3.8, -0.5);
			
			\node[above right, font=\tiny] at (1.5, 0.95) {$w_j$};
			
			\draw[middle,thick] (AB_Mid) to[out=2,in=178] (CD_Mid);
			\node[right, font=\tiny] at (Uj_Label) {$U_j \subset W^{\mathrm{su}}(w_j)$};
			\draw[->, very thin] (Uj_Label) to[out=180,in=-75] ($0.8*(CD_Mid) + 0.2*(AB_Mid) + (0, -0.05)$);
			
			\coordinate (BC_Mid) at (1.5,1.825);
			\coordinate (AD_Mid) at (1.5,0.225);
			
			\draw[middle,thick] (AD_Mid) to[out=80,in=-80] (BC_Mid);
			\node[right, font=\tiny] at (Sj_Label) {$S_j \subset W^{\mathrm{ss}}(w_j)$};
			\draw[->, very thin] (Sj_Label) to[out=180,in=-15] ($0.8*(AD_Mid) + 0.2*(BC_Mid) + (0.13, 0)$);
			
			\coordinate (A') at (1.4,1.6);
			\coordinate (B') at (1.4,2.4);
			\coordinate (C') at (7,2.4);
			\coordinate (D') at (7,1.6);
			
			\draw[middle,dashed] (A') to[out=80,in=-80] (B') to[out=2,in=178] (C') to[out=-80,in=80] (D') to[out=178,in=2] (A') -- cycle;
			
			\coordinate (A'') at (-2,-2);
			\coordinate (B'') at (-2,2);
			\coordinate (C'') at (-0.5,2);
			\coordinate (D'') at (-0.5,-2);
			
			\draw[middle,dashed] (A'') to[out=80,in=-80] (B'') to[out=2,in=178] (C'') to[out=-80,in=80] (D'') to[out=178,in=2] (A'') -- cycle;
			
			\draw[middle,dashed] (A'') to[out=60,in=215] (A) to[out=35,in=240] (A');
			\draw[middle,dashed] (B'') to[out=-15,in=190] (B) to[out=10,in=215] (B');
			\draw[middle,dashed] (C'') to[out=-15,in=185] (C) to[out=5,in=180] (C');
			\draw[middle,dashed] (D'') to[out=40,in=210] (D) to[out=30,in=190] (D');
			
			\coordinate (I) at (-2.3,-1);
			\coordinate (J) at (-2.3,1.1);
			\coordinate (K) at (0.7,1.1);
			\coordinate (L) at (0.7,-1);
			
			\draw[front, fill = front, fill opacity=0.7] (I) to[out=80,in=-80] (J) to[out=2,in=178] (K) to[out=-80,in=80] (L) to[out=178,in=2] (I) -- cycle;
		\end{tikzpicture}
		\caption{The Markov property.}
		\label{fig:MarkovProperty}
	\end{figure}
	
	We now recall a theorem of Pollicott \cite{Pol87} generalizing the works of Bowen and Ratner \cite{Bow70,Rat73} on the existence of Markov sections.
	
	\begin{theorem}
		\label{thm:MarkovSectionForMetricAnosovFlow}
		For all metric Anosov flows with respect to $(W^{\mathrm{su}}, W^{\mathrm{ss}})$ on a compact metric space, there exists a Markov section with respect to $(W^{\mathrm{su}}, W^{\mathrm{ss}})$.
	\end{theorem}
	
	\subsection{Markov section for the translation flow}
	\label{subsec:MarkovSectionForTheTranslationFlow}
	In this subsection, we return to the translation flow and show existence of a Markov section.
	
	We need to first introduce reparametrizations of flows for our discussion.
	
	\begin{definition}[Reparametrization]
		Let $\mathcal{Y}$ be a metric space and $\phi: \mathcal{Y} \times \mathbb R \to \mathcal{Y}$ be a flow. We say that a flow $\psi: \mathcal{Y} \times \mathbb R \to \mathcal{Y}$ is a \emph{reparametrization} of $\phi$ if it is of the form $\psi_t(x) = \phi_{\alpha(x, t)}(x)$ for all $x \in \mathcal{Y}$ and $t \in \mathbb R$, where $\alpha: \mathcal{Y} \times \mathbb R \to \mathbb R$ is a continuous map, which we also call a \emph{reparametrization}, satisfying
		\begin{enumerate}
			\item $\alpha(x, t) > 0$ for all $x \in \mathcal{Y}$ and $t > 0$,
			\item $\alpha(x, s + t) = \alpha(\psi_s(x), t) + \alpha(x, s)$ for all $x \in \mathcal{Y}$ and $s, t \in \mathbb R$ called the \emph{cocycle condition}.
		\end{enumerate}
		Moreover, $\phi$ is itself a reparametrization of $\psi$ for some continuous map $\alpha^*: \mathcal{Y} \times \mathbb R \to \mathbb R$ which we call the \emph{inverse} of $\alpha$. It satisfies
		\begin{enumerate}
			\item $(\alpha^*)^* = \alpha$,
			\item $\alpha(x, \alpha^*(x, t)) = \alpha^*(x, \alpha(x, t)) = t$ for all $(x, t) \in \mathcal{Y} \times \mathbb R$.
		\end{enumerate}
		We say that a reparametrization is Lipschitz (resp. H\"{o}lder) if $\alpha(\cdot, t)$ is Lipschitz (resp. H\"{o}lder) continuous for all $t \in \mathbb R$.
	\end{definition}
	
	\begin{remark}
		\label{rem:ReparametrizationRegularity}
		Suppose $\alpha(x, \cdot)$ is differentiable for all $x \in \mathcal{Y}$ and the partial derivative in the second coordinate $\partial_2\alpha$ is positive continuous. Then clearly $\alpha(x, t) = \int_0^t f(\phi_s(x)) \, ds$ for all $(x, t) \in \mathcal{Y} \times \mathbb R$ for the positive continuous function $f = \partial_2\alpha(\cdot, 0)$. Moreover, $\alpha^*$ has the same regularity and hence $\alpha^*(x, t) = \int_0^t f^*(\phi_s(x)) \, ds$ for all $(x, t) \in \mathcal{Y} \times \mathbb R$ for the positive continuous function $f^* = \partial_2\alpha^*(\cdot, 0)$. If we further assume that the reparametrization is Lipschitz (resp. H\"{o}lder), then its inverse reparametrization is also Lipschitz (resp. H\"{o}lder), and $f$ and $f^*$ are also Lipschitz (resp. H\"{o}lder) continuous.
	\end{remark}
	
	We now recount the Gromov geodesic flow associated to $\Gamma$ and its essential properties in \cref{thm:GromovGeodesicFlow}. It was introduced by Gromov \cite[Subsection 8.3]{Gro87} and further developed by Champetier \cite{Cha94} and Mineyev \cite{Min05}. We refer to \cite[Theorem 5.1]{CP02} for a more detailed theorem.
	
	\begin{theorem}
		\label{thm:GromovGeodesicFlow}
		There exists a proper geodesic Gromov hyperbolic metric space $(\tilde{\mathcal{G}}, d)$ associated to the Gromov hyperbolic group $\Gamma$ with a left $(\Gamma \times \mathbb Z/2\mathbb Z)$-action by isometries and a flow $\phi: \tilde{\mathcal{G}} \times \mathbb R \to \tilde{\mathcal{G}}$ called the \emph{Gromov geodesic flow} with the properties:
		\begin{enumerate}
			\item $\phi_t \circ \gamma = \gamma \circ \phi_t$ and $\phi_t \circ \iota = \iota \circ \phi_{-t}$ for all $t \in \mathbb R$, $\gamma \in \Gamma$, and the generator $\iota \in \mathbb Z/2\mathbb Z$;
			\item \label{itm:Gamma-Action} the $\Gamma$-action is properly discontinuous and cocompact, and all its orbit maps $\Gamma \to \tilde{\mathcal{G}}$ are quasi-isometries so that $\partial\tilde{\mathcal{G}} \cong \Gboundary$ are homeomorphic;
			\item \label{itm:R-Action} the $\mathbb R$-action given by the flow $\phi$ is free, all its orbit maps $\mathbb R \to \tilde{\mathcal{G}}$ are quasi-isometric embeddings, and it induces a $(\Gamma \times \mathbb Z/2\mathbb Z)$-equivariant homeomorphism $\tilde{\mathcal{G}}/\mathbb R \cong \Gboundary^{(2)}$ where we use the diagonal left $\Gamma$-action on $\Gboundary^{(2)}$ and $\iota(x, y) = (y, x)$ for all $(x, y) \in \Gboundary^{(2)}$ and the generator $\iota \in \mathbb Z/2\mathbb Z$.
		\end{enumerate}
	\end{theorem}
	
	By \cref{itm:R-Action} in \cref{thm:GromovGeodesicFlow}, $\tilde{\mathcal{G}} \cong \Gboundary^{(2)} \times \mathbb R$ are homeomorphic and the Gromov geodesic flow $\phi: (\Gboundary^{(2)} \times \mathbb R) \times \mathbb R \to \Gboundary^{(2)} \times \mathbb R$ is simply by $\mathbb R$-translations on the last factor. We denote the left $\Gamma$-action on $\tilde{\mathcal{G}} \cong \Gboundary^{(2)} \times \mathbb R$ by ${}_*$. We denote by $\backslash{}_*$ quotients by ${}_*$. By the properties of the $\Gamma$-action, the metric $d$ descends to the quotient and we get a compact metric space $(\mathcal{G}, d)$ where $\mathcal{G} = \Gamma \backslash{}_* \tilde{\mathcal{G}}$. The Gromov geodesic flow also descends to $\phi: \mathcal{G} \times \mathbb R \to \mathcal{G}$.
	
	We need to discuss the construction of the metric $d$ on $\tilde{\mathcal{G}}$ for later reference. There exists a locally finite open cover $\{(\gamma {}_* U_j, \gamma {}_* x_j)\}_{\gamma \in \Gamma, 1 \leq j \leq j_0}$ of $\tilde{\mathcal{G}} \cong \Gboundary^{(2)} \times \mathbb R$, for some $j_0 \in \mathbb N$, with distinguished points $x_j \in U_j$ and such that $\gamma {}_* U_j \cap \gamma' {}_* U_j = \varnothing$ for all distinct $\gamma, \gamma' \in \Gamma$ and $1 \leq j \leq j_0$. Let $\gamma, \gamma' \in \Gamma$ and $1 \leq j \leq j_0$. We assume that the open set is a flow box $U_j \cong B_j \times (0, \epsilon_j)$ where $B_j \subset \Gboundary^{(2)}$ is precompact and the second factor is parametrized according to the Gromov geodesic flow $\phi$. Fix the metric $d_{e, j} = d_{x_j} \times d_{x_j} \times d_{\mathbb R}$ on $U_j$ where $d_{x_j}$ is the visual metric on $\Gboundary$ from the viewpoint of $x_j$ and $d_{\mathbb R}$ is the Euclidean metric on $\mathbb R$, and then fix $d_{\gamma, j} = \gamma_*d_{e, j}$ on $\gamma {}_* U_j$. By \cite[Lemma 4.10]{Cha94}, we can assume that the parametrization $U_j \cong B_j \times (0, \epsilon_j)$ is such that $d_{\gamma, j}$ and $d_{\gamma', j'}$ are equivalent on the intersection $\gamma {}_* U_j \cap \gamma' {}_* U_{j'}$ whenever it is nonempty. Now, by \cite[Lemma 4.11]{Cha94} (cf. \cite[Lemma 5.2]{BCLS15}), a left $\Gamma$-invariant metric $d$ can be constructed from the family of metrics $\{d_{\gamma, j}\}_{\gamma \in \Gamma, 1 \leq j \leq j_0}$ such that it is equivalent to $d_{\gamma, j}$ on $\gamma {}_* U_j$.
	
	Let $\rho: \Gamma \to \SL_n(\mathbb R)$ be a projective Anosov representation. In \cite{BCLS15} a metric space $(\tilde{\mathcal{G}}_\rho, d_\rho)$ was introduced equipped with a left $\Gamma$-action by isometries and a $\Gamma$-equivariant flow $\phi^\rho: \tilde{\mathcal{G}}_\rho \times \mathbb R \to \tilde{\mathcal{G}}_\rho$. We then obtain a metric space $(\mathcal{G}_\rho, d_\rho)$ where $\mathcal{G}_\rho = \Gamma \backslash \tilde{\mathcal{G}}_\rho$. The flow $\phi^\rho$ descends to a flow $\phi^\rho: \mathcal{G}_\rho \times \mathbb R \to \mathcal{G}_\rho$. We will call both of these flows the \emph{$\rho$-geodesic flow}. The following theorem summarizes \cite[Propositions 4.1 and 5.7]{BCLS15}.
	
	\begin{theorem}
		\label{thm:rho-GeodesicFlowConjugateToGromovGeodesicFlowAndMetricAnosov}
		Let $\rho: \Gamma \to \SL_n(\mathbb R)$ be a projective Anosov representation. Then, we have:
		\begin{enumerate}
			\item \label{itm:ConjugateToGromovGeodesicFlow} the $\rho$-geodesic flow on $\mathcal{G}_\rho$ is H\"{o}lder conjugate to a H\"{o}lder reparametrization of the Gromov geodesic flow on $\mathcal{G}$;
			\item \label{itm:MetricAnosov} the $\rho$-geodesic flow on $\mathcal{G}_\rho$ is a metric Anosov flow.
		\end{enumerate}
	\end{theorem}
	
	\begin{remark}
		Existence of a projective Anosov representation of $\Gamma$ is guaranteed using the \emph{Pl\"{u}cker representation}, i.e., $\tilde{\rho}: G \to \SL(V)$ defined by $\tilde{\rho}(g) = \bigwedge^m \Ad_g$ for all $g \in G$ where $V = \bigwedge^m \LieG$ and $m = \dim(\LieN^+) = \#\Phi^+$. We fix $\rho = \tilde{\rho}|_{\Gamma}$ henceforth.
	\end{remark}
	
	Analogous to \cref{itm:ConjugateToGromovGeodesicFlow} in \cref{thm:rho-GeodesicFlowConjugateToGromovGeodesicFlowAndMetricAnosov}, we have the following theorem. This result can also be found in \cite[Proposition A.1]{Car21} by Carvajales. We provide an independent and more detailed proof using techniques from \cite{BCLS15,BCLS18}, the Morse property, and an analogue of Sullivan's shadow lemma.
	
	\begin{theorem}
		\label{thm:TranslationFlowConjugateToGromovGeodesicFlow}
		We have the following:
		\begin{enumerate}
			\item 
			\label{itm:TranslationFlowConjugateToGromovGeodesicFlow}
			the translation flow on $\mathcal{X}$ is conjugate to a reparametrization of the Gromov geodesic flow on $\mathcal{G}$;
			\item moreover, the conjugating map and the reparametrization are H\"{o}lder.
			\label{itm:HolderReparametrization}
		\end{enumerate}
	\end{theorem}
	%The translation flow on $\mathcal{X}$ is H\"{o}lder conjugate to a H\"{o}lder reparametrization of the Gromov geodesic flow on $\mathcal{G}$.
	
	To avoid interruptions, we will now obtain an appropriate Markov section for the translation flow before we prove \cref{thm:TranslationFlowConjugateToGromovGeodesicFlow} at the end of this subsection. A consequence of \cref{itm:ConjugateToGromovGeodesicFlow} in \cref{thm:rho-GeodesicFlowConjugateToGromovGeodesicFlowAndMetricAnosov} and \cref{thm:TranslationFlowConjugateToGromovGeodesicFlow} is that $\mathcal{G}_\rho$ and $\mathcal{X}$ are compact. Thus, using \cref{itm:MetricAnosov} in \cref{thm:rho-GeodesicFlowConjugateToGromovGeodesicFlowAndMetricAnosov} and recalling \cref{thm:MarkovSectionForMetricAnosovFlow}, we conclude that the $\rho$-geodesic flow on $\mathcal{G}_\rho$ has a Markov section with respect to the foliations coming from the metric Anosov property.
	
	Fix such a Markov section on $\mathcal{G}_\rho$ of size $\hat{\delta}_\rho > 0$. Using the homeomorphism $\mathcal{G}_\rho \to \mathcal{X}$ which conjugates the $\rho$-geodesic flow to the translation flow, denote the image of the Markov section in $\mathcal{X}$ by $\mathcal{R}^* = \{R_1^*, R_2^*, \dotsc, R_N^*\}$ and the image foliations by $W^{\rho, \mathrm{su}}$, $W^{\rho, \mathrm{ss}}$, $W^{\rho, \mathrm{wu}}$, and $W^{\rho, \mathrm{ws}}$. Observe that $G$ has two foliations given by the right $N^+$-orbits and the right $N^-$-orbits called the \emph{horospherical foliations}. Since $M \subset P^- = N_G(N^-)$ and $M \subset P^+ = N_G(N^+)$, they descend to well-defined horospherical foliations on $G/M$. Using the Hopf parametrization and the projection map $\pi_\psi$, the horospherical foliations induce the natural \emph{strong unstable} and \emph{strong stable} foliations $W^{\mathrm{su}}$ and $W^{\mathrm{ss}}$ on $\mathcal{X}$ defined by
	\begin{align}
		\label{eqn:StrongUnstableAndStrongStableFoliations}
		\begin{aligned}
			W^{\mathrm{su}}(\Gamma (x, y, t)) &= \bigl\{\Gamma\bigl((gh)^+, g^-, \psi\bigl(\beta_{(gh)^+}(e, gh)\bigr)\bigr): h \in N^+, (gh)^+ \in \limitset\bigr\}, \\
			W^{\mathrm{ss}}(\Gamma (x, y, t)) &= \bigl\{\Gamma\bigl(g^+, (gn)^-, \psi\bigl(\beta_{(gn)^+}(e, gn)\bigr)\bigr): n \in N^-, (gn)^- \in \limitset\bigr\}
		\end{aligned}
	\end{align}
	for any choice of $g \in G$ with $g^+ = x$, $g^- = y$, and $\psi(\beta_{g^+}(e, g)) = t$, for all $\Gamma (x, y, t) \in \mathcal{X}$. The following lemma ensures that $W^{\mathrm{su}}$ and $W^{\mathrm{ss}}$ are indeed well-defined.
	
	\begin{lemma}
		\label{lem:FoliationsWellDefined}
		Let $\Gamma (x, y, t) \in \mathcal{X}$. Let $g_1, g_2 \in G$ with $g_1^+ = g_2^+ = x$, $g_1^- = g_2^- = y$, and $\psi\bigl(\beta_{g_1^+}(e, g_1)\bigr) = \psi\bigl(\beta_{g_2^+}(e, g_2)\bigr) = t$. Then:
		\begin{enumerate}
			\item\label{itm:StrongUnstable} for all $h_1, h_2 \in N^+$ with $(g_1h_1)^+ = (g_2h_2)^+$, we have 
			\begin{align*}
				\psi\bigl(\beta_{(g_1h_1)^+}(e, g_1h_1)\bigr) = \psi\bigl(\beta_{(g_2h_2)^+}(e, g_2h_2)\bigr);
			\end{align*}
			\item\label{itm:StrongStable} $\psi\bigl(\beta_{(g_1n)^+}(e, g_1n)\bigr) = t$ for all $n \in N^-$.
		\end{enumerate}
	\end{lemma}
	
	\begin{proof}
		Let $\Gamma (x, y, t) \in \mathcal{X}$. Let $g_1, g_2 \in G$ be as in the lemma. We first show \cref{itm:StrongUnstable}. Suppose $h_1, h_2 \in N^+$ with $(g_1h_1)^+ = (g_2h_2)^+$. We have $g_1 = g_2am$ for some $am \in AM$ from the hypothesis $g_1^+ = g_2^+$ and $g_1^- = g_2^-$. Thus
		\begin{align*}
			\psi\bigl(\beta_{g_1^+}(e, g_1)\bigr) = \psi\bigl(\beta_{g_2^+}(e, g_2am)\bigr) = \psi\bigl(\beta_{g_2^+}(e, g_2)\bigr) + \psi\bigl(\beta_{g_2^+}(g_2, g_2am)\bigr).
		\end{align*}
		Using the hypothesis $\psi\bigl(\beta_{g_1^+}(e, g_1)\bigr) = \psi\bigl(\beta_{g_2^+}(e, g_2)\bigr)$ above gives
		\begin{align}
			\label{eqn:BusemannOf_a_Vanishes}
			\psi\bigl(\beta_{e^+}(e, a)\bigr) = \psi\bigl(\beta_{e^+}(e, am)\bigr) = 0.
		\end{align}
		Note that $(g_1h_1)^- = g_1^- = g_2^- = (g_2h_2)^-$. Since we also have $(g_1h_1)^+ = (g_2h_2)^+$, we conclude $g_1h_1 = g_2h_2a'm'$ for some $a'm' \in AM$. This implies $m \cdot a \cdot h_1 = m' \cdot a' \cdot (a'm')^{-1}h_2a'm'$. Recalling the Iwasawa decomposition, we have $a = a'$, $m = m'$, and $h_1 = (a'm')^{-1}h_2a'm'$. Thus, using \cref{eqn:BusemannOf_a_Vanishes}, we have
		\begin{align*}
			\psi\bigl(\beta_{(g_1h_1)^+}(e, g_1h_1)\bigr) &= \psi\bigl(\beta_{(g_2h_2)^+}(e, g_2h_2am)\bigr) \\
			&= \psi\bigl(\beta_{(g_2h_2)^+}(e, g_2h_2)\bigr) + \psi\bigl(\beta_{(g_2h_2)^+}(g_2h_2, g_2h_2am)\bigr) \\
			&= \psi\bigl(\beta_{(g_2h_2)^+}(e, g_2h_2)\bigr) + \psi\bigl(\beta_{e^+}(e, am)\bigr) \\
			&= \psi\bigl(\beta_{(g_2h_2)^+}(e, g_2h_2)\bigr).
		\end{align*}
		
		Now we show \cref{itm:StrongStable}. For all $n \in N^-$, we have
		\begin{align*}
			\psi\bigl(\beta_{(g_1n)^+}(e, g_1n)\bigr) &= \psi\bigl(\beta_{g_1^+}(e, g_1n)\bigr) \\
			&= \psi\bigl(\beta_{g_1^+}(e, g_1)\bigr) + \psi\bigl(\beta_{g_1^+}(g_1, g_1n)\bigr) \\
			&= t + \psi\bigl(\beta_{e^+}(e, n)\bigr) \\
			&= t.
		\end{align*}
	\end{proof}
	
	The foliations $W^{\mathrm{su}}$ and $W^{\mathrm{ss}}$ are transverse to $\mathcal{W}$. We then obtain the corresponding \emph{weak unstable} and \emph{weak stable} foliations $W^{\mathrm{wu}} = (W^{\mathrm{su}})^{\mathrm{c}}$ and $W^{\mathrm{ws}} = (W^{\mathrm{ss}})^{\mathrm{c}}$ on $\mathcal{X}$ and $(W^{\mathrm{wu}}, W^{\mathrm{ss}})$ and $(W^{\mathrm{ws}}, W^{\mathrm{su}})$ have local product structures with some constant $\epsilon_0 > 0$ which works for \cref{def:LocalProductStructure} and for transversality. Fix some $\hat{\delta} \in (0, \epsilon_0)$. Note that if $\hat{\delta}_\rho$ is sufficiently small, then $\mathcal{R}^*$ will already satisfy all the properties required for a Markov section for the translation flow $\mathcal{W}$ with an appropriately modified first return time map. However, the rectangles of $\mathcal{R}^*$ are with respect to $(W^{\rho, \mathrm{wu}}, W^{\rho, \mathrm{ss}})$ rather than $(W^{\mathrm{wu}}, W^{\mathrm{ss}})$. We will modify the rectangles of $\mathcal{R}^*$ to ensure the latter. From definitions, for all $x \in \mathcal{X}$, we have a coordinate chart $\varphi_x: O_x \to \tilde{V}_x^1 \times (-\epsilon_0, \epsilon_0) \times V_x^2$ from the local product structure of $(W^{\mathrm{wu}}, W^{\mathrm{ss}})$. Let $\proj_x: \tilde{V}_x^1 \times (-\epsilon_0, \epsilon_0) \times V_x^2 \to \tilde{V}_x^1 \times \{0\} \times V_x^2$ denote the projection map given by $(v_1,t,v_2) \mapsto (v_1,v_2)$ and denote by $\proj_x^{\mathrm{c}}: \tilde{V}_x^1 \times (-\epsilon_0, \epsilon_0) \times V_x^2 \to (-\epsilon_0, \epsilon_0)$ the projection given by $(v_1,t,v_2) \mapsto t$. Choose any $x_k \in R_k^*$ and define
	\begin{align*}
		R_k &= \varphi_{x_k}^{-1}(\proj_{x_k}(\varphi_{x_k}(R_k^*))) \qquad \text{for all $1 \leq k \leq N$}, \\
		\mathcal{R} &= \{R_1, R_2, \dotsc, R_N\}.
	\end{align*}
	By compactness of $\mathcal{X}$, we can ensure that $\diam(R_k) < \hat{\delta}$ for all $1 \leq k \leq N$ by requiring $\hat{\delta}_\rho$ be sufficiently small. From \cite[Subsection 5.2]{BCLS15}, we can directly deduce that $W^{\mathrm{wu}} = W^{\rho, \mathrm{wu}}$ and $W^{\mathrm{ws}} = W^{\rho, \mathrm{ws}}$. Hence the above definition truly gives rectangles with respect to $(W^{\mathrm{wu}}, W^{\mathrm{ss}})$. Moreover, the Markov property will be automatically satisfied once we ensure that $\mathcal{R}$ is a complete set of rectangles of size $\hat{\delta}$ with respect to $(W^{\mathrm{wu}}, W^{\mathrm{ss}})$. Now, note that by the construction of Markov sections (see \cite{Bow70,Rat73,Pol87}), we can first ensure a positive lower bound for $\underline{\tau} > 0$ and then independently choose a uniform upper bound for $\diam(R_k^*)$ for all $1 \leq k \leq N$. Thus, again by compactness of $\mathcal{X}$, we can also ensure that $\sup_{x \in R_k}\diam(\proj_x^{\mathrm{c}}(\varphi_x(R_k))) < \frac{1}{4}\underline{\tau}$ and $\tau(x_k) + \frac{1}{2}\underline{\tau} < \hat{\delta}$ for all $1 \leq k \leq N$ by requiring $\hat{\delta}_\rho$ be sufficiently small. This guarantees \cref{itm:MarkovProperty1,itm:MarkovProperty2} in \cref{def:CompleteSetOfRectangles}. This completes the modification and $\mathcal{R}$ is the desired Markov section on $\mathcal{X}$ with respect to $(W^{\mathrm{su}}, W^{\mathrm{ss}})$. We record this in the following corollary.
	
	\begin{corollary}
		\label{cor:TranslationFlowHasMarkovSection}
		The translation flow on $\mathcal{X}$ has a Markov section with respect to $(W^{\mathrm{su}}, W^{\mathrm{ss}})$, where $W^{\mathrm{su}}$ and $W^{\mathrm{ss}}$ are as in \cref{eqn:StrongUnstableAndStrongStableFoliations}.
	\end{corollary}
	
	We now provide a proof of \cref{thm:TranslationFlowConjugateToGromovGeodesicFlow}.
	
	\begin{proof}[Proof of \cref{thm:TranslationFlowConjugateToGromovGeodesicFlow}]
		First we prove \cref{itm:TranslationFlowConjugateToGromovGeodesicFlow}. The left $\Gamma$-action on $\tilde{\mathcal{G}} \cong \Gboundary^{(2)} \times \mathbb R$ induces a left $\Gamma$-action on $\limitset^{(2)} \times \mathbb R$ via $\zeta$ which we also denote by ${}_*$. It is different from the left $\Gamma$-action on $\limitset^{(2)} \times \mathbb R$ from \cref{subsec:TheVectorBundleAndTheTranslationFlow} which was denoted by $\cdot$. Similarly, the Gromov geodesic flow $\phi: \tilde{\mathcal{G}} \times \mathbb R \to \tilde{\mathcal{G}}$ induces a flow $\phi: \bigl(\limitset^{(2)} \times \mathbb R\bigr) \times \mathbb R \to \limitset^{(2)} \times \mathbb R$ via $\zeta$.
		
		Let $\tilde{L} = \bigl(\limitset^{(2)} \times \mathbb R\bigr) \times \mathbb R_{>0}$ be the trivial principal $\mathbb R_{>0}$-bundle. We define a left $\Gamma$-action on $\tilde{L}$ by principal bundle automorphisms
		\begin{align*}
			\gamma {}_* (x, y, s, v) = \bigl(\gamma {}_* (x, y, s), ve^{\psi(\beta_x(\gamma^{-1}, e))}\bigr)
		\end{align*}
		for all $\gamma \in \Gamma$ and $(x, y, s, v) \in \bigl(\limitset^{(2)} \times \mathbb R\bigr) \times \mathbb R_{>0}$. Then, we have the quotient $L = \Gamma \backslash{}_* \tilde{L} \cong \Gamma \backslash{}_* \bigl(\limitset^{(2)} \times \mathbb R\bigr) \times \mathbb R_{>0}$ by a similar reasoning as in \cref{subsec:TheVectorBundleAndTheTranslationFlow}. We have a natural map $\exp: \limitset^{(2)} \times \mathbb R \to \limitset^{(2)} \times \mathbb R_{>0}$ induced by $\exp: \mathbb R \to \mathbb R_{>0}$ whose inverse we denote by $\log: \limitset^{(2)} \times \mathbb R_{>0} \to \limitset^{(2)} \times \mathbb R$. Note that the left $\Gamma$-action ${}_*$ on $\tilde{L}$ restricts to a left $\Gamma$-action ${}_*$ on the $\limitset^{(2)} \times \mathbb R_{>0}$ factor which is isomorphic to the left $\Gamma$-action $\cdot$ on $\limitset^{(2)} \times \mathbb R$ via $\log$. We also trivially lift the flow $\phi$ to a flow $\hat{\phi}: \tilde{L} \times \mathbb R \to \tilde{L}$ defined by $\hat{\phi}_t(x, y, s, v) = (\phi_t(x, y, s), v)$ for all $(x, y, s, v) \in \bigl(\limitset^{(2)} \times \mathbb R\bigr) \times \mathbb R_{>0}$ and $t \in \mathbb R$, which is automatically left $\Gamma$-equivariant and hence descends to a flow $\hat{\phi}: L \times \mathbb R \to L$.
		
		Using a partition of unity subordinate to the cover $\{\Gamma {}_* U_j\}_{1 \leq j \leq j_0}$, we can equip the principal $\mathbb R_{>0}$-bundle $\tilde{L}$ with a left $\Gamma$-invariant continuous norm $\|\cdot\|$ by viewing the typical fibers as subsets $\mathbb R_{>0} \subset \mathbb R$, by which we simply mean that the unique left $\Gamma$-equivariant section $u: \limitset^{(2)} \times \mathbb R \to \tilde{L}$ with $\|u(x, y, s)\| = 1$ for all $(x, y, s) \in \limitset^{(2)} \times \mathbb R$ is continuous. It is convenient to define $\hat{u}: \limitset^{(2)} \times \mathbb R \to \mathbb R_{>0}$ such that $u(x, y, s) = (x, y, s, \hat{u}(x, y, s))$ for all $(x, y, s) \in \limitset^{(2)} \times \mathbb R$. Then, using the map $\log: \limitset^{(2)} \times \mathbb R_{>0} \to \limitset^{(2)} \times \mathbb R$, we see that in fact $\log \circ \hat{u}$ is continuous. Moreover, we can ensure that $\hat{u}(x, y, \cdot)$ is continuously differentiable. By left $\Gamma$-invariance of the norm $\|\cdot\|$, it descends to a continuous norm $\|\cdot\|$ on $L$.
		
		We argue that the flow $\hat{\phi}: L \times \mathbb R \to L$ is \emph{discretely contracting}, i.e., there exist $T, \eta > 0$ such that
		\begin{align}
			\label{eqn:DiscretelyContracting}
			\|\hat{\phi}_T\bigl(\Gamma {}_* (X, v)\bigr)\| \leq e^{-\eta}\|\Gamma {}_* (X, v)\| \qquad \text{for all $\Gamma {}_* (X, v) \in L$}.
		\end{align}
		Let $X \in \limitset^{(2)} \times \mathbb R$. Denote by $\tilde{X} \in \tilde{\mathcal{G}}$ the point corresponding to $X$ via $\zeta^{-1}$. Fix a compact fundamental domain $\tilde{D} \subset \tilde{\mathcal{G}}$ containing $\tilde{X}$ and denote the corresponding compact set via $\zeta$ by $D \subset \limitset^{(2)} \times \mathbb R$. Let $\{\gamma_j\}_{j \in \mathbb Z_{\geq 0}} \subset \Gamma$ be a geodesic ray with respect to the word metric such that $\gamma_0 = e \in \Gamma$ and $\gamma_\infty = \tilde{X}^+ \in \Gboundary$. Define the orbit map $\Theta: \Gamma \to \tilde{\mathcal{G}}$ by $\Theta(\gamma) = \gamma {}_* \tilde{X}$ for all $\gamma \in \Gamma$. By \cref{itm:Gamma-Action} in \cref{thm:GromovGeodesicFlow}, $\{\Theta(\gamma_j)\}_{j \in \mathbb Z_{\geq 0}} \subset \Gamma$ is a sequence along a quasi-geodesic ray with $\Theta(\gamma_\infty) = \tilde{X}^+$. But $[0, +\infty) \to \tilde{\mathcal{G}}$ defined by $t \mapsto \phi_t(\tilde{X})$ is also a quasi-geodesic ray with $\phi_\infty(\tilde{X}) = \tilde{X}^+$. So, there exists a sequence $\{T_{X, j}\}_{j \in \mathbb N} \subset \mathbb R_{>0}$ such that $d(\phi_{T_{X, j}}(\tilde{X}), \gamma_j {}_* \tilde{X})$ is uniformly bounded for all $j \in \mathbb N$. 
		Thus, there exists a finite subset $F' \subset \Gamma$ and a sequence $\{f_j'\}_{j \in \mathbb N} \subset F'$ such that $\phi_{T_{X, j}}(X) \in \gamma_j f_j' {}_* D$. Note that $\lim_{j \to \infty} \gamma_j f_j' {}_* X = X^+$. By \cite[Lemmas 4.5.1 and 4.5.2]{Ben97}, we can fix a finite subset $F \subset \Gamma$ and $C_1 > 0$ such that for all $\gamma \in \Gamma$, there exists $f \in F$ such that $\|\mu(\gamma f) - \lambda(\gamma f)\| \leq C_1$. For all $j \in \mathbb N$, let $f_j \in F$ such that
		\begin{align*}
			\|\mu(\gamma_j f_j) - \lambda(\gamma_j f_j)\| \leq C_1.
		\end{align*}
		We still have $\lim_{j \to \infty} \gamma_j f_j {}_* X = X^+$. Now, using the Morse property \cite[Proposition 5.12]{LO20b} of Kapovich--Leeb--Porti \cite[Proposition 5.16]{KLP17}, we can conclude that $\lim_{j \to \infty} \gamma_j o = X^+$ conically (cf. \cite[Proposition 7.4]{LO20b}). Thus, there exist $R > 0$ such that $X^+ \in O_R(o, \gamma_j f_j o) \subset \Fboundary$ for all $j \in \mathbb N$. Using the analogue of Sullivan's shadow lemma due to Lee--Oh \cite[Lemma 5.7]{LO20b} and Thirion \cite{Thi07}, there exists $C_2 > 0$ corresponding to $R$ such that
		\begin{align*}
			\|\beta_{X^+}(o, \gamma_j f_jo) - \mu(\gamma_j f_j)\| \leq C_2
		\end{align*}
		for all $j \in \mathbb N$. Combining the above inequalities and applying $\psi$, there exists $C_3 > 0$ such that 
		\begin{align*}
			|\psi(\beta_{X^+}(\gamma_j f_j, e)) + \psi(\lambda(\gamma_j f_j))| \leq C_3 \qquad \text{for all } j \in \mathbb N.    
		\end{align*} Now, $(\gamma_j f_j)^{-1} {}_* \phi_{T_{X, j}}(X) \in f_j^{-1} f_j' {}_* D$ where $f_j^{-1} f_j'$ is in the finite set $F^{-1} \cdot F'$. Thus, by continuity on a compact set, we can fix $C_4 = \sup_{X' \in F^{-1} \cdot F' {}_* D} \frac{\|(X', 1)\|}{\|(X, 1)\|}$. Then, recalling that $\phi_t(X)^+ = X^+$ for all $t \in \mathbb R$, we have
		\begin{align*}
			\bigl\|\hat{\phi}_{T_{X, j}}(X, v)\bigr\| &= \|(\phi_{T_{X, j}}(X), v)\| = \|(\gamma_j f_j)^{-1} {}_* (\phi_{T_{X, j}}(X), v)\| \\
			&= \left\|\left((\gamma_j f_j)^{-1} {}_* \phi_{T_{X, j}}(X), ve^{\psi\bigl(\beta_{X^+}(\gamma_j f_j, e)\bigr)}\right)\right\| \\
			&= e^{\psi\bigl(\beta_{X^+}(\gamma_j f_j, e)\bigr)} \|((\gamma_j f_j)^{-1} {}_* \phi_{T_{X, j}}(X), v)\| \\
			&\leq C_4e^{C_3}e^{-\psi(\lambda(\gamma_j f_j))}
		\end{align*}
		for all $j \in \mathbb N$. Since $\lim_{j \to \infty} \|\lambda(\gamma_j f_j)\| = +\infty$, there exists a sufficiently large $j \in \mathbb N$ such that setting $T_X = T_{X, j}$ and $\eta_X = \frac{1}{2}\psi(\lambda(\gamma_j f_j)) > 0$ by \cref{itm:PositiveOnLimitCone} in \cref{thm:PSDensityForAnosovSubgroups}, we have
		\begin{align*}
			\bigl\|\hat{\phi}_{T_X}(X, v)\bigr\| \leq e^{-\eta_X}.
		\end{align*}
		Descending to the quotient, we have shown that for all $\Gamma {}_* X \in \Gamma \backslash{}_* \bigl(\limitset^{(2)} \times \mathbb R\bigr)$, there exists $T_{\Gamma {}_* X}, \eta_{\Gamma {}_* X} > 0$ such that
		\begin{align*}
			\bigl\|\hat{\phi}_{T_{\Gamma {}_* X}}\bigl(\Gamma {}_* (X, v)\bigr)\bigr\| \leq e^{-\eta_{\Gamma {}_* X}} \|(X, v)\| \qquad \text{for all $v \in \mathbb R_{>0}$}.
		\end{align*}
		The discretely contracting property in \cref{eqn:DiscretelyContracting} follows by a compactness argument.
		
		Repeating the averaging trick as in \cite[Lemma 4.3]{BCLS15}, we can modify the norm $\|\cdot\|$ such that the flow $\hat{\phi}: L \times \mathbb R \to L$ is \emph{continuously contracting}, i.e., there exists $\eta > 0$ such that
		\begin{align*}
			\|\hat{\phi}_t\bigl(\Gamma {}_* (X, v)\bigr)\| \leq e^{-\eta t}\|\Gamma {}_* (X, v)\| \qquad \text{for all $\Gamma {}_* (X, v) \in L$ and $t \geq 0$}.
		\end{align*}
		Note that the averaging process preserves the regularity of $u$ and $\hat{u}$. We now show that the partial derivative in the third coordinate $\partial_3(\log \circ \hat{u})$ is positive which implies that $\partial_3\hat{u}$ is also positive. Using the continuous contraction property and the definition of $u$, we have
		\begin{align*}
			e^{\eta t} \leq \frac{\|u(x, y, s + t)\|}{\|\hat{\phi}_t(u(x, y, s))\|} = \frac{\|(x, y, s + t, \hat{u}(x, y, s + t))\|}{\|(x, y, s + t, \hat{u}(x, y, s))\|} = \frac{\hat{u}(x, y, s + t)}{\hat{u}(x, y, s)}
		\end{align*}
		for all $(x, y, s) \in \limitset^{(2)} \times \mathbb R$ and $t \in \mathbb R$. Hence
		\begin{align*}
			\partial_3(\log \circ \hat{u})(x, y, s) = \lim_{t \to 0} \frac{(\log \circ \hat{u})(x, y, s + t) - (\log \circ \hat{u})(x, y, s)}{t} \geq \eta > 0
		\end{align*}
		for all $(x, y, s) \in \limitset^{(2)} \times \mathbb R$.
		
		We now define the space of orbits of $\hat{\phi}$ by $\tilde{J} = \tilde{L}/\mathbb R = \limitset^{(2)} \times \mathbb R_{>0}$ which is also a principal $\mathbb R_{>0}$-bundle. Define $J = \Gamma \backslash{}_* \tilde{J}$. We define a flow $\kappa: \tilde{J} \times \mathbb R \to \tilde{J}$ by $\kappa_t(x, y, v) = (x, y, ve^t)$ for all $(x, y, v) \in \tilde{J}$ and $t \in \mathbb R$. The flow $\kappa$ is left $\Gamma$-equivariant and descends to a flow $\kappa: J \times \mathbb R \to J$. Then, $\log$ is a conjugation between the translation flow $\mathcal{W}$ and the flow $\kappa$. Taking quotients, the translation flow $\mathcal{W}$ on $\mathcal{X}$ is conjugate to the flow $\kappa$ on $J$.
		
		Since $\zeta$ is continuous, we need to show that the flow $\kappa$ on $J$ is conjugate to a reparametrization of the flow $\phi$ on $\Gamma \backslash{}_* \bigl(\limitset^{(2)} \times \mathbb R\bigr)$ to prove \cref{itm:TranslationFlowConjugateToGromovGeodesicFlow}. Recall the section $u: \limitset^{(2)} \times \mathbb R \to \tilde{L}$ and its associated function $\hat{u}: \limitset^{(2)} \times \mathbb R \to \mathbb R_{>0}$. We define the left $\Gamma$-equivariant orbit preserving continuous map $\Psi: \limitset^{(2)} \times \mathbb R \to \tilde{J}$ by
		\begin{align*}
			\Psi(x, y, s) = (x, y, \hat{u}(x, y, s)) \qquad \text{for all $(x, y, s) \in \limitset^{(2)} \times \mathbb R$}.
		\end{align*}
		It suffices to show that $\Psi$ is a homeomorphism. Recalling that $\partial_3\hat{u}$ is positive, we conclude that $\Psi$ is injective along orbits of $\phi$ and hence injective everywhere. We now show that $\Psi$ is a proper map. To obtain a contradiction, suppose $\Psi$ is not proper. As $\limitset^{(2)} \times \mathbb R$ and $\tilde{J}$ are metrizable, there exists a sequence $\{(x_j, y_j, s_j)\}_{j \in \mathbb N}$ which escapes all compact subsets of $\limitset^{(2)} \times \mathbb R$ such that $\lim_{j \to \infty} \Psi(x_j, y_j, s_j) = (x, y, v)$ for some $(x, y, v) \in \tilde{J}$. We can assume that either $\{s_j\}_{j \in \mathbb N}$ is negative and $\lim_{j \to \infty} s_j = -\infty$ or $\{s_j\}_{j \in \mathbb N}$ is positive and $\lim_{j \to \infty} s_j = +\infty$. Then, we have $\lim_{j \to \infty} \hat{u}(x_j, y_j, s_j) = v$ and hence
		\begin{align*}
			\lim_{j \to \infty} \hat{\phi}_{-s_j}(u(x_j, y_j, s_j)) = \lim_{j \to \infty} (x_j, y_j, 0, \hat{u}(x_j, y_j, s_j)) = (x, y, 0, v).
		\end{align*}
		Taking norms above and using the continuously contracting property, we either have $\lim_{j \to \infty} \bigl\|\hat{\phi}_{-s_j}(u(x_j, y_j, s_j))\bigr\| \leq e^{\eta s_j} = 0$ or $\lim_{j \to \infty} \bigl\|\hat{\phi}_{-s_j}(u(x_j, y_j, s_j))\bigr\| \geq e^{\eta s_j} = +\infty$ both of which contradict $\|(x, y, 0, v)\| \in \mathbb R_{> 0}$. Thus, $\Psi$ is a proper continuous injection. Now, $\Psi$ restricted to orbits of $\phi$ surjects onto fibers of $\tilde{J}$ since proper continuous injections $\mathbb R \to \mathbb R$ are surjections (and a fortiori homeomorphisms). Thus, $\Psi$ is a proper continuous bijection. This is sufficient to conclude that $\Psi$ is in fact a homeomorphism since $\tilde{J}$ is first countable and Hausdorff. 
		
		Now we prove \cref{itm:HolderReparametrization}. First we construct left $\Gamma$-invariant metrics on $\tilde{J}$ and $\tilde{L}$. Having established \cref{itm:TranslationFlowConjugateToGromovGeodesicFlow}, we can use the left $\Gamma$-invariant metric $d$ on $\limitset^{(2)} \times \mathbb R$ from \cref{subsec:TheVectorBundleAndTheTranslationFlow}. It induces a left $\Gamma$-invariant metric $d$ on $\limitset^{(2)} \times \mathbb R_{>0}$ so that $\exp$ is an isometry. This gives a metric $d$ on $\tilde{J}$ which descends to a metric $d$ on $J$. We use the locally finite pointed cover $\{(\gamma {}_* U_j, \gamma {}_* x_j)\}_{\gamma \in \Gamma, 1 \leq j \leq j_0}$. For all $\gamma \in \Gamma$ and $1 \leq j \leq j_0$, by interchanging the last two factors, we can fix the metric $d_{e, j} = d \times d_{\mathbb R}$ on $U_j \times \mathbb R_{>0} \cong B_j \times \mathbb R_{>0} \times (0, \epsilon_j)$, and then fix $d_{\gamma, j} = \gamma_*d_{e, j}$ on $\gamma {}_* (U_j \times \mathbb R_{>0})$. We can now repeat the construction outlined previously as in \cite{Cha94}. We thus obtain a left $\Gamma$-invariant metric $d$ on $\tilde{L}$ which is equivalent to $d_{\gamma, j}$ on $\gamma {}_* (U_j \times \mathbb R_{>0})$ for all $\gamma \in \Gamma$ and $1 \leq j \leq j_0$. It descends to a metric $d$ on $L$.
		
		Then, the norm $\|\cdot\|$ on $\tilde{L}$ is Lipschitz, meaning that the $\Gamma$-equivariant section $u$ is Lipschitz continuous. Consequently, $\log \circ \hat{u}$ is Lipschitz continuous and the norm $\|\cdot\|$ on $L$ is Lipschitz. Thus, $\Psi$ is Lipschitz continuous. We conclude that the flow $\kappa$ on $J$ is Lipschitz conjugate to \emph{some} reparametrization of the flow $\phi$ on $\Gamma \backslash{}_* \bigl(\limitset^{(2)} \times \mathbb R\bigr)$. Since $\zeta$ is H\"{o}lder continuous, it suffices to show that the reparametrization is \emph{Lipschitz}. At the level of covers, let $\alpha, \alpha^*: (\limitset^{(2)} \times \mathbb R) \times \mathbb R \to \mathbb R$ be the reparametrization and its inverse reparametrization. Note that they satisfy $\alpha(\gamma {}_* X, t) = \alpha(X, t)$ and $\alpha^*(\gamma {}_* X, t) = \alpha^*(X, t)$ for all $X \in \limitset^{(2)} \times \mathbb R$, $t \in \mathbb R$, and $\gamma \in \Gamma$. Then for all $X \in \limitset^{(2)} \times \mathbb R$ and $t \in \mathbb R$, we have $\kappa_{\alpha^*(X, t)}(\Psi(X)) = \Psi(\phi_t(X))$ which implies $\hat{u}(X)e^{\alpha^*(X, t)} = \hat{u}(\phi_t(X))$ from definitions and hence
		\begin{align*}
			\alpha^*(X, t) = \log(\hat{u}(\phi_t(X))) - \log(\hat{u}(X)).
		\end{align*}
		Recalling that the flow $\phi$ is simply by $\mathbb R$-translations on the last factor and that $\partial_3(\log \circ \hat{u})$ is positive continuous, we conclude that $\alpha^*(X, \cdot)$ is differentiable for all $X \in \limitset^{(2)} \times \mathbb R$ and $\partial_2\alpha^*$ is positive continuous. Thus, by \cref{rem:ReparametrizationRegularity}, the Lipschitz regularity of $\alpha$ follows from the Lipschitz regularity of $\alpha^*$. In fact, by compactness of $\Gamma \backslash{}_* \bigl(\limitset^{(2)} \times \mathbb R\bigr) \cong \mathcal{G}$, it suffices to show that $\alpha^*(\cdot, t)$ is \emph{locally} Lipschitz continuous for all $t \in \mathbb R$. But this follows from the fact that $\phi_t$ is a locally bi-Lipschitz homeomorphism for all $t \in \mathbb R$ and $\log \circ \hat{u}$ is Lipschitz continuous.
	\end{proof}
	
	\subsection{Symbolic dynamics}
	\label{subsec:SymbolicDynamics}
	The main purpose of Markov sections is to provide a coding for flows so that techniques from symbolic dynamics and thermodynamic formalism can be exploited. Henceforth, we let $\mathcal{R}$ be a Markov section for the translation flow on $\mathcal{X}$ with respect to $(W^{\mathrm{su}}, W^{\mathrm{ss}})$  given by \cref{cor:TranslationFlowHasMarkovSection} and we use the same associated notations from \cref{subsec:MetricAnosovFlowsAndMarkovSections}. Denote $\mathcal A = \{1, 2, \dotsc, N\}$ to be the \emph{alphabet} for the coding given by the Markov section $\mathcal{R}$ for the translation flow on $\mathcal{X}$. Define the $N \times N$ \emph{transition matrix} $T$ by
	\begin{align*}
		T_{j, k} =
		\begin{cases}
			1, & \interior(R_j) \cap \mathcal{P}^{-1}(\interior(R_k)) \neq \varnothing \\
			0, & \text{otherwise}
		\end{cases}
		\qquad
		\text{for all $1 \leq j, k \leq N$}.
	\end{align*}
	The transition matrix $T$ is \emph{topologically mixing} \cite[Theorem 4.3]{Rat73}, i.e., there exists $N_T \in \mathbb N$ such that $T^{N_T}$ consists only of positive entries. Define the spaces of bi-infinite and infinite \emph{admissible sequences}
	\begin{align*}
		\Sigma &= \{(\dotsc, x_{-1}, x_0, x_1, \dotsc) \in \mathcal A^{\mathbb Z}: T_{x_j, x_{j + 1}} = 1 \text{ for all } j \in \mathbb Z\}, \\
		\Sigma^+ &= \{(x_0, x_1, \dotsc) \in \mathcal A^{\mathbb Z_{\geq 0}}: T_{x_j, x_{j + 1}} = 1 \text{ for all } j \in \mathbb Z_{\geq 0}\}
	\end{align*}
	respectively. We will use the term \emph{admissible sequences} for finite sequences as well in the natural way. For all $x \in \Sigma$, we call the entries $x_j$ for all $j \geq 0$ the future coordinates and the entries $x_j$ for all $j < 0$ the past coordinates. For any fixed $\beta_0 \in (0, 1)$, we endow $\Sigma$ with the metric $d$ defined by $d(x, y) = \beta_0^{\inf\{|j| \in \mathbb Z_{\geq 0}: x_j \neq y_j\}}$ for all $x, y \in \Sigma$. We similarly endow $\Sigma^+$ with a metric which we also denote by $d$.
	
	\begin{definition}[Cylinder]
		For all $k \in \mathbb Z_{\geq 0}$, for all admissible sequences $x = (x_0, x_1, \dotsc, x_k)$, we define the corresponding \emph{cylinders} to be
		\begin{align*}
			\mathtt{C}_U[x] &= \{u \in U: \sigma^j(u) \in \interior(U_{x_j}) \text{ for all } 0 \leq j \leq k\} \\
			\mathtt{C}[x] &= \{y \in \Sigma^+: y_j = x_j \text{ for all } 0 \leq j \leq k\}
		\end{align*}
		with \emph{length} $\len(\mathtt{C}_U[x]) = \len(\mathtt{C}[x]) = \len(x) = k$. We will denote cylinders simply by $\mathtt{C}_U$ or $\mathtt{C}$ (or other typewriter style letters) when we do not need to specify the corresponding admissible sequence.
	\end{definition}
	
	By a slight abuse of notation, let $\sigma$ also denote the shift map on $\Sigma$ or $\Sigma^+$. There exist natural continuous surjections $\eta: \Sigma \to R$ and $\eta^+: \Sigma^+ \to U$ defined by $\eta(x) = \bigcap_{j = -\infty}^\infty \overline{\mathcal{P}^{-j}(\interior(R_{x_j}))}$ for all $x \in \Sigma$ and $\eta^+(x) = \bigcap_{j = 0}^\infty \overline{\sigma^{-j}(\interior(U_{x_j}))}$ for all $x \in \Sigma^+$. Define $\hat{\Sigma} = \eta^{-1}(\hat{R})$ and $\hat{\Sigma}^+ = (\eta^+)^{-1}(\hat{U})$. Then the restrictions $\eta|_{\hat{\Sigma}}: \hat{\Sigma} \to \hat{R}$ and $\eta^+|_{\hat{\Sigma}^+}: \hat{\Sigma}^+ \to \hat{U}$ are bijective and satisfy $\eta|_{\hat{\Sigma}} \circ \sigma|_{\hat{\Sigma}} = \mathcal{P}|_{\hat{R}} \circ \eta|_{\hat{\Sigma}}$ and $\eta^+|_{\hat{\Sigma}^+} \circ \sigma|_{\hat{\Sigma}^+} = \sigma|_{\hat{U}} \circ \eta^+|_{\hat{\Sigma}^+}$.
	
	Assume $\beta_0 \in (0, 1)$ is sufficiently close to $1$ so that $\eta$ and $\eta^+$ are Lipschitz continuous \cite[Lemma 2.2]{Bow73}. Let $L(\Sigma, \R)$ denote the space of Lipschitz continuous functions $f: \Sigma \to \mathbb R$ which is a Banach space with the norm $\|f\|_{\Lip} = \|f\|_\infty + \Lip(f)$ where $\Lip(f) = \sup_{x, y \in \Sigma, x \neq y} \frac{|f(x) - f(y)|}{d(x, y)}$ is the Lipschitz seminorm. We use similar notations for Lipschitz function spaces with domain $\Sigma^+$ and other codomains. For all function spaces, we suppress the codomain when it is $\R$.
	
	Since the horospherical foliations on $G$ are smooth, we conclude that $\tau$ is Lipschitz continuous on cylinders of length $1$. Then, $(\tau \circ \eta)|_{\hat{\Sigma}}$ is Lipschitz continuous and, abusing notation, we also denote its unique Lipschitz extension by $\tau: \Sigma \to \mathbb R$. Since $\tau$ is independent of past coordinates, again abusing notation, we obtain a well-defined Lipschitz continuous map $\tau: \Sigma^+ \to \mathbb R$. These are distinct from naive compositions with $\eta$ or $\eta^+$ because they may differ precisely on $x \in \Sigma$ for which $\proj_U(\eta(x)) \in \partial\mathtt{C}_U$ or $x \in \Sigma^+$ for which $\eta^+(x) \in \partial\mathtt{C}_U$ for some cylinder $\mathtt{C}_U \subset U$ with $\len(\mathtt{C}_U) = 1$. For all $k \in \mathbb N$, we define
	\begin{align*}
		\tau_k(x) &= \sum_{j = 0}^{k - 1} \tau(\sigma^j(x))
	\end{align*}
	and $\tau_0(x) = 0$ for all $x \in \Sigma^+$.
	
	\subsection{Thermodynamics}
	\label{subsec:Thermodynamics}
	\begin{definition}[Pressure]
		\label{def:Pressure}
		For all $f \in L(\Sigma)$, called the \emph{potential}, the \emph{pressure} is defined by
		\begin{align*}
			\Pr_\sigma(f) = \sup_{\nu \in \mathcal{M}^1_\sigma(\Sigma)}\left\{\int_\Sigma f \, d\nu + h_\nu(\sigma)\right\}
		\end{align*}
		where $\mathcal{M}^1_\sigma(\Sigma)$ is the set of $\sigma$-invariant Borel probability measures on $\Sigma$ and $h_\nu(\sigma)$ is the measure theoretic entropy of $\sigma$ with respect to $\nu$.
	\end{definition}
	
	For all $f \in L(\Sigma)$, there is a unique $\sigma$-invariant Borel probability measure $\nu_f$ on $\Sigma$ which attains the supremum in \cref{def:Pressure} called the \emph{$f$-equilibrium state} \cite[Theorems 2.17 and 2.20]{Bow08} and it satisfies $\nu_f(\hat{\Sigma}) = 1$ \cite[Corollary 3.2]{Che02}.
	
	Associated to $\psi$, denote the exponential growth rate
	\begin{align*}
		\delta_\psi = \limsup_{t \to +\infty} \frac{1}{t} \log\#\{[\gamma]: \psi(\lambda(\gamma)) \leq t\} \in (0, +\infty]
	\end{align*}
	where $[\gamma]$ denotes the conjugacy class of $\gamma \in \Gamma$. By \cite[Lemma 1]{Led95}, $\delta_\psi > 0$ since $\psi|_{\limitcone \setminus \{0\}} > 0$.
	In fact, $\psi$ being tangent to the concave function $\growthindicator$ implies that $\delta_\psi = 1$ \cite[Theorem 4.20]{Sam15}. Due to \cref{cor:TranslationFlowHasMarkovSection}, we can use \cite[Theorem A.2]{Car21} which states that $m_\mathsf{v}$ is a measure of maximal entropy for the translation flow which attains the maximal entropy of $\delta_\psi = 1$. We will consider in particular the probability measure $\nu_{-\delta_\psi\tau} = \nu_{-\tau}$ on $\Sigma$ which we will denote simply by $\nu_\Sigma$. We also define the probability measure $\nu_{\Sigma^+} = (\proj_{\Sigma^+})_*(\nu_\Sigma)$. Then $\nu_\Sigma(\hat{\Sigma}) = \nu_{\Sigma^+}(\hat{\Sigma}^+) = 1$, $\nu_{\Sigma^+}(\tau) = \nu_\Sigma(\tau)$, and has corresponding pressure $\Pr_\sigma(-\delta_\psi\tau) = \Pr_\sigma(-\tau) = 0$ \cite[Proposition 3.1]{BR75} (cf. \cite[Theorem 4.4]{Che02}). Following these references, the translation flow has a \emph{unique} measure of maximal entropy which can be constructed using the coding as follows. Consider the suspension space $\Sigma^\tau = (\Sigma \times \mathbb R)/\mathord{\sim}$ where $\sim$ is the equivalence relation on $\Sigma \times \mathbb R$ defined by $(x, t + \tau(x)) \sim (\sigma(x), t)$. Then we have a surjection $\Sigma^\tau \to \mathcal{X}$ defined by $(x, t) \mapsto \mathcal{W}_t(\eta(x))$. We can define the measure $\nu^\tau$ on $\Sigma^\tau$ as the product measure $\nu_\Sigma \times m^{\mathrm{Leb}}$ on $\{(x, t) \in \Sigma \times \mathbb R: 0 \leq t < \tau(x)\}$. Then the aforementioned surjection is bijective on a full measure subset and the pushforward measure $\nu^\tau$ on $\mathcal{X}$, by abuse of notation, is the unique measure of maximal entropy. Thus, $m_\mathsf{v} = \frac{\nu^\tau}{\nu_{\Sigma^+}(\tau)}$.
	
	\section{First return vector map and holonomy}
	\label{sec:FirstReturnVectorMapAndHolonomy}
	In this section we introduce two important objects, first return vector map and holonomy. We then prove some essential properties.
	
	Recall that $\Fboundary^\circ = G/P^\circ \cong K/M^\circ$ and $\Fboundary = G/P \cong K/M$, both of which have a left $K$-invariant metric $d$ induced from the one on $G$. We have a natural smooth projection $\pi^\circ: \Fboundary^\circ \to \Fboundary$. It is in fact a regular cover with deck transformation group $M/M^\circ$. Recall from \cref{subsec:ErgodicDecomposition} that $\mathcal{Y}_\Gamma$ denotes the set of $\Gamma$-minimal subsets of $\Fboundary^\circ$ and $\Lambda_0$ is a fixed set in $\mathcal{Y}_\Gamma$. Denote $N_\epsilon^+ = N^+ \cap B_\epsilon(e)$ for all $\epsilon > 0$ and define $\tilde{\Lambda}_\Gamma = \bigcup_{\Lambda \in \mathcal{Y}_\Gamma} \Lambda$. Using the open dense subset $N^+P \subset G$, we obtain the following lemma.
	
	\begin{lemma}
		\label{lem:FCirc_DiffeoTo_F_OnSmallNeighborhood}
		Let $\xi^\circ \in \Fboundary^\circ$ and $\xi = \pi^\circ(\xi^\circ) \in \Fboundary$. There exists $\epsilon_0' > 0$ such that for all $\epsilon \in (0, \epsilon_0')$, the subsets $\xi^\circ N_\epsilon^+ \subset \Fboundary^\circ$ and $\xi N_\epsilon^+ \subset \Fboundary$ are open sets and they are diffeomorphic under $\pi^\circ$.
	\end{lemma}
	
	\begin{lemma}
		\label{lem:SmallNeighborhoodOfLambdaTildeInLambda0}
		There exists $\delta_0 > 0$ such that for all $\xi \in \Lambda_0$, we have $\tilde{\Lambda}_\Gamma \cap B_{\delta_0}(\xi) \subset \Lambda_0$.
	\end{lemma}
	
	\begin{proof}
		Let $\xi \in \Lambda_0$. Note that $\Lambda$ is compact and disjoint from $\Lambda_0$ for all $\Lambda \in \mathcal{Y}_\Gamma$ with $\Lambda \neq \Lambda_0$. Also, $\#\mathcal{Y}_\Gamma =  [M : M_\Gamma] \leq [M : M^\circ]$ is finite. Thus, choosing $\delta_\xi \in \bigl(0, \frac{1}{2}\min\{d(\xi, \Lambda): \Lambda \in \mathcal{Y}_\Gamma, \Lambda \neq \Lambda_0\}\bigr)$, we have $\tilde{\Lambda}_\Gamma \cap B_{\delta_\xi}(x) \subset \Lambda_0$ for all $x \in B_{\delta_\xi}(\xi)$. The lemma now follows by compactness of $\Lambda_0$.
	\end{proof}
	
	Without loss of generality, we make the following assumptions henceforth. Replacing $\Gamma$ with an appropriate conjugate in $G$ if necessary, we assume that $e^\pm \in \limitset$. Choosing $\Lambda_0 = \overline{\Gamma eP^\circ} \in \mathcal{Y}_\Gamma$, we also assume $[e] \in \Omega_0 \subset \Gamma \backslash G$. We further assume that the Markov section $\mathcal{R}$ was constructed such that $\pi_\psi(\Gamma e M)$ is the center of the rectangle $R_1 \in \mathcal{R}$.
	
	\begin{remark}
		The above assumptions are largely to resolve ambiguities for the holonomy along closed orbits in a compatible fashion with the definition of $\Gamma^\star$ in \cref{subsec:DensityTools} so that \cref{lem:FirstReturnVectorAndHolonomyEqualsGeneralizedJordanProjectionOnFixedPoints} holds.
	\end{remark}
	
	We will now construct a section 
	\begin{align*}
		F: R \to \Omega_0 \subset \Gamma \backslash G,
	\end{align*}
	i.e, $(\pi_\psi \circ \pi \circ F)(u) = u$ for all $u \in R$ where $\pi: \Gamma \backslash G \to \Gamma \backslash G/M$ is the quotient map, such that it satisfies \cref{eqn:SectionPropertyOnUnstableLeafs,eqn:SectionPropertyOnStableLeafs}. Let $D \subset \limitset^{(2)} \times \mathbb R$ be a fundamental domain for $\mathcal{X}$ so that $R_j$ has a unique isometric lift $\tilde{R}_j \subset D$ for all $j \in \mathcal{A}$ and so that the center of $R_1$ lifts to $(e^+, e^-, 0) \in D$. Denote $\tilde{R} = \bigsqcup_{j \in \mathcal{A}} \tilde{R}_j$. For all $u \in R$, we also denote its unique lift $\tilde{u} \in \tilde{R}$. We can choose a compact lift $\tilde{D}/M \subset \limitset^{(2)} \times \LieA \cong G/M$ under the map $\pi_\psi$ by using the decomposition $\LieA = \mathbb R \mathsf{v} \oplus \ker\psi$ and ensuring that elements in $\tilde{D}/M$ have vanishing $\ker\psi$ component. Let $w_j$ be the center of $R_j$ for all $j \in \mathcal{A}$. First, recalling \cref{eqn:Omega_mProjectsToE}, we can choose $F(w_j) = [\tilde{w}_j] \in \Omega_0$ to be a preimage of $w_j$ in $\Omega_0$ under $\pi_\psi \circ \pi$ such that $\tilde{w}_j \in \tilde{D}M$ for all $j \in \mathcal{A}$. Note that $\tilde{w}_1 = e$ and $F(w_1) = [e] \in \Omega_0$. Then, we extend the section $F$ such that for all $j \in \mathcal{A}$ and $u, u' \in U_j$, we have that $F(u)$ and $F(u')$ are backwards asymptotic, i.e., $\lim_{t \to -\infty} d(F(u)a_{t\mathsf{v}}, F(u')a_{t\mathsf{v}}) = 0$. Recalling \cref{eqn:HorosphericalSubgroups,eqn:UniqueElementsInN^+AndN^-}, we must have 
	\begin{align}
		\label{eqn:SectionPropertyOnUnstableLeafs}    
		F(u') = F(u)h
	\end{align}
	for some unique $h \in N^+$. Then, we further extend the section $F$ such that for all $j \in \mathcal{A}$, $u \in U_j$, and $s, s' \in S_j$, we have that $F([u, s])$ and $F([u, s'])$ are forwards asymptotic, i.e., $\lim_{t \to +\infty} d(F([u, s])a_{t\mathsf{v}}, F([u, s'])a_{t\mathsf{v}}) = 0$. Again by \cref{eqn:HorosphericalSubgroups,eqn:UniqueElementsInN^+AndN^-}, we must have 
	\begin{align}
		\label{eqn:SectionPropertyOnStableLeafs}        
		F([u, s']) = F([u, s])n
	\end{align}
	for some unique $n \in N^-$. Note that $F(R_1) \subset \Gamma N^+N^-$. The above construction is possible due to the definitions of the foliations $W^{\mathrm{su}}$ and $W^{\mathrm{ss}}$ from \cref{eqn:StrongUnstableAndStrongStableFoliations}. Without loss of generality, we assume that the size $\hat{\delta}$ of the Markov section $\mathcal{R}$ is sufficiently small such that using compactness of $\tilde{D}M$ and \cref{lem:FCirc_DiffeoTo_F_OnSmallNeighborhood,lem:SmallNeighborhoodOfLambdaTildeInLambda0}, we can ensure that the image of $F$ lies in $\Omega_0 \subset \Gamma \backslash G$.
	
	We introduce a convenient inner product $\langle \cdot, \cdot \rangle_\psi$ on $\LieA$ such that it coincides with $\langle \cdot, \cdot \rangle$ on $\ker\psi$, the decomposition $\LieA = \mathbb R \mathsf{v} \oplus \ker\psi$ is orthogonal, and $\langle \mathsf{v}, \mathsf{v} \rangle_\psi = 1$. Denote by $\|\cdot\|_\psi$ the corresponding norm and $\proj_{\ker\psi}: \LieA \to \ker\psi$ the orthogonal projection with respect to $\langle \cdot, \cdot \rangle_\psi$.
	
	\begin{definition}[First return vector map, Holonomy]
		The \emph{first return $\LieA$-vector map} and \emph{holonomy} are maps $\mathsf{K}: R \to \LieA$ and $\vartheta: R \to M_\Gamma$ respectively that associate to each $u \in R$ the unique elements $\mathsf{K}(u) \in \LieA$ and $\vartheta(u) \in M_\Gamma$ which satisfy
		\begin{align*}
			F(\mathcal{P}(u)) = F(u) a_{\mathsf{K}(u)}\vartheta(u).
		\end{align*}
		The \emph{first return $\ker\psi$-vector map} is the map $\widehat{\mathsf{K}} = \proj_{\ker\psi} \circ \mathsf{K}: R \to \ker\psi$. We drop $\LieA$- and $\ker\psi$- in the above terminology when we refer to either maps.
	\end{definition}
	
	\begin{remark}
		A priori, it is clear that there exists a unique $\vartheta(u) \in M$ satisfying the above equation. We can ensure the stronger condition that $\vartheta(u) \in M_\Gamma$ because the image of $F$ lies in $\Omega_0$. Thus, holonomy is well-defined.
	\end{remark}
	
	From definitions, we then have
	\begin{align}
		\label{eqn:K_Equals_tau+K-hat}
		\mathsf{K}(u) = \tau(u)\mathsf{v} + \widehat{\mathsf{K}}(u) \qquad \text{for all $u \in R$}.
	\end{align}
	
	\begin{lemma}
		\label{lem:ConstantOnStrongStableLeaves}
		The maps $\tau$, $\widehat{\mathsf{K}}$, $\mathsf{K}$, and $\vartheta$ are all constant on $[u, S_j]$ for all $u \in U_j$ and $j \in \mathcal{A}$.
	\end{lemma}
	
	\begin{proof}
		The lemma for $\tau$ was already observed in \cref{subsec:MetricAnosovFlowsAndMarkovSections}. Thus, in light of \cref{eqn:K_Equals_tau+K-hat}, it suffices to prove the lemma only for $\mathsf{K}$ and $\vartheta$. Let $j \in \mathcal{A}$, $u \in U_j$, $s \in S_j$, and $u' = [u, s]$. We have $F(u') = F(u)n$ for some $n \in N^-$. From definitions, we have $F(\mathcal{P}(u)) = F(u)a_{\mathsf{K}(u)}\vartheta(u)$ and $F(\mathcal{P}(u')) = F(u')a_{\mathsf{K}(u')}\vartheta(u') = F(u)na_{\mathsf{K}(u')}\vartheta(u')$. We assume that $\epsilon_0$ was chosen sufficiently small when using the locally isometric covering map $G \to \Gamma \backslash G$. Using left $G$-invariance and right $K$-invariance of the distance function $d$ on $G$, we have
		\begin{align*}
			d(F(\mathcal{P}(u))a_{t\mathsf{v}}, F(\mathcal{P}(u'))a_{t\mathsf{v}}) ={}& d(F(u)a_{\mathsf{K}(u) + t\mathsf{v}}\vartheta(u), F(u)n a_{\mathsf{K}(u') + t\mathsf{v}}\vartheta(u')) \\
			={}& d(a_{\mathsf{K}(u) - \mathsf{K}(u')}\vartheta(u)\vartheta(u')^{-1}, a_{-(\mathsf{K}(u') + t\mathsf{v})}n a_{\mathsf{K}(u') + t\mathsf{v}}) \\
			\geq{}& d(a_{\mathsf{K}(u) - \mathsf{K}(u')}\vartheta(u)\vartheta(u')^{-1}, e) \\
			&{}- d(e, a_{-(\mathsf{K}(u') + t\mathsf{v})}n a_{\mathsf{K}(u') + t\mathsf{v}})
		\end{align*}
		for all $t \geq 0$. As $t \to +\infty$, both the left hand side and the last term vanish by definitions. Thus, $d(a_{\mathsf{K}(u) - \mathsf{K}(u')}\vartheta(u)\vartheta(u')^{-1}, e) = 0$ which implies $\mathsf{K}(u') = \mathsf{K}(u)$ and $\vartheta(u') = \vartheta(u)$.
	\end{proof}
	
	As with $\tau$ in \cref{subsec:SymbolicDynamics}, since the horospherical foliations on $G$ are smooth, we obtain Lipschitz continuous maps corresponding to $\widehat{\mathsf{K}}$, $\mathsf{K}$, and $\vartheta$ which, abusing notation, we also denote by $\widehat{\mathsf{K}}: \Sigma \to \ker\psi$, $\mathsf{K}: \Sigma \to \LieA$, and $\vartheta: \Sigma \to M_\Gamma$. Now, \cref{lem:ConstantOnStrongStableLeaves} guarantees that $\widehat{\mathsf{K}}$, $\mathsf{K}$, and $\vartheta$ are independent of past coordinates and hence, again abusing notation, we obtain well-defined Lipschitz continuous maps $\widehat{\mathsf{K}}: \Sigma^+ \to \ker\psi$, $\mathsf{K}: \Sigma^+ \to \LieA$, and $\vartheta: \Sigma^+ \to M_\Gamma$. Note that on each cylinder of length $1$, the holonomy $\vartheta$ takes values in a single corresponding connected component of $M_\Gamma$. For all $k \in \mathbb N$, we define
	\begin{align*}
		\widehat{\mathsf{K}}_k(x) &= \sum_{j = 0}^{k - 1} \widehat{\mathsf{K}}(\sigma^j(x)), & \mathsf{K}_k(x) &= \sum_{j = 0}^{k - 1} \mathsf{K}(\sigma^j(x)), & \vartheta^k(x) &= \prod_{j = 0}^{k - 1} \vartheta(\sigma^j(x))
	\end{align*}
	and $\mathsf{K}_0(x) = \widehat{\mathsf{K}}_0(x) = 0 \in \LieA$ and $\vartheta^0(x) = e \in M_\Gamma$ for all $x \in \Sigma^+$, where the terms in the product are in \emph{ascending} order from left to right. Note that for all $k \in \mathbb Z_{\geq 0}$ we then have
	\begin{align*}
		F(\mathcal{W}_{\tau_k(x)}(\eta^+(x))) = F(\eta^+(x))a_{\mathsf{K}_k(x)}\vartheta^k(x) \qquad \text{for all $x \in \Sigma^+$}.
	\end{align*}
	
	Define $\varrho: M_\Gamma \to \U(L^2(M_\Gamma, \mathbb C))$ to be the unitary left regular representation, i.e., $\varrho(\tilde{m})(\phi)(m) = \phi(\tilde{m}^{-1}m)$ for all $m \in M_\Gamma$, $\phi \in L^2(M_\Gamma, \mathbb C)$, and $\tilde{m} \in M_\Gamma$. Let $\widehat{M}_\Gamma$ denote the unitary dual of $M_\Gamma$. We denote the trivial irreducible representation by $1 \in \widehat{M}_\Gamma$. By the Peter--Weyl theorem, we obtain an orthogonal Hilbert space decomposition
	\begin{align*}
		L^2(M_\Gamma, \mathbb C) = \widehat{\bigoplus}_{\mu \in \widehat{M}_\Gamma} V_\mu^{\oplus \dim(\mu)}
	\end{align*}
	corresponding to the decomposition $\varrho = \widehat{\bigoplus}_{\mu \in \widehat{M}_\Gamma} \mu^{\oplus \dim(\mu)}$.
	
	\section{Spectra of the transfer operators with holonomy}
	\label{sec:SpectraOfTheTransferOperatorsWithHolonomy}
	In this section we define the transfer operators with holonomy and characterize their spectra.
	
	\subsection{Transfer operators}
	\label{subsec:TransferOperators}
	\begin{definition}[Transfer operator with holonomy]
		For all $a \in \mathbb R$, $v \in \LieA$, and $\mu \in \widehat{M}_\Gamma$, the \emph{transfer operator with holonomy} $\mathcal{L}_{a\tau, v, \mu}: C(\Sigma^+, V_\mu) \to C(\Sigma^+, V_\mu)$ is defined by
		\begin{align*}
			\mathcal{L}_{a\tau, v, \mu}(H)(x) = \sum_{x' \in \sigma^{-1}(x)} e^{a\tau(x') + i\langle v, \mathsf{K}(x') \rangle_\psi} \mu(\vartheta(x')^{-1}) H(x')
		\end{align*}
		for all $x \in \Sigma^+$ and $H \in C(\Sigma^+, V_\mu)$.
	\end{definition}
	
	For all $a \in \mathbb R$, we denote $\mathcal{L}_{a\tau} = \mathcal{L}_{a\tau, 0, 1}$ and simply call it the \emph{transfer operator}. We recall the Ruelle--Perron--Frobenius (RPF) theorem along with the theory of Gibbs measures \cite{Bow08,PP90}.
	
	\begin{theorem}
		\label{thm:RPFonU}
		For all $a \in \mathbb R$, the operator $\mathcal{L}_{a\tau}: C(\Sigma^+, \mathbb C) \to C(\Sigma^+, \mathbb C)$ and its dual $\mathcal{L}_{a\tau}^*: C(\Sigma^+, \mathbb C)^* \to C(\Sigma^+, \mathbb C)^*$ has eigenvectors with the following properties. There exist a unique positive function $h \in L(\Sigma^+)$ and a unique Borel probability measure $\nu$ on $\Sigma^+$ such that:
		\begin{enumerate}
			\item	$\mathcal{L}_{a\tau}(h) = e^{\Pr_\sigma(a\tau)}h$;
			\item	$\mathcal{L}_{a\tau}^*(\nu) = e^{\Pr_\sigma(a\tau)}\nu$;
			\item	the eigenvalue $e^{\Pr_\sigma(a\tau)}$ is maximal simple and the rest of the spectrum of $\mathcal{L}_{a\tau}|_{L(\Sigma^+, \mathbb C)}$ is contained in a disk of radius strictly less than $e^{\Pr_\sigma(a\tau)}$;
			\item	$\nu(h) = 1$ and the Borel probability measure $\mu$ defined by $d\mu = h \, d\nu$ is $\sigma$-invariant and is the projection of the $a\tau$-equilibrium state to $\Sigma^+$, i.e., $\mu = (\proj_{\Sigma^+})_*(\nu_{a\tau})$.
		\end{enumerate}
	\end{theorem}
	
	In light of \cref{thm:RPFonU}, it is convenient to normalize the transfer operators defined above. Let $a \in \mathbb R$. Define $\kappa_a = e^{\Pr_\sigma(-(1 + a)\tau)}$ which is the maximal simple eigenvalue of $\mathcal{L}_{-(1 + a)\tau}$ by \cref{thm:RPFonU} and recall from \cref{subsec:Thermodynamics} that $\kappa_0 = 1$. Define the eigenvectors, the unique positive function $h_a \in L(\Sigma^+)$ and the unique probability measure $\nu_a$ on $\Sigma^+$ with $\nu_a(h_a) = 1$ such that
	\begin{align*}
		\mathcal{L}_{-(1 + a)\tau}(h_a) &= \kappa_a h_a, & \mathcal{L}_{-(1 + a)\tau}^*(\nu_a) &= \kappa_a \nu_a
	\end{align*}
	provided by \cref{thm:RPFonU}. Note that $d\nu_{\Sigma^+} = h_0 \, d\nu_0$. We define
	\begin{align*}
		\tau^{(a)} = -(1 + a)\tau + \log(h_0) - \log(h_0 \circ \sigma).
	\end{align*}
	For all $k \in \mathbb Z_{\geq 0}$, we define the notation $\tau_k^{(a)}$ similar to $\tau_k$ in \cref{subsec:SymbolicDynamics}.
	
	We now normalize the transfer operators. Let $a \in \mathbb R$, $v \in \LieA$, and $\mu \in \widehat{M}_\Gamma$. We define $\mathcal{L}_{a, v, \mu}: C(\Sigma^+, V_\mu) \to C(\Sigma^+, V_\mu)$ by
	\begin{align*}
		\mathcal{L}_{a, v, \mu}(H)(x) &= \sum_{x' \in \sigma^{-1}(x)} e^{\tau^{(a)}(x') + i\langle v, \mathsf{K}(x') \rangle_\psi} \mu(\vartheta(x')^{-1}) H(x')
	\end{align*}
	for all $x \in \Sigma^+$ and $H \in C(\Sigma^+, V_\mu)$. For all $k \in \mathbb N$, its $k$\textsuperscript{th} iteration is
	\begin{align}
		\label{eqn:k^thIterationOfCongruenceTransferOperatorOfType_rho}
		\mathcal{L}_{a, v, \mu}^k(H)(x) &= \sum_{x' \in \sigma^{-k}(x)} e^{\tau_k^{(a)}(x') + i\langle v, \mathsf{K}_k(x') \rangle_\psi} \mu(\vartheta^k(x')^{-1}) H(x')
	\end{align}
	for all $x \in \Sigma^+$ and $H \in C(\Sigma^+, V_\mu)$. For all $a \in \mathbb R$, $v \in \LieA$, and $\mu \in \widehat{M}_\Gamma$, we denote $\mathcal{L}_a = \mathcal{L}_{a, 0, 1}$ and $\mathcal{L}_{v, \mu} = \mathcal{L}_{0, v, \mu}$. With this normalization, we have $\mathcal{L}_0(\chi_{\Sigma^+}) = \chi_{\Sigma^+}$ and $\mathcal{L}_0^*(\nu_{\Sigma^+}) = \nu_{\Sigma^+}$.
	
	Throughout the paper, for all $a \in \mathbb R$, $v \in \LieA$, and $\mu \in \widehat{M}_\Gamma$, we abuse notation and denote $\mathcal{L}_{a, v, \mu} = \mathcal{L}_{a, v, \mu}|_{L(\Sigma^+, V_\mu)}: L(\Sigma^+, V_\mu) \to L(\Sigma^+, V_\mu)$ and consequently their spectra shall always mean the spectra of the operators restricted to $L(\Sigma^+, V_\mu)$.
	
	\subsection{Density tools}
	\label{subsec:DensityTools}
	The following is the first required density proposition.
	
	\begin{proposition}
		\label{pro:DenseOrbitInSigma+TimesM_Gamma}
		There exists $x \in \Sigma^+$ such that $\{(\sigma^k(x), \vartheta^k(x)): k \in \mathbb N\} \subset \Sigma^+ \times M_\Gamma$ is dense. 
	\end{proposition}
	
	\begin{proof}
		Recall the notation in \cref{subsec:ErgodicDecomposition}. We will also freely use \cref{eqn:PsiConjugation}. Let $\pi: \mathcal{Z} \to \mathcal{X}$ denote the natural projection map. Then $\pi \circ \mathcal{W}_t = \mathcal{W}_t \circ \pi$ for all $t \in \mathbb R$. \Cref{pro:R-ergodicity} provides a dense $\R_{> 0}$-orbit $\mathcal{W}_{\R_{> 0}}(\tilde{u}) \subset \mathcal{Z}_{[e]}$ for some $\tilde{u} \in \mathcal{Z}_{[e]}$. Set $u = \pi(\tilde{u})$. Now, $\Omega_\Psi := \Gamma N^+P \cap \Omega_0 \subset \Omega_0$ is a $AM_\Gamma$-invariant open dense subset on which $\Psi$ is defined. We may assume that the section $F$ has been chosen such that $F(\mathcal{W}_{\R_{> 0}}(u) \cap R) \cap \Omega_\Psi \ne \varnothing$. Indeed, since $\mathcal{W}_{\R_{> 0}}(u) \subset \mathcal{X}$ is dense, there exists a lift of $\mathcal{W}_{\R_{> 0}}(u)$ which intersects $\Omega_\Psi$. By $A$-invariance, that entire lift lies in $\Omega_\Psi$. In particular, there is some lift of $\mathcal{W}_{\R_{> 0}}(u) \cap R$ that lies in $\Omega_\Psi$, and we can obtain the desired section by multiplying by some element in $AM_\Gamma$.  
		Moreover, after truncating some initial part of the $\R_{>0}$-orbit, we may assume that $u \in R$ and $F(u) \in \Omega_\Psi$. Fix $y_j \in \mathcal{A}$ for $-k \le j \le k$ and an open subset $V \subset M_\Gamma$ and choose an open subset $U \subset R$ such that $u \notin U \subset \bigcap_{-k \le j \le k}\mathcal{P}^{-j}(\interior(R_{y_j}))$. By the above, $(\Psi \circ F)(\mathcal{W}_{\R_{> 0}}(u))$ is well-defined and we let $\tilde{m} \in M$ such that $(\Psi \circ F)(u) = \tilde{u}\tilde{m}$. Since any sufficiently small $A$-translation of $F(U)V\tilde{m}^{-1}$ will contain an open set and $\Psi$ is an open map, any small $\R$-translation of $\Psi(F(U)V\tilde{m}^{-1} \cap \Omega_\Psi) = \Psi(F(U)\cap \Omega_\Psi)V\tilde{m}^{-1}$ will contain an open subset of $\mathcal{Z}_{[e]}$, so by density of $\mathcal{W}_{\R_{> 0}}(\tilde{u})$, there exists $t_0>0$ such that $\mathcal{W}_{t_0}(\tilde{u}) \in \Psi(F(U)\cap \Omega_\Psi)V\tilde{m}^{-1}$. Applying the map $\pi$, we see that there exists $k \in \N$ such that $\mathcal{P}^k(u) = \mathcal{W}_{t_0}(u) \in U$. Applying $\Psi$ to the defining equation for holonomy and rearranging, we get
		\begin{equation*}
			(\Psi \circ F)(\mathcal{P}^k(u))\vartheta^k(u)^{-1}\tilde{m}^{-1} = \mathcal{W}_{t_0}(\tilde{u}) \in (\Psi \circ F)(U)V\tilde{m}^{-1},
		\end{equation*}
		which implies that $\vartheta^k(u) \in V$. Choosing a sequence in $\eta^{-1}(u)$ and taking future coordinates provides the required $x \in \Sigma^+$.
	\end{proof}
	
	Our goal is now to prove \cref{pro:SubgroupGeneratedBy_K_AndHolonomyDenseInLieATimesM_Gamma}. We recall definitions from \cite[Subsection 3.2]{LO20a} with slight generalizations. Denote
	\begin{align*}
		G^\star = \{g \in G: \text{there exists $h \in N^+N^-$ with $g \in h\interior(A^+)Mh^{-1}$}\} \subset G.
	\end{align*}
	Let $\lambda^{AM}: G^\star \to \interior(A^+)M$ denote the \emph{generalized Jordan projection}, i.e., $\lambda^{AM}(g) \in \interior(A^+)M$ is the unique element such that $g = h\lambda^{AM}(g)h^{-1}$ for some unique $h \in N^+N^-$, for all $g \in G^\star$. Define the components of the generalized Jordan projection $\lambda^A: G^\star \to \interior(A^+)$ and $\lambda^M: G^\star \to M$ so that $\lambda^{AM}(g) = \lambda^A(g) \cdot \lambda^M(g) = a_{\lambda(g)} \cdot \lambda^M(g) \in \interior(A^+)M$ for all $g \in G^\star$.
	
	Let $\Gamma_0 < G$ be any Zariski dense discrete sub\emph{semigroup}. Denote $\Gamma_0^\star = \Gamma_0 \cap G^\star \subset G$ which we note is also Zariski dense.
	
	\begin{definition}[Generalized length spectrum]
		The subset denoted by $\lambda^{AM}(\Gamma_0) = \{\lambda^{AM}(\gamma): \gamma \in \Gamma_0^\star\} \subset \interior(A^+)M$ is called the \emph{generalized length spectrum} of $\Gamma_0$.
	\end{definition}
	
	It is clear from the proofs that \cite[Lemmas 3.3 and 3.4]{LO20a} also hold for subsemigroups. Together with \cite[Theorem 1.9]{GR07}, we have the following proposition.
	
	\begin{proposition}
		\label{pro:SubgroupOfGeneralizedLengthSpectrumOfGamma_0DenseInAM_Gamma_0}
		We have $\overline{\langle \lambda^{AM}(\Gamma_0) \rangle} = AM_{\Gamma_0}$ where $M_{\Gamma_0} < M$ is a closed normal subgroup containing $M^\circ$ and $M_{\Gamma_0}/M^\circ \cong (\mathbb Z/2 \mathbb Z)^p$ for some integer $0 \leq p \leq \rank$.
	\end{proposition}
	
	We would like to use \cref{pro:SubgroupOfGeneralizedLengthSpectrumOfGamma_0DenseInAM_Gamma_0} but first we need \cref{lem:FirstReturnVectorAndHolonomyEqualsGeneralizedJordanProjectionOnFixedPoints} which relates the first return vector map and the holonomy with the $A$ and $M$-components of the generalized Jordan projection.
	
	There is a natural bijection between the set of periodic orbits with multiplicity for the Gromov geodesic flow on $\mathcal{G}$ and the set of conjugacy classes of loxodromic elements in $\Gamma$. By \cref{thm:TranslationFlowConjugateToGromovGeodesicFlow}, the same is true for the translation flow on $\mathcal{X}$. In fact, we can parametrize a periodic orbit as $\mathcal{W}_\bullet(u): \mathbb R \to \mathcal{X}$ with the data of a preferred starting rectangle $\mathcal{W}_0(u) = u \in R$ and multiplicity, and then lift it to a \emph{unique} parametrized orbit $\mathcal{W}_\bullet(\tilde{u}): \mathbb R \to \limitset^{(2)} \times \mathbb R$ with the data of a preferred starting rectangle $\mathcal{W}_0(\tilde{u}) = \tilde{u} \in \tilde{R}$ and multiplicity, which in turn corresponds to a \emph{unique} loxodromic element $\gamma \in \Gamma$ with $\gamma^\pm = \tilde{u}^\pm$. This gives a map
	\begin{align*}
		\{\mathcal{W}_\bullet(u): \mathcal{W}_\bullet(u) \text{ is a periodic orbit in } \mathcal{X}, u \in R\} \times \mathbb N \to \Gamma.
	\end{align*}
	Its image of course surjects onto the set of conjugacy classes of loxodromic elements in $\Gamma$. There is also a natural surjection from the set of periodic admissible sequences with the data of multiplicity to the above domain. Precisely, we have a surjection
	\begin{align*}
		\bigcup_{k \in \mathbb N} \Fix(\sigma^k) \times \mathbb N \to \{\mathcal{W}_\bullet(u): \mathcal{W}_\bullet(u) \text{ is a periodic orbit in } \mathcal{X}, u \in R\} \times \mathbb N
	\end{align*}
	where we denote $\Fix(\sigma^k) = \{x \in \Sigma^+: \sigma^k(x) = x\} \subset \Sigma^+$ for all $k \in \mathbb N$. Composing the above two maps gives a map
	\begin{align*}
		\bigcup_{k \in \mathbb N} \Fix(\sigma^k) \times \mathbb N \to \Gamma.
	\end{align*}
	Denote the image of any $(x, j) \in \bigcup_{k \in \mathbb N} \Fix(\sigma^k) \times \mathbb N$ under this map by $\gamma_{x, j} \in \Gamma$.
	
	\begin{lemma}
		\label{lem:FirstReturnVectorAndHolonomyEqualsGeneralizedJordanProjectionOnFixedPoints}
		Let $x \in \Fix(\sigma^k)$ with multiplicity $j \in \mathbb N$ corresponding to some $k \in \mathbb N$. Let $\gamma = \gamma_{x, j}$. Then, we have
		\begin{align*}
			\mathsf{K}_k(x) = \lambda(\gamma).
		\end{align*}
		Moreover, if $x_0 = 1 \in \mathcal{A}$, then $\gamma \in \Gamma^\star$ and we also have
		\begin{align*}
			\vartheta^k(x) = \lambda^M(\gamma).
		\end{align*}
	\end{lemma}
	
	\begin{proof}
		Let $x \in \Fix(\sigma^k)$ and $\gamma \in \Gamma$ be as in the lemma. Let $\overline{x} \in \Sigma$ be the bi-infinite extension of $x$, $u = \eta(\overline{x}) \in R_{x_0}$, and $\tilde{u} \in \tilde{R}_{x_0}$ be its lift. Note that $\tilde{u}^\pm \in \limitset$ are the attracting and repelling fixed points of $\gamma$ and $\mathcal{W}_{\tau_k(x)}(\tilde{u}) = \gamma \tilde{u}$. From definitions, we have
		\begin{align*}
			F(u)a_{\mathsf{K}_k(x)}\vartheta^k(x) = F(\mathcal{W}_{\tau_k(x)}(u)) = F(u).
		\end{align*}
		Due to the Hopf parametrization, we can denote $F(u) = \Gamma g$ for some $g \in G$ such that $g^\pm = \tilde{u}^\pm$. Now, the above equation implies $ga_{\mathsf{K}_k(x)}\vartheta^k(x) = \gamma' g$ for some $\gamma' \in \Gamma$. Consider the natural left $\Gamma$-equivariant projection $\pi: G \to \limitset^{(2)} \times \mathbb R$ defined by $\pi(g) = \pi_\psi(gM)$ for all $g \in G$. We have
		\begin{align*}
			\gamma' \tilde{u} = \gamma'\pi(g) = \pi(\gamma' g) = \pi(ga_{\mathsf{K}_k(x)}\vartheta^k(x)) = \mathcal{W}_{\tau_k(x)}(\tilde{u}) = \gamma \tilde{u}.
		\end{align*}
		Therefore, $\gamma' = \gamma$ which implies $ga_{\mathsf{K}_k(x)}\vartheta^k(x)g^{-1} = \gamma$. The first part of the lemma follows by taking Jordan projections. If $x_0 = 1 \in \mathcal{A}$, then $u \in R_1$ and by construction of $F$, we can take $g \in N^+N^-$. We conclude that $\gamma \in \Gamma^\star$ and the second part of the lemma follows by taking generalized Jordan projections.
	\end{proof}
	
	The following is the second required density proposition.
	
	\begin{proposition}
		\label{pro:SubgroupGeneratedBy_K_AndHolonomyDenseInLieATimesM_Gamma}
		The subgroup $\langle\{(\mathsf{K}_k(x), \vartheta^k(x)): x \in \Fix(\sigma^k), k \in \mathbb N\}\rangle < \LieA \times M_\Gamma$ is dense.
	\end{proposition}
	
	\begin{proof}
		Let $w_1 \in R_1$ be the center. Recall that $\tilde{w}_1 \in \tilde{R}_1$ has forward and backward limit points $\tilde{w}_1^\pm = e^\pm \in \limitset \subset \Fboundary$. By definition of rectangles, there exists $\epsilon' > 0$ such that $W_{\epsilon'}^{\mathrm{su}}(\tilde{w}_1) \subset \tilde{U}_1$ and $W_{\epsilon'}^{\mathrm{ss}}(\tilde{w}_1) \subset \tilde{S}_1$. Then there exists $\epsilon > 0$ such that
		\begin{align*}
			B_\epsilon(e^+) \cap \limitset &\subset \{u^+: u \in W_{\epsilon'}^{\mathrm{su}}(\tilde{w}_1)\}, & B_\epsilon(e^-) \cap \limitset &\subset \{u^-: u \in W_{\epsilon'}^{\mathrm{ss}}(\tilde{w}_1)\}.
		\end{align*}
		Let $\gamma_1 \in \Gamma$ be an element whose attracting and repelling fixed points are $x_1^\pm = e^\pm \in \limitset$. By \cite[Lemma 4.3]{Ben97} and its proof, we can choose $\gamma_2 \in \Gamma$ with attracting and repelling fixed points $x_2^\pm \in \limitset$ such that $d\bigl(x_2^\pm, x_1^\pm\bigr) < \frac{\epsilon}{2}$ and $\bigl\{\gamma_1^k, \gamma_2^k\bigr\}$ generate a Zariski dense Schottky subsemigroup $\Gamma_0 \subset \Gamma < G$ for all sufficiently large $k \in \mathbb N$. Recalling the definition of Schottky subsemigroups, for sufficiently large $k \in \mathbb N$, we can additionally ensure that for all $\gamma \in \Gamma_0$, we have $d\bigl(x^\pm, x_1^\pm\bigr) < \epsilon$ where $x^\pm \in \limitset$ is the attracting and repelling fixed points of $\gamma$. We now fix such a Zariski dense Schottky subsemigroup $\Gamma_0 \subset \Gamma < G$. By \cref{pro:SubgroupOfGeneralizedLengthSpectrumOfGamma_0DenseInAM_Gamma_0}, we have $\overline{\langle \lambda^{AM}(\Gamma_0) \rangle} = AM_{\Gamma_0}$ which contains $AM^\circ$. Now, by construction, any element of $\Gamma_0$ corresponds to an orbit $\mathcal{W}_\bullet(\tilde{u})$ with $\tilde{u} \in \tilde{R}_1$. In other words
		\begin{align*}
			\Gamma_0 \subset \{\gamma_{x, j}: x \in \Fix(\sigma^k), x_0 = 1, k \in \mathbb N\}.
		\end{align*}
		Thus, we can use \cref{lem:FirstReturnVectorAndHolonomyEqualsGeneralizedJordanProjectionOnFixedPoints} to conclude that
		\begin{align*}
			\lambda^{AM}(\Gamma_0) \subset \{a_{\mathsf{K}_k(x)}\vartheta^k(x): x \in \Fix(\sigma^k), x_0 = 1, k \in \mathbb N\}.
		\end{align*}
		The above two properties combine to give
		\begin{align*}
			\LieA \times M^\circ < \overline{\langle\{(\mathsf{K}_k(x), \vartheta^k(x)): x \in \Fix(\sigma^k), k \in \mathbb N\}\rangle} < \LieA \times M_\Gamma.
		\end{align*}
		To finish proving the proposition, it suffices to show that for any connected component $mM^\circ \subset M_\Gamma$, there exists $x \in \Fix(\sigma^k)$ for some $k \in \mathbb N$ such that $\vartheta^k(x) \in mM^\circ$. Applying \cref{pro:DenseOrbitInSigma+TimesM_Gamma}, there exist $x' \in \Sigma^+$ and $k \in \mathbb N$ such that $\sigma^k(x') \in \mathtt{C}[x_0']$, i.e., $x_k' = x_0'$ and $\vartheta^k(x') \in mM^\circ$. Let $\hat{x}' = (x_0', x_1', \dotsc, x_{k - 1}', x_0')$ and note that $x' \in \mathtt{C}[\hat{x}']$. Then, we can extend $\hat{x}'$ in an admissible and periodic fashion to obtain $x \in \Fix(\sigma^k) \cap \mathtt{C}[\hat{x}']$. Since $\vartheta^k|_{\mathtt{C}[\hat{x}']}$ takes values in a fixed connected component of $M_\Gamma$, we also have $\vartheta^k(x) \in mM^\circ$ as desired.
	\end{proof}
	
	\subsection{Spectra}
	A function in $C(\Sigma^+, 2\pi\Z)$ is called a \emph{lattice function} and two functions $f,g \in C(\Sigma^+, \C)$ are called \emph{cohomologous} if there exists $\omega \in C(\Sigma^+, \C)$ such that $f-g = \omega - \omega \circ \sigma$. The following lemma is from \cite[Proposition 2]{Pol84}. 
	
	\begin{lemma}
		\label{lem:CohomologousToALatticeFunction}
		Let $\theta \in \mathbb R$ and consider a general transfer operator $\mathcal{L}_\phi: L(\Sigma^+, \mathbb C) \to L(\Sigma^+, \mathbb C)$, corresponding to $\phi = \phi_\Re + i \phi_\Im$ where $\phi_\Re, \phi_\Im \in L(\Sigma^+)$, defined by
		\begin{align*}
			\mathcal{L}_\phi(h)(x) = \sum_{x' \in \sigma^{-1}(x)} e^{\phi(x')} h(x')
		\end{align*}
		for all $x \in \Sigma^+$ and $h \in L(\Sigma^+, \mathbb C)$. Then, $\mathcal{L}_\phi$ has an eigenvalue $e^{i\theta + \Pr_\sigma(\phi_\Re)}$ if and only if $\phi_\Im - \theta$ is cohomologous to a lattice function.
	\end{lemma}
	
	The following proposition summarizes \cite[Theorems 2.2, 4.5, and 10.2]{PP90}.
	
	\begin{proposition}
		\label{pro:PreliminaryTransferOperatorSpectrum}
		For all $(v, \mu) \in \LieA \times \widehat{M}_\Gamma$, we have:
		\begin{enumerate}
			\item\label{itm:SpectralRadiusAtMost1} the spectral radius of $\mathcal{L}_{v, \mu}$ is at most $1$;
			\item\label{itm:Mod1EigenvalueIfSpectralRadius1} if the spectral radius of $\mathcal{L}_{v, \mu}$ is $1$, then $\mathcal{L}_{v, \mu}$ has a maximal simple eigenvalue of modulus $1$ and the rest of the spectrum is contained in a disk of radius strictly less than $1$.
		\end{enumerate}
	\end{proposition}
	
	\begin{theorem}
		\label{thm:SpectralBound}
		Let $(v, \mu) \in \LieA \times \widehat{M}_\Gamma$ with $(v, \mu) \neq (0, 1)$. If $\mathcal{L}_{v, \mu}$ has an eigenvalue of the form $e^{i\theta\pi}$ for some $\theta \in \mathbb R$, then $\theta \notin \mathbb Q$.
	\end{theorem}
	
	\begin{proof}
		We first fix some sequences for later. By \cref{lem:FirstReturnVectorAndHolonomyEqualsGeneralizedJordanProjectionOnFixedPoints,pro:SubgroupGeneratedBy_K_AndHolonomyDenseInLieATimesM_Gamma} and their proofs, there exist $y_l \in \Fix(\sigma^{n_l})$ with $y_{l, 0} = 1 \in \mathcal{A}$ corresponding to some $n_l \in \mathbb N$ for all integers $1 \leq l \leq l_0$ for some $l_0 \in \mathbb N$ such that $\{\vartheta^{n_l}(y_l)M^\circ\}_{l = 1}^{l_0} \subset M_\Gamma/M^\circ$ is a finite generating set. Without loss of generality, we can assume that the size $\hat{\delta}$ is sufficiently small and the Markov section $\mathcal{R}$ is arranged such that $n_l = 0 \pmod{2}$ for all $1 \leq l \leq l_0$ because if necessary, we could have inserted an intermediate rectangle which is simply the image of an existing rectangle under the translation flow for a sufficiently small time so that the primitive periods of $y_l$ are $0 \pmod{2}$ for all $1 \leq l \leq l_0$.
		
		Let $(v, \mu) \in \LieA \times \widehat{M}_\Gamma$ with $(v, \mu) \neq (0, 1)$. Suppose $\mathcal{L}_{v, \mu}$ has an eigenvalue of the form $e^{i\theta\pi}$ for some $\theta \in [0, 2)$. We will first show that $\dim(\mu) = 1$. By definition, there exists a nontrivial $H \in L(\Sigma^+, V_\mu)$ such that
		\begin{align*}
			\mathcal{L}_{v, \mu}(H) = e^{i\theta\pi}H.
		\end{align*}
		Using arguments as in the proof of \cite[Proposition 4.4]{PP90}, we conclude that
		\begin{align*}
			e^{i\langle v, \mathsf{K}(x) \rangle_\psi}\mu(\vartheta(x)^{-1})H(x) = e^{i\theta\pi}H(\sigma(x))
		\end{align*}
		for all $x \in \Sigma^+$. Define the continuous function $f: \Sigma^+ \times M_\Gamma \to V_\mu$ by $f(x, m) = \mu(m)H(x)$ for all $(x, m) \in \Sigma^+ \times M_\Gamma$. Then using the above identity, we have
		\begin{align*}
			f(\sigma(x), m\vartheta(x)) &= \mu(m)\mu(\vartheta(x))H(\sigma(x)) \\
			&= e^{i(\langle v, \mathsf{K}(x) \rangle_\psi - \theta\pi)}\mu(m)H(x) \\
			&= e^{i(\langle v, \mathsf{K}(x) \rangle_\psi - \theta\pi)}f(x, m)
		\end{align*}
		for all $(x, m) \in \Sigma^+ \times M_\Gamma$. By iterating, $f(\sigma^k(x), \vartheta^k(x)) \in \{e^{i\theta'}H(x): \theta' \in \mathbb R\}$ for all $x \in \Sigma^+$ and $k \in \mathbb N$. Now by \cref{pro:DenseOrbitInSigma+TimesM_Gamma}, we can fix $x \in \Sigma^+$ such that $\{(\sigma^k(x), \vartheta^k(x))\}_{k \in \mathbb N} \subset \Sigma^+ \times M_\Gamma$ is dense. Without loss of generality, we can assume that $H(x) \neq 0$. By continuity, the image of $f$ is contained in $\{e^{i\theta'}H(x): \theta' \in \mathbb R\}$ and hence $\mu|_{\mathbb C H(x)}$ is a nontrivial subrepresentation. By irreducibility of $\mu$, we must have $\dim(\mu) = 1$. We now view it as a homomorphism $\mu: M_\Gamma \to \U(1)$.
		
		To obtain a contradiction, suppose $\theta \in \mathbb Q$. Let $\tilde{\mu}: M_\Gamma \to \R$ be any function such that $\mu(m) = e^{i\tilde{\mu}(m)}$ for all $m \in M_\Gamma$ and $\tilde{\mu} \circ \vartheta \in C(\Sigma^+)$. Then, $\sum_{k = 0}^{n - 1} \tilde{\mu}(\vartheta(\sigma^k(x))) = \tilde{\mu}(\vartheta^n(x)) \pmod{2\pi}$ and so by \cref{lem:CohomologousToALatticeFunction}, we have
		\begin{align}
			\label{eqn:Periodic_x_MultipleOf2pi}
			\langle v, \mathsf{K}_n(x) \rangle_\psi + \tilde{\mu}(\vartheta^n(x)) - \theta\pi n \in 2\pi\mathbb Z
		\end{align}
		for all $x \in \Fix(\sigma^n)$ and $n \in \mathbb N$. Define the continuous function $f: \LieA \times M_\Gamma \to \U(1)$ by $f(w, m) = e^{i\langle v, w \rangle_\psi} \mu(m)$ for all $(w, m) \in \LieA \times M_\Gamma$. Then \cref{eqn:Periodic_x_MultipleOf2pi} becomes
		\begin{align*}
			f(\mathsf{K}_n(x), \vartheta^n(x)) = e^{i\theta\pi n} \qquad \text{for all $x \in \Fix(\sigma^n)$ and $n \in \mathbb N$}.
		\end{align*}
		By the density result of \cref{pro:SubgroupGeneratedBy_K_AndHolonomyDenseInLieATimesM_Gamma} and finiteness of $\{e^{i\theta\pi n}: n \in \mathbb Z\}$, the function $f$ must be constant on connected components. Since $f(0, e) = 1$, we must have $f|_{\LieA \times M^\circ} = 1$. Thus, $f(w, e) = e^{i\langle v, w \rangle_\psi} = 1$ for all $w \in \LieA$ which implies that $v = 0$. Similarly, $\mu|_{M^\circ} = 1$. Moreover, $M_\Gamma/M^\circ \cong (\mathbb Z/2 \mathbb Z)^p$ for some integer $0 \leq p \leq \rank$ implies that $\mu|_{mM^\circ} = \pm 1$ for all $m \in M_\Gamma$. By the topological transitivity of the transition matrix $T$, there exist $x \in \Fix(\sigma^{n}) \neq \varnothing$ and $x' \in \Fix(\sigma^{n + 1}) \neq \varnothing$ for any $n \geq N_T$ (see \cref{subsec:SymbolicDynamics}). So, $\mu\bigl(\vartheta^n(x)\bigr) = e^{i\theta\pi n}$ and $\mu\bigl(\vartheta^{n + 1}(x')\bigr) = e^{i\theta\pi (n + 1)}$ which implies $\theta \in \{0, 1\}$. If $\theta = 0$, then clearly $\mu = 1$ which contradicts $(v, \mu) \neq (0, 1)$. Now suppose $\theta = 1$. Using $y_l \in \Fix(\sigma^{n_l})$ fixed in the beginning, we get $\mu(\vartheta^{n_l}(y_l)) = e^{i\pi n_l} = 1$ for all $1 \leq l \leq l_0$. Since $\{\vartheta^{n_l}(y_l)M^\circ\}_{l = 1}^{l_0}$ generates $M_\Gamma/M^\circ$, we again have $\mu = 1$ which contradicts $(v, \mu) \neq (0, 1)$.
	\end{proof}
	
	\begin{remark}
		Although for all $x \in \Fix(\sigma^k)$ with multiplicity $j \in \mathbb N$, the elements $\lambda^M(\gamma_{x, j}), \vartheta^k(x) \in M_\Gamma$ are in the same \emph{conjugacy class} in $M$, they may \emph{not} be in the same conjugacy class in $M_\Gamma$. Thus, unlike in \cite{OP19}, we cannot simply descend to the abelianization $M_\Gamma / [M_\Gamma, M_\Gamma]$ in the above proof.
	\end{remark}
	
	We need the following lemma where $\|\cdot\|$ is the standard norm on $\mathbb R^2$.
	
	\begin{lemma}
		\label{lem:Norm<=1ImpliesNorm<1AwayFromOrigin}
		Fix $n \in \N$. Let $\mathcal{O} \subset \R^n$ be an open neighborhood of $0$, and $\alpha: \mathcal{O} \to \mathbb R^2$ be an analytic map such that $\|\alpha\| \leq 1$ and $\|\alpha(x)\| = 1$ implies $\alpha(x) \in D$ for all $x \in \mathcal{O}$ where $D \subset S^1$ is a totally disconnected subset. If $\|\alpha(0)\| = 1$, then further suppose that $\alpha$ is nonconstant on any nonconstant analytic curve in $\mathcal{O}$ through $0$. Then, shrinking $\mathcal{O}$ if necessary, $\|\alpha(x)\| < 1$ for all $x \in \mathcal{O} \setminus \{0\}$.
	\end{lemma}
	
	\begin{proof}
		Let $\alpha: \mathcal{O} \to \mathbb R^2$ be an analytic map as in the lemma. If $\|\alpha(0)\| < 1$, then the lemma follows by continuity. Now suppose $\|\alpha(0)\| = 1$. Applying Lojasiewicz's theorem \cite[Theorem 6.3.3]{KP02} to the analytic function $\|\alpha\|^2-1$ and shrinking $\mathcal{O}$ if necessary, we have either $(\|\alpha\|^2-1)^{-1}(0) \cap (\mathcal{O} \setminus \{0\}) = \varnothing$ or $(\|\alpha\|^2-1)^{-1}(0)$ contains an analytic submanifold of dimension at least 1 whose closure contains $0$. If we have the latter, then there exists a nonconstant analytic curve in $\mathcal{O}$ through $0$ which $\alpha$ maps into the totally disconnected set $D$ and so $\alpha$ must be constant on this curve, contradicting the hypothesis. Thus, $(\|\alpha\|^2-1)^{-1}(0) \cap (\mathcal{O} \setminus \{0\}) = \varnothing$.
	\end{proof}
	
	For all $\mu \in \widehat{M}_\Gamma$, we denote by $\mathcal{B}(L(\Sigma^+, V_\mu))$ the Banach algebra of bounded operators on $L(\Sigma^+, V_\mu)$.
	
	\begin{theorem}
		\label{thm:AnalyticityOf_Id-TransferOperator}
		For all $\mu \in \widehat{M}_\Gamma$, define the map $L_\mu: D_\mu \to \mathcal{B}(L(\Sigma^+, V_\mu))$ by
		\begin{align*}
			L_\mu(v) = \sum_{k = 0}^\infty \mathcal{L}_{v, \mu}^k = \bigl(\Id_{L(\Sigma^+, V_\mu)} - \mathcal{L}_{v, \mu}\bigr)^{-1} \qquad \text{for all $v \in D_\mu$}
		\end{align*}
		where $D_\mu \subset \LieA$ is the subset on which the sum converges. Then, we have:
		\begin{enumerate}
			\item\label{itm:Nontrivial_mu_GeometricSeriesTransferOperatorAnalytic} for all $\mu \in \widehat{M}_\Gamma$ with $\mu \neq 1$, the subset $D_\mu \subset \LieA$ is cocountable, hence of full Lebesgue measure, and there exists an analytic extension $L_\mu: \LieA \to \mathcal{B}(L(\Sigma^+, V_\mu))$;
			\item\label{itm:Trivial_mu_GeometricSeriesTransferOperatorAnalyticOutside0} the subset $D_1 \subset \LieA$ is cocountable, hence of full Lebesgue measure, and there exists an analytic extension $L_1: \LieA \setminus \{0\} \to \mathcal{B}(L(\Sigma^+, \mathbb C))$, and moreover there exists an open neighborhood $\mathcal{O} \subset \LieA$ of $0$ and analytic maps
			\begin{itemize}[label=\textbullet]
				\item $\mathcal{O} \to \mathbb C$ denoted by $v \mapsto \kappa_{v, 1}$,
				\item $\mathcal{O} \to \mathcal{B}(L(\Sigma^+, \mathbb C))$ denoted by $v \mapsto P_{v, 1}$,
				\item $\mathcal{O} \to \mathcal{B}(L(\Sigma^+, \mathbb C))$ denoted by $v \mapsto Q_{v, 1}$
			\end{itemize}
			such that
			\begin{align*}
				L_1(v) = \frac{P_{v, 1}}{1 - \kappa_{v, 1}} + Q_{v, 1} \qquad \text{for all $v \in \mathcal{O} \setminus \{0\}$}.
			\end{align*}
		\end{enumerate}
	\end{theorem}
	
	\begin{proof}
		Let $\mu \in \widehat{M}_\Gamma$ and define $L_\mu: D_\mu \to \mathcal{B}(L(\Sigma^+, V_\mu))$ as in the theorem. Note that any closed subset of $\LieA$ is a countable union of compact subsets and is hence a Lindel\"{o}f space, i.e., any open cover has a countable subcover. Thus, for \cref{itm:Nontrivial_mu_GeometricSeriesTransferOperatorAnalytic} and the first part of \cref{itm:Trivial_mu_GeometricSeriesTransferOperatorAnalyticOutside0}, it suffices to show that for all $v \in \LieA$, there exists an open neighborhood $\mathcal{O}_v \subset \LieA$ of $v$ such that $L_\mu(v') = \sum_{k = 0}^\infty \mathcal{L}_{v', \mu}^k$ converges for all $v' \in \mathcal{O}_v \setminus \{v\}$ and if $(v, \mu) \neq (0, 1)$, then $L_\mu|_{\mathcal{O}_v \setminus \{v\}}$ has an analytic extension to $L_\mu|_{\mathcal{O}_v}: \mathcal{O}_v \to \mathcal{B}(L(\Sigma^+, V_\mu))$.
		
		Let $v \in \LieA$. By \cref{pro:PreliminaryTransferOperatorSpectrum}, \cref{thm:SpectralBound}, \cref{lem:Norm<=1ImpliesNorm<1AwayFromOrigin}, and perturbation theory (see \cite[Chapter 7]{Kat95} and \cite[Proposition 4.6]{PP90}), there exists an open neighborhood $\mathcal{O}_v \subset \LieA$ of $v$ and analytic maps
		\begin{itemize}[label=\textbullet]
			\item $\mathcal{O}_v \to \mathbb C$ denoted by $v' \mapsto \kappa_{v', \mu}$,
			\item $\mathcal{O}_v \to \mathcal{B}(L(\Sigma^+, V_\mu))$ denoted by $v' \mapsto P_{v', \mu}$,
			\item $\mathcal{O}_v \to \mathcal{B}(L(\Sigma^+, V_\mu))$ denoted by $v' \mapsto R_{v', \mu}$
		\end{itemize}
		such that for all $v' \in \mathcal{O}_v$, we have:
		\begin{itemize}[label=\textbullet]
			\item $\mathcal{L}_{v', \mu} = \kappa_{v', \mu} P_{v', \mu} + R_{v', \mu}$;
			\item $\kappa_{v', \mu}$ is a maximal simple eigenvalue of $\mathcal{L}_{v', \mu}$ with $|\kappa_{v', \mu}| < 1$ if $v' \neq v$, $|\kappa_{v, \mu}| \leq 1$, $\kappa_{v, \mu} \neq 1$ if $(v, \mu) \neq (0, 1)$, and $\kappa_{0, 1} = 1$;
			\item $P_{v', \mu}$ is a projection operator onto the $1$-dimensional eigenspace corresponding to $\kappa_{v', \mu}$;
			\item $P_{v', \mu}R_{v', \mu} = R_{v', \mu}P_{v', \mu} = 0$;
			\item the spectral radius of $R_{v', \mu}$ is strictly less than $1$.
		\end{itemize}
		Note that the hypothesis of \cref{lem:Norm<=1ImpliesNorm<1AwayFromOrigin} is satisfied because if $\kappa_{v, \mu} = e^{i\theta\pi}$ for some $\theta \in (\mathbb R \setminus \mathbb Q) \cup \{0\}$ and the analytic map $\mathcal{O}_v \to \mathbb C$ defined by $v' \mapsto \kappa_{v', \mu}$ is constant on any nonconstant continuous curve in $\mathcal{O}_v$ through $v$, then the arguments in \cref{thm:SpectralBound} hold and we obtain a contradiction directly from \cref{eqn:Periodic_x_MultipleOf2pi} and \cref{pro:SubgroupGeneratedBy_K_AndHolonomyDenseInLieATimesM_Gamma}. Using the above properties, we have
		\begin{align*}
			L_\mu(v') = \sum_{k = 0}^\infty \mathcal{L}_{v', \mu}^k &= \left(\sum_{k = 0}^\infty \kappa_{v', \mu}\right) P_{v', \mu} + \sum_{k = 0}^\infty R_{v', \mu} \\
			&= \frac{P_{v', \mu}}{1 - \kappa_{v', \mu}} + (\Id_{L(\Sigma^+, V_\mu)} - R_{v', \mu})^{-1} \\
			&= \bigl(\Id_{L(\Sigma^+, V_\mu)} - \mathcal{L}_{v', \mu}\bigr)^{-1}
		\end{align*}
		for all $v' \in \mathcal{O}_v \setminus \{v\}$ as desired. But of course, $\sum_{k = 0}^\infty R_{v', \mu}$ converges for all $v' \in \mathcal{O}_v$ and if $(v, \mu) \neq (0, 1)$, then $\frac{P_{v, \mu}}{1 - \kappa_{v, \mu}}$ is also well-defined. Thus, it is clear that if $(v, \mu) \neq (0, 1)$, we obtain the required analytic extension $L_\mu|_{\mathcal{O}_v}: \mathcal{O}_v \to \mathcal{B}(L(\Sigma^+, V_\mu))$ using either of the last two formulas above. Setting $\mathcal{O} = \mathcal{O}_{0}$ and $Q_{v, 1} = (\Id_{L(\Sigma^+, V_\mu)} - R_{v, \mu})^{-1}$ for all $v \in \mathcal{O}$, the above also proves the second part of \cref{itm:Trivial_mu_GeometricSeriesTransferOperatorAnalyticOutside0}.
	\end{proof}
	
	Fix $\mathcal{O} \subset \LieA$ and $\kappa_{v, 1}$, $P_{v, 1}$, and $Q_{v, 1}$ for all $v \in \mathcal{O}$ provided by \cref{thm:AnalyticityOf_Id-TransferOperator} henceforth. Without loss of generality, using $\LieA = \mathbb R\mathsf{v} \oplus \ker\psi$, assume that $\mathcal{O} = \mathcal{O}_{\mathbb R} \times \mathcal{O}_\psi$ for some open neighborhoods $\mathcal{O}_{\mathbb R} \subset \mathbb R\mathsf{v}$ of $0$ and $\mathcal{O}_\psi \subset \ker\psi$ of $0$.
	
	\section{Local mixing}
	\label{sec:LocalMixing}
	Recall that we denote by $\mathsf{m}$ the restriction of $\BMS$ to one of the $A$-ergodic components of the support $\Omega$, namely $\Omega_0$ (see \cref{subsec:ErgodicDecomposition}). We prove the following theorem from which \cref{thm:LocalMixing} follows.
	
	\begin{theorem}
		\label{thm:LocalMixingForErgodicComponent}
		There exist $\kappa_\mathsf{v} > 0$ and an inner product $\langle \cdot, \cdot \rangle_*$ and a corresponding norm $\|\cdot\|_*$ on $\LieA$ such that for all $r: \R_{\geq 0} \to \R_{\geq 0}$ with $\lim\limits_{t \to +\infty}\frac{r(t)}{t} \in [0, +\infty]$ and $\lim\limits_{t \to +\infty}\frac{r(t)^2}{t} =: \ell \in [0, +\infty]$, $\mathsf{u} \in \ker\psi$, and $\phi_1, \phi_2 \in C_{\mathrm{c}}(\Gamma \backslash G)$, we have
		\begin{align*}
			\lim_{t \to +\infty} t^{\frac{\rank - 1}{2}} \int_{\Gamma \backslash G} \phi_1(x a_{t\mathsf{v} + r(t)\mathsf{u}}) \phi_2(x) \, d\mathsf{m}(x) = \kappa_\mathsf{v} e^{-\ell I(\mathsf{u})} \mathsf{m}(\phi_1) \mathsf{m}(\phi_2)
		\end{align*}
		with the convention that $+\infty\cdot 0 = 0$ in the exponent, where $I: \ker\psi \to \R_{\geq 0}$ is defined by $I(\mathsf{u}) = \|\mathsf{u}\|_*^2 - \frac{\langle \mathsf{u}, \mathsf{v} \rangle_*^2}{\|\mathsf{v}\|_*^2}$ for all $\mathsf{u} \in \ker\psi$. Moreover, there exists $\kappa_\mathsf{v}(\phi_1, \phi_2) > 0$ such that:
		\begin{enumerate}
			\item\label{itm:UniformBoundAgain} the left hand side in absolute value is bounded above by
			\begin{align*}
				\kappa_\mathsf{v}(\phi_1, \phi_2)e^{-2(\|t\mathsf{v} + r(t)\mathsf{u}\|_* \cdot \|\mathsf{v}\|_* - \langle t\mathsf{v} + r(t)\mathsf{u}, \mathsf{v} \rangle_*)}
			\end{align*}
			$(t, \mathsf{u}) \in \mathbb R_{\geq 0} \times \ker\psi$;
			\item\label{itm:UniformBoundOnConesAgain} if $\ell \in (0, +\infty]$, then there exist $T_r > 0$ and $\ell' \in (0, \ell)$ such that the left hand side in absolute value is bounded above by $\kappa_\mathsf{v}(\phi_1, \phi_2) e^{-\ell' I(\mathsf{u})}$ for all $(t, \mathsf{u}) \in [T_r, +\infty) \times \ker\psi$ with $t\mathsf{v} + r(t)\mathsf{u} \in \LieA^+$;
			\item\label{itm:UniformBoundOnCompactsAgain} for all compact subsets $\mathcal{K} \subset \ker\psi$, the convergence of the left hand side is uniform in $\mathsf{u} \in \mathcal{K}$.
		\end{enumerate}
	\end{theorem}
	
	\begin{remark}
		\label{rem:ExponentialDecayAlongRays}
		Taking a linear $r$ in \cref{thm:LocalMixingForErgodicComponent} shows that for $\mathsf{w} \in \mathsf{v} + (\ker\psi \setminus \{0\})$, we have 
		\begin{align*}
			\lim_{t \to +\infty} t^{\frac{\rank - 1}{2}} \int_{\Gamma \backslash G} \phi_1(x a_{t\mathsf{w}}) \phi_2(x) \,  d\mathsf{m}(x) = 0.    
		\end{align*}
		In fact, applying \cref{itm:UniformBoundAgain} in \cref{thm:LocalMixingForErgodicComponent} to $r(t) = t$ and $\mathsf{u} = \mathsf{w} - \mathsf{v}$, the left hand side in absolute value is bounded above by $\kappa_\mathsf{v}(\phi_1, \phi_2)e^{-2(\|\mathsf{w}\|_* \cdot \|\mathsf{v}\|_* - \langle \mathsf{w}, \mathsf{v} \rangle_*)t}$ for all $t \geq 0$, so the above convergence is exponentially fast. Note that for $\mathsf{w} \in \LieA^+ \setminus\limitcone$, the above integral vanishes for sufficiently large $t > 0$ by \cite[Proposition 2.19]{ELO20}.
	\end{remark}
	
	\begin{remark}
		\label{rem:OnProofsOfProperties}
		In \cref{thm:LocalMixingForErgodicComponent}, we have used \cite[Proposition 2.19]{ELO20} to simplify the statement of \cref{itm:UniformBoundOnConesAgain}. \Cref{itm:UniformBoundAgain,itm:UniformBoundOnConesAgain,itm:UniformBoundOnCompactsAgain} are obtained from a finer analysis of \cref{sec:LocalMixing}. Given that $\phi_1$ and $\phi_2$ have compact support, it can be checked that all the relevant convergences which appear are uniform in $\mathsf{u} \in \ker\psi$ except in \cref{lem:LimitOf_f_t_IsExp_I_r}. Thus, it suffices to examine \cref{lem:LimitOf_f_t_IsExp_I_r} in more detail. There are other bounds and convergence properties which can be deduced in a similar fashion.
	\end{remark}
	
	\begin{remark}
		\label{rem:ConstantKappa}
		The constant $\kappa_\mathsf{v}$ in \cref{thm:LocalMixingForErgodicComponent} is explicitly given by the formula
		\begin{align*}
			\kappa_\mathsf{v} = (2\pi)^{-\frac{\rank - 1}{2}} \mathfrak{c}^{-\frac{1}{2}}
		\end{align*}
		where $\mathfrak{c} = \det(D^2P(0))$ is the Gaussian curvature at $0$ of the analytic map $P: \mathcal{O}_\psi \to \mathbb R$ in the proof of \cref{pro:ExpansionOf_kappa}. We use the natural convention $\mathfrak{c} = 1$ if $\rank = 1$. A priori, $\mathfrak{c}$ and hence $\kappa_\mathsf{v}$, and $\langle \cdot, \cdot \rangle_*$ and hence $I$ depend on the choice of Markov section from \cref{subsec:MarkovSectionForTheTranslationFlow}. However, taking $\mathsf{u} = 0$ in \cref{thm:LocalMixingForErgodicComponent}, it is clear that $\kappa_\mathsf{v}$ is independent of the choice of Markov section and depends \emph{only} on $\mathsf{v} \in \interior(\limitcone)$. Interestingly, $\mathfrak{c}$ and $I$ are then independent of the choice of Markov section and depends \emph{only} on $\mathsf{v} \in \interior(\limitcone)$. Moreover, examining the determining equation for $P$, we see that $\kappa_\mathsf{v}$, $\mathfrak{c}$, and $I$ are continuous in $\mathsf{v} \in \interior(\limitcone)$. In \cref{thm:LocalMixingNoPsi,thm:LocalMixing,thm:decayofmatrixcoefficients}, since $\BMS$ was not normalized as in \cref{eqn:ProductOfBMSAndLebesgue}, the formula for $\kappa_\mathsf{v}$ is to be adjusted accordingly.
	\end{remark}
	
	\subsection{Expansion of the maximal simple eigenvalue}
	For all $v \in \mathcal{O}$, the transfer operator $\mathcal{L}_{v, 1}$ has a unique positive eigenvector $H_{v, 1}$ for the maximal simple eigenvalue $\kappa_{v, 1}$ which is normalized so that $\int_{\Sigma^+} H_{v, 1} \, d\nu_{\Sigma^+} = 1$. Then $v \mapsto H_{v, 1}$ is an analytic map on $\mathcal{O}$. Note that $H_0 := H_{0, 1} = \chi_{\Sigma^+}$. Denote by $\kappa: \mathcal{O} \to \mathbb C$ the map which gives the eigenvalue $\kappa(v) = \kappa_{v, 1}$ for all $v \in \mathcal{O}$. In this subsection we prove \cref{pro:ExpansionOf_kappa} which provides the second order expansion of $\kappa$ about $0 \in \mathcal{O}$. The proof requires the identities in the following proposition.
	
	\begin{proposition}
		\label{pro:AverageOfFirstReturnVectorMaps}
		The first return vector maps have the averaging properties
		\begin{align*}
			\int_{\Sigma^+} \mathsf{K} \, d\nu_{\Sigma^+} &= \nu_{\Sigma^+}(\tau)\mathsf{v}, & \int_{\Sigma^+} \widehat{\mathsf{K}} \, d\nu_{\Sigma^+} &= 0.
		\end{align*}
	\end{proposition}
	
	\begin{proof}
		It suffices to prove the first formula because it implies the second using \cref{eqn:K_Equals_tau+K-hat}. We recall the averaging formula in \cite[Page 1775]{Sam15} given by
		\begin{align}
			\label{eqn:AverageK'InTangentDirection}
			\int_{\Sigma} \mathsf{K}' \, d\nu_{\Sigma} = \mathsf{v}' \int_{\Sigma} \tau \, d\nu_{\Sigma}
		\end{align}
		where we denote $\mathsf{K}': \Sigma \to \LieA$ to be the map as defined in \cite[Page 1774]{Sam15} and according to \cite[Definition 3.3]{Sam15}, $\mathsf{v}' \propto \int_{\mathcal{X}} F_\sigma \, dm_\mathsf{v} \in \limitcone$ normalized such that $\psi(\mathsf{v}') = 1$ where $F_\sigma: \mathcal{X} \to \LieA$ is a map associated to the Iwasawa cocycle $\sigma$ due to a theorem of Ledrappier \cite[Theorem 2.14]{Sam15}. They have the property
		\begin{align*}
			\mathsf{K}'_k(\overline{x}) = \int_0^{\tau_k(x)} F_\sigma(\mathcal{W}_t(\eta(\overline{x}))) \, dt = \lambda(\gamma_{x, j})
		\end{align*}
		for all $x \in \Fix(\sigma^k)$ with multiplicity $j \in \mathbb N$ corresponding to some $k \in \mathbb N$ where $\overline{x} \in \Sigma$ is the periodic bi-infinite extension of $x$. Moreover, $\mathsf{v}' = \mathsf{v}$ by \cite[Corollary 4.9]{Sam14a}. Now, using the above property, \cref{lem:FirstReturnVectorAndHolonomyEqualsGeneralizedJordanProjectionOnFixedPoints}, and a theorem of Liv\v{s}ic \cite[Proposition 3.7]{PP90}, we can conclude that $\mathsf{K}: \Sigma \to \LieA$ and $\mathsf{K}': \Sigma \to \LieA$ are cohomologous. Thus, using \cref{eqn:AverageK'InTangentDirection}, we simply have
		\begin{align*}
			\int_{\Sigma^+} \mathsf{K} \, d\nu_{\Sigma^+} = \int_{\Sigma} \mathsf{K} \, d\nu_{\Sigma} = \int_{\Sigma} \mathsf{K}' \, d\nu_{\Sigma} = \mathsf{v} \int_{\Sigma} \tau \, d\nu_{\Sigma} = \mathsf{v} \int_{\Sigma^+} \tau \, d\nu_{\Sigma^+}.
		\end{align*}
	\end{proof}
	
	\begin{remark}
		Care needs to be taken above since $\mathcal{X}$ is not necessarily a manifold.
		As \cref{eqn:AverageK'InTangentDirection} is from \cite[Proposition 3.5]{Sam15}, we check its validity. Firstly, \cite[Theorem 2.20]{Sam15} can be replaced with \cref{thm:TranslationFlowConjugateToGromovGeodesicFlow} and \cref{subsec:Thermodynamics}. Secondly, similar to the proof of \cite[Theorem 3]{Led95} we can explicitly construct $F_\sigma: \mathcal{X} \to \LieA$. Fixing a reference point $X_0 \in \limitset^{(2)} \times \mathbb R$ and any nonnegative $\rho \in C_{\mathrm{c}}^\infty(\mathbb R)$ such that $\rho(0) = 1$, $\rho'(0) = \rho''(0) = 0$, and $\rho(t) > \frac{1}{2}$ if $|t| < 2\sup_{X \in \limitset^{(2)} \times \mathbb R} d(X, \Gamma \cdot X_0)$, we define the map $\tilde{F}_\sigma: \limitset^{(2)} \times \mathbb R \to \LieA$ by
		\begin{align*}
			\tilde{F}_\sigma(X) = \left.\frac{d}{dt}\right|_{t = 0} \log\sum_{\gamma \in \Gamma} \rho(d(\mathcal{W}_t(X), \gamma \cdot X_0)) e^{-\sigma(\gamma^{-1}, X^+)}
		\end{align*}
		for all $X \in \limitset^{(2)} \times \mathbb R$, where we interpret the operations componentwise using any identification $\LieA \cong \R^\rank$ (cf. \cite[Eq. (5)]{Sam14a}). Then $\tilde{F}_\sigma$ is left $\Gamma$-invariant and descends to the desired $F_\sigma$. Lastly, \cite[Lemma 2.18]{Sam15} is proved using \cite[Lemmas 3.8 and 3.9]{Sam14b} (see \cite[Page 464]{Sam14b}) which hold in general for compact metric spaces.
	\end{remark}
	
	We now present \cref{pro:ExpansionOf_kappa}. The first property is that $D\kappa(0)$ can be written in terms of $\mathsf{v}$. A special case appears in \cite[Propositions 12.53 and B.2]{Thi07} but we prove it differently in our setting. We denote by $D_v$ the derivative in the parameter $v \in \LieA$ with respect to $\langle \cdot, \cdot \rangle_\psi$. In the expansion of $\kappa$ in \cite[Lemma A.3]{Thi07} or \cite[Lemma 10.6]{Thi09}, and \cite[Proposition 1.22]{Bab88}, it seems that a specific formula for $D^2\kappa(0)$ in terms of $D_v^2|_{v = 0}\mathcal{L}_{v, 1}$ was used from \cite[Lemma 3]{Gui84} while the hypothesis that $D_v|_{v = 0}\mathcal{L}_{v, 1}(H_0) = 0$ was \emph{not satisfied}. Nevertheless, the second property that $D^2\kappa(0)$ is \emph{negative definite} can be proven in our setting.
	
	\begin{proposition}
		\label{pro:ExpansionOf_kappa}
		There exist an inner product $\langle \cdot, \cdot \rangle_*$ and a corresponding norm $\|\cdot\|_*$ on $\LieA$ such that
		\begin{align*}
			\kappa(v) = 1 + i\nu_{\Sigma^+}(\tau)\langle v, \mathsf{v} \rangle_\psi - \|v\|_*^2 + o(\|v\|_\psi^2) \qquad \text{for all $v \in \mathcal{O}$}.
		\end{align*}
	\end{proposition}
	
	%Whole proof needs to be changed.
	\begin{proof}
		We have $\kappa(v) = \int_{\Sigma^+} \mathcal{L}_{v, 1}(H_{v, 1}) \, d\nu_{\Sigma^+}$ for all $v \in \mathcal{O}$. Differentiating and using both $\mathcal{L}_0^*(\nu_{\Sigma^+}) = \nu_{\Sigma^+}$ and \cref{pro:AverageOfFirstReturnVectorMaps} gives (cf. \cite[Proposition 4.10]{PP90})
		\begin{align}
			\label{eqn:DerivativeOfEigenvalue}
			\begin{aligned}
				D\kappa(0) = D_v|_{v = 0} \kappa(v) ={}&\int_{\Sigma^+} D_v|_{v = 0}\mathcal{L}_{v, 1}(H_0) \, d\nu_{\Sigma^+} + \int_{\Sigma^+} \mathcal{L}_0(D_v|_{v = 0}H_{v, 1}) \, d\nu_{\Sigma^+} \\
				={}&i\int_{\Sigma^+} \langle \cdot, \mathsf{K}(x) \rangle_\psi H_0(x) \, d\nu_{\Sigma^+}(x) + D_v|_{v = 0} \int_{\Sigma^+} H_{v, 1} \, d\nu_{\Sigma^+} \\
				={}&i\left\langle \cdot, \int_{\Sigma^+} \mathsf{K} \, d\nu_{\Sigma^+} \right\rangle_\psi \\
				={}&i\left\langle \cdot, \nu_{\Sigma^+}(\tau)\mathsf{v} \right\rangle_\psi.
			\end{aligned}
		\end{align}
		Thus, by Taylor's theorem, we already obtain the formula
		\begin{align}
			\label{eqn:PreliminaryExpansionOf_kappa}
			\kappa(v) = 1 + i\nu_{\Sigma^+}(\tau)\langle v, \mathsf{v} \rangle_\psi + \frac{1}{2}D^2\kappa(0)(v, v) + o(\|v\|_\psi^2) \qquad \text{for all $v \in \mathcal{O}$}.
		\end{align}
		
		The map $(t, v) \mapsto \Pr_\sigma(-t\tau + i\langle v, \widehat{\mathsf{K}} \rangle)$ is analytic on some neighborhood of $(1, 0) \in \mathbb R \times \ker\psi$ by perturbation theory where we use the more general definition of pressure as the principal logarithm of the maximal simple eigenvalue of a corresponding transfer operator \cite[Chapter 4]{PP90}. Calculating the derivative as in \cref{eqn:DerivativeOfEigenvalue}, we have $D_t|_{(t, v) = (1, 0)}\Pr_\sigma(-t\tau + i\langle v, \widehat{\mathsf{K}} \rangle) = -\nu_{\Sigma^+}(\tau) \neq 0$. Thus, assuming $\mathcal{O}_\psi$ is sufficiently small, we can invoke the implicit function theorem to define the analytic function $P: \mathcal{O}_\psi \to \mathbb R$ with $P(0) = 1$ (see \cref{subsec:Thermodynamics}) satisfying
		\begin{align*}
			\Pr_\sigma(-P(v)\tau + i\langle v, \widehat{\mathsf{K}} \rangle) = 0 \qquad \text{for all $v \in \mathcal{O}_\psi$}.
		\end{align*}
		Exponentiating the above and calculating the derivative as in \cref{eqn:DerivativeOfEigenvalue} using \cref{pro:AverageOfFirstReturnVectorMaps} gives $DP(0) = 0$. Now, $D^2P(0)$ is of course Hermitian and a similar calculation of second derivatives as in the proof of \cite[Lemma 8(3)]{PS94} using $DP(0) = 0$ shows that it is moreover negative definite by \cite[Lemma 5]{PS94} (cf. \cite[Propositions 4.11 and 4.12]{PP90}) provided that $\langle v, \widehat{\mathsf{K}} \rangle$ is not cohomologous to a constant for all nonzero $v \in \ker\psi$. To obtain a contradiction, suppose $\langle v, \widehat{\mathsf{K}} \rangle - c = \omega - \omega \circ \sigma$ for some nonzero $v \in \ker\psi$, $\omega \in C(\Sigma^+)$, and $c \in \mathbb R$. Integrating over $\Sigma^+$ with respect to $\nu_{\Sigma^+}$ and using \cref{pro:AverageOfFirstReturnVectorMaps} implies $c = 0$. Using \cref{lem:FirstReturnVectorAndHolonomyEqualsGeneralizedJordanProjectionOnFixedPoints} gives $\langle v, \proj_{\ker\psi}(\lambda(\Gamma))\rangle = \{0\}$ and then \cref{pro:SubgroupOfGeneralizedLengthSpectrumOfGamma_0DenseInAM_Gamma_0} implies $v = 0$ which is a contradiction.
		
		Using the Weierstrass preparation theorem (cf. \cite[Appendix]{LS06}), we can write
		\begin{align}
			\label{eqn:WeierstrassPreparation}
			\begin{aligned}
				\kappa(b\mathsf{v} + v) - 1 &= e^{\Pr_\sigma(-\tau + ib\tau + i\langle v, \widehat{\mathsf{K}} \rangle)} - e^{\Pr_\sigma(-P(v)\tau + i\langle v, \widehat{\mathsf{K}} \rangle)} \\
				&= Q(b\mathsf{v} + v)(P(v) - 1 + ib)
			\end{aligned}
		\end{align}
		for all $(b\mathsf{v}, v) \in \mathcal{O}_\R \times \mathcal{O}_\psi$, where $Q: \mathcal{O} \to \mathbb C$ is an analytic function with $Q(0) \neq 0$. In fact, \cref{eqn:DerivativeOfEigenvalue} and $DP(0) = 0$ gives $Q(0) = \nu_{\Sigma^+}(\tau) > 0$. So, $D_v^2|_{(b, v) = (0, 0)} \kappa(b\mathsf{v} + v) = Q(0)D^2P(0)$ is negative definite. In other words, $D^2\kappa(0)(v, v) < 0$ for all $v \in \ker\psi$.
		
		It remains to show that $D^2\kappa(0)(v, v) < 0$ also for all $v \in \mathcal{O} \setminus \ker\psi$. To this end, observe that
		\begin{align*}
			\mathcal{L}_{-v, 1}(\overline{H_{v, 1}}) = \overline{\kappa_{v, 1}} \cdot \overline{H_{v, 1}} \qquad \text{for all $v \in \mathcal{O}$}.
		\end{align*}
		Hence, $\kappa(-v) = \overline{\kappa(v)}$ for all $v \in \mathcal{O}$. Taking derivatives, $D\kappa(0) \in \Hom(\LieA, i\mathbb R)$ and $D^2\kappa(0) \in \Hom(\LieA \otimes \LieA, \mathbb R)$. Let $v \in \mathcal{O} \setminus \ker\psi$ and $c = D^2\kappa(0)(v, v) \in \mathbb R$. As mentioned in \cref{thm:AnalyticityOf_Id-TransferOperator}, we have $|\kappa(v)|^2 < 1$ by \cref{thm:SpectralBound} and its proof and \cref{lem:Norm<=1ImpliesNorm<1AwayFromOrigin}. Using this in \cref{eqn:PreliminaryExpansionOf_kappa} gives
		\begin{align*}
			%1 + \nu_{\Sigma^+}(\tau)|\langle v, \mathsf{v} \rangle_\psi|^2 + c\|v\|_\psi^2 + o(\|v\|_\psi^2) < 1.
			1 + \nu_{\Sigma^+}(\tau)^2|\langle v, \mathsf{v} \rangle_\psi|^2 + c + o(\|v\|_\psi^2) < 1.
		\end{align*}
		This implies $c < 0$ because $\nu_{\Sigma^+}(\tau)^2|\langle v, \mathsf{v} \rangle_\psi|^2 > 0$ as $v \notin \ker\psi$. Thus $D^2\kappa(0)$ is negative definite. Defining the inner product $\langle \cdot, \cdot \rangle_*$ on $\LieA$ by $\langle v, w \rangle_* = -\frac{1}{2}D^2\kappa(0)(v, w)$ for all $v, w \in \LieA$ completes the proof.
	\end{proof}
	
	\subsection{Fourier analysis}
	For the rest of this section, fix $\langle \cdot, \cdot \rangle_*$ and $\|\cdot\|_*$ provided by \cref{pro:ExpansionOf_kappa}, and fix $r:\R_{\geq 0} \to \R_{\geq 0}$, $\ell \in [0,+\infty]$, $\mathsf{u} \in \ker\psi$, $I: \ker\psi \to \R_{\geq 0}$, and the convention that $+\infty \cdot 0 = 0$ for the product $\ell I(\mathsf{u})$ as in \cref{thm:LocalMixingForErgodicComponent}.
	
	We need some tools regarding Bessel functions. One way to define the modified Bessel function of the second kind $K_\alpha: \mathbb R_{>0} \to \mathbb R$ for each parameter $\alpha \in \mathbb R$ is by
	\begin{align*}
		K_\alpha(x) = \frac{1}{2} \int_0^{+\infty} t^{-(\alpha + 1)} e^{-\frac{x}{2}\left(t + \frac{1}{t}\right)} \, dt \qquad \text{for all $x > 0$}.
	\end{align*}
	Define the function $E: \mathbb R_{>0} \to \mathbb R$ by
	\begin{align*}
		E(x) = \frac{1}{\Gamma\bigl(\frac{\rank - 1}{2}\bigr)} \int_0^{+\infty} e^{-t} \left(t + \frac{t^2}{2x}\right)^{\frac{\rank - 3}{2}} \, dt.
	\end{align*}
	
	\begin{remark}
		If $\rank = 3$, then $E = 1$.
	\end{remark}
	
	Babillot \cite[Subsection 2.28]{Bab88} recalls the following from \cite[Section 7.3]{Wat95}.
	
	\begin{lemma}
		\label{lem:BesselFunctionInTermsOf_E}
		We have:
		\begin{enumerate}
			\item $K_{\rank/2 - 1}(x) = \sqrt{\frac{\pi}{2x}}e^{-x}E(x)$ for all $x > 0$;
			\item\label{itm:E_GoesTo1} $\lim_{x \to +\infty} E(x) = 1$;
			\item there exists $C > 0$ such that $0 \leq E(x) \leq C\left(1 + x^{-\frac{\rank - 3}{2}}\right)$ for all $x > 0$.
		\end{enumerate}
	\end{lemma}
	
	\begin{lemma}
		\label{lem:LimitOf_f_t_IsExp_I_r}
		Let $\mathcal{K} \subset \LieA$ be a compact subset. For all sufficiently large $t > 0$ consider the function $f_t: \mathcal{K} \to \mathbb R$ defined by
		\begin{align*}
			f_t(v) ={}&e^{2\langle t\mathsf{v} + r(t)\mathsf{u} + v, \mathsf{v} \rangle_* - 2\|t\mathsf{v} + r(t)\mathsf{u} + v\|_* \cdot \|\mathsf{v}\|_*} E(2\|t\mathsf{v} + r(t)\mathsf{u} + v\|_* \cdot \|\mathsf{v}\|_*) \\
			{}&\cdot \left(\frac{\|t\mathsf{v} + r(t)\mathsf{u}\|_*}{\|t\mathsf{v} + r(t)\mathsf{u} + v\|_*}\right)^{\frac{\rank - 1}{2}}
		\end{align*}
		for all $v \in \LieA$. Then $f_t \xrightarrow{t \to +\infty} e^{-\ell I(\mathsf{u})}$ uniformly.
	\end{lemma}
	
	\begin{proof}
		Let $\mathcal{K} \subset \LieA$ and $f_t: \mathcal{K} \to \mathbb R$ for all sufficiently large $t > 0$ be as in the lemma. Using \cref{itm:E_GoesTo1} in \cref{lem:BesselFunctionInTermsOf_E} and observing $\frac{\|t\mathsf{v} + r(t)\mathsf{u}\|_*}{\|t\mathsf{v} + r(t)\mathsf{u} + v\|_*} \xrightarrow{t \to +\infty} 1$ uniformly for all $v \in \mathcal{K}$, it suffices to show that
		\begin{align*}
			\|t\mathsf{v} + r(t)\mathsf{u} + v\|_* \cdot \|\mathsf{v}\|_* - \langle t\mathsf{v} + r(t)\mathsf{u} + v, \mathsf{v} \rangle_* \xrightarrow{t \to +\infty} \ell\frac{\|\mathsf{u}\|_*^2 \cdot \|\mathsf{v}\|_*^2 - \langle \mathsf{u}, \mathsf{v} \rangle_*^2}{2\|\mathsf{v}\|_*^2}
		\end{align*}
		uniformly for all $v \in \mathcal{K}$. But this is immediate from the simple calculation that for all $v \in \mathcal{K}$ and sufficiently large $t > 0$, we have
		\begin{align*}
			&\|t\mathsf{v} + r(t)\mathsf{u} + v\|_* \cdot \|\mathsf{v}\|_* - \langle t\mathsf{v} + r(t)\mathsf{u} + v, \mathsf{v} \rangle_* \\
			={}&\frac{r(t)^2}{t} \cdot \frac{\bigl\|\mathsf{u} + \frac{1}{r(t)}v\bigr\|_*^2 \cdot \|\mathsf{v}\|_*^2 - \bigl\langle \mathsf{u} + \frac{1}{r(t)}v, \mathsf{v} \bigr\rangle_*^2}{\bigl\|\mathsf{v} + \frac{r(t)}{t}\mathsf{u} + \frac{1}{t}v\bigr\|_* \cdot \|\mathsf{v}\|_* + \bigl\langle \mathsf{v} + \frac{r(t)}{t}\mathsf{u} + \frac{1}{t}v, \mathsf{v} \bigr\rangle_*}
		\end{align*}
		for the case that $\mathsf{u} = 0$ or $\lim_{t \to +\infty} \frac{r(t)}{t} \in [0, +\infty)$, and
		\begin{align*}
			&\|t\mathsf{v} + r(t)\mathsf{u} + v\|_* \cdot \|\mathsf{v}\|_* - \langle t\mathsf{v} + r(t)\mathsf{u} + v, \mathsf{v} \rangle_* \\
			={}&r(t)\left(\left\|\frac{t}{r(t)}\mathsf{v} + \mathsf{u} + \frac{1}{r(t)}v\right\|_* \cdot \|\mathsf{v}\|_* - \left\langle \frac{t}{r(t)}\mathsf{v} + \mathsf{u} + \frac{1}{r(t)}v, \mathsf{v} \right\rangle_*\right)
		\end{align*}
		for the case that $\mathsf{u} \neq 0$ and $\lim_{t \to +\infty} \frac{r(t)}{t} = (0, +\infty]$.
	\end{proof}
	
	Fix any $\rho_\mathcal{O} \in C^\infty_\mathrm{c}(\LieA)$ such that $\rho_\mathcal{O}$ takes the value 1 on a neighborhood of $0$ and $\supp(\rho_\mathcal{O}) \subset \mathcal{O}$ for the rest of the subsection. Let us also fix $f \in L(\Sigma^+)$ and $\omega \in C_\mathrm{c}(\LieA)$ for this subsection. For all $k \in \mathbb Z_{\geq 0}$, $\mu \in \widehat{M}_\Gamma$, $w \in V_\mu$, and $t \in \mathbb R$, define the maps $\mathsf{Q}^{(k)}_{\mu,w,t}(f,\omega), \mathsf{Q}_{\mu,w,t}(f,\omega): \Sigma^+ \times \ker\psi \to V_\mu$ by
	\begin{align*}
		\mathsf{Q}^{(k)}_{\mu,w,t}(f,\omega)(x,u) &= \frac{1}{2\pi}\int_{\LieA}e^{-i\langle v, t\mathsf{v} + u \rangle_\psi}\widehat{\omega}(v)\mathcal{L}_{v, \mu}^k(f w)(x) \, dv, \\
		\mathsf{Q}_{\mu,w,t}(f,\omega)(x,u) &= \frac{1}{2\pi}\int_{\LieA}e^{-i\langle v, t\mathsf{v} + u \rangle_\psi}\widehat{\omega}(v)\sum_{k = 0}^\infty\mathcal{L}_{v, \mu}^k(f w)(x) \, dv 
	\end{align*}
	for all $(x, u) \in \Sigma^+ \times \ker\psi$, where $\widehat{\omega}$ is the Fourier transform of $\omega$ with the convention $\widehat{\omega}(v) = \int_{\LieA} e^{-i\langle \xi,  v \rangle_\psi} \omega(\xi) \, d\xi$ for all $v \in \LieA$. Recall that $dv$ denotes the Lebesgue measure on $\ker\psi$ which is compatible with $\langle \cdot, \cdot \rangle_\psi$. Note that the second integrand above is defined on a subset of $\LieA$ of full Lebesgue measure by \cref{thm:AnalyticityOf_Id-TransferOperator}. Set $\mathsf{Q}^{(k)}_t := \mathsf{Q}^{(k)}_{1,1,t}$ and $\mathsf{Q}_t := \mathsf{Q}_{1,1,t}$ for all $k \in \mathbb Z_{\geq 0}$ and $t \in \mathbb R$. We will need to use the following lemmas later, the first lemma comes from \cite[Subsections 2.42 and 2.43]{Bab88} using \cref{pro:ExpansionOf_kappa} and the second lemma is proved as in \cite[Lemma A.6]{Thi07} using \cref{lem:BesselFunctionInTermsOf_E} and \cref{lem:LimitOf_f_t_IsExp_I_r}.
	
	\begin{lemma}
		\label{lem:ThirionEstimate}
		Suppose $\widehat{\omega} \in C^N_\mathrm{c}(\LieA)$ for some $N \ge \frac{\rank}{2} + 2$. Then
		\begin{multline*}
			\int_{\LieA}e^{-i\langle v, t\mathsf{v} + u \rangle_\psi}\widehat{\omega}(v)\rho_\mathcal{O}(v)\frac{P_{v, 1}(f)(x)}{1 - \kappa_{v, 1}} \, dv  \\
			= \widehat{\omega}(0)\nu_{\Sigma^+}(f)\int_{\LieA} \frac{e^{-i\langle v, t\mathsf{v} + u \rangle_\psi} \rho_\mathcal{O}(v)}{-i\nu_{\Sigma^+}(\tau)\langle v, \mathsf{v} \rangle_\psi + \|v\|_*^2 } \, dv 
			+ o\left(\|t\mathsf{v} + u\|_\psi^{-\frac{\rank-1}{2}}\right)
		\end{multline*}
		for all $(t, u) \in \mathbb R \times \ker\psi$ as $t\mathsf{v} + u \to +\infty$, uniformly in $x \in \Sigma^+$.
	\end{lemma}
	
	\begin{lemma}
		\label{lem:ThirionLimit}
		There exists $C > 0$, independent of $\mathsf{u}$, $r$, and $\rho_\mathcal{O}$, such that
		\begin{align*}
			\lim\limits_{t \to +\infty}t^{\frac{\rank-1}{2}}\int_{\LieA}\frac{e^{-i\langle v, (t-s)\mathsf{v} + r(t)\mathsf{u} - u \rangle_\psi} \rho_\mathcal{O}(v)}{-i\nu_{\Sigma^+}(\tau)\langle v, \mathsf{v} \rangle_\psi + \|v\|_*^2 } \, dv = \frac{2\pi Ce^{-\ell I(\mathsf{u})}}{\nu_{\Sigma^+}(\tau)}
		\end{align*}
		for all $(s, u) \in \R \times \ker\psi$, where the convergence is uniform on compact subsets of $\R \times \ker\psi$.
	\end{lemma}
	
	We prove the following theorem by a combination of techniques from \cite[Theorem 3.19]{OP19} which is based on \cite[Appendix]{LS06} and \cite[Appendix A]{Thi07} which is based on \cite{Bab88}. We denote the Gaussian curvature $\mathfrak c = \det(D^2P(0))$, where $P: \mathcal{O}_\psi \to \mathbb R$ is the function in the proof of \cref{pro:ExpansionOf_kappa}. We use the natural convention $\mathfrak{c} = 1$ if $\rank = 1$.
	
	\begin{theorem}
		\label{thm:QProperties}
		Let $\mu \in \widehat{M}_\Gamma$ and $w, w_1, w_2 \in V_\mu$. Then:
		\begin{enumerate}
			\item
			\label{itm:QProperties1} for all $t \in \mathbb R$, we have
			\begin{align*}
				\sum_{k = 0}^\infty\mathsf{Q}^{(k)}_{\mu,w,t}(f,\omega)  = \mathsf{Q}_{\mu,w,t}(f,\omega)
			\end{align*}
			where the convergence is uniform on compact subsets of $\Sigma^+ \times \ker\psi$;
			\item
			\label{itm:QProperties2}
			we have
			\begin{align*}
				\lim_{t \to +\infty} t^{\frac{\rank-1}{2}} \mathsf{Q}_{t-s}(f,\omega)(x,r(t)\mathsf{u} - u)  = \frac{(2\pi)^{\frac{\rank - 1}{2}}\widehat{\omega}(0)e^{-\ell I(\mathsf{u})}\nu_{\Sigma^+}(f)}{\sqrt{\mathfrak c} \cdot \nu_{\Sigma^+}(\tau)}
			\end{align*}
			for all $(x, s, u) \in \Sigma^+ \times \R \times \ker\psi$, where the convergence is uniform on compact subsets of $\Sigma^+ \times \R \times \ker\psi$;
			\item
			\label{itm:QProperties3}
			if $\mu \neq 1$, then we have
			\begin{align*}
				\lim_{t \to +\infty} \left\langle w_1, t^{\frac{\rank-1}{2}}\mathsf{Q}_{\mu, w_2, t - s}(f,\omega)(x,  r(t)\mathsf{u} - u) \right\rangle = 0
			\end{align*}
			for all $(x, s, u) \in \Sigma^+ \times \R \times \ker\psi$, where the convergence is uniform on compact subsets of $\Sigma^+ \times \R \times \ker\psi$.
		\end{enumerate}
	\end{theorem}
	
	\begin{proof}
		Let $\mu \in \widehat{M}_\Gamma$ and $w, w_1, w_2 \in V_\mu$. By \cite[Lemma 2.4]{BL98}, we will prove the theorem for nonnegative $\omega \in L^1(\LieA)$ such that $\widehat{\omega} \in C^N_\mathrm{c}(\LieA)$ for some $N \ge \frac{\rank - 1}{2} + 1$ which will imply the theorem for our original $\omega \in C_\mathrm{c}(\LieA)$.
		
		First we prove \cref{itm:QProperties1}. If $\mu = 1$, we can define the analytic function $R = 1 - P: \mathcal{O}_\psi \to \mathbb R$ in \cref{eqn:WeierstrassPreparation} and repeat Step 7 of \cite[Appendix]{LS06}. Now suppose $\mu \neq 1$. Note that the infinite sum $\sum_{k = 0}^\infty \mathcal{L}_{v, \mu}^k$ converges to $L_\mu(v)$ in norm for Lebesgue almost every $v \in \LieA$ where $L_\mu: \LieA \to \mathcal{B}(L(\Sigma^+, V_\mu))$ is the analytic extension from \cref{thm:AnalyticityOf_Id-TransferOperator}. Thus, by dominated convergence theorem, it suffices to show that $\|(\widehat{\omega} \cdot L_\mu)(\cdot)(fw)(x)\|_2$ is dominated by an integrable function $\LieA \to \mathbb R$. But this follows from
		\begin{align*}
			\|(\widehat{\omega} \cdot L_\mu)(\cdot)(fw)(x)\|_2 \leq \|\widehat{\omega}\|_\infty \cdot \|fw\|_{\Lip} \cdot \|(\chi_{\supp(\widehat{\omega})} \cdot L_\mu)(\cdot)\|
		\end{align*}
		because $\|L_\mu(\cdot)\|$ is bounded on $\supp(\widehat{\omega})$ by analyticity of $L_\mu$.
		
		Now we prove \cref{itm:QProperties2}. We use \cref{itm:Trivial_mu_GeometricSeriesTransferOperatorAnalyticOutside0} in \cref{thm:AnalyticityOf_Id-TransferOperator}, which provides analytic maps and allows replacing $\sum_{k = 0}^\infty\mathcal{L}_{v, 1}^k$ with $L_1(v)$ when $v \notin \mathcal{O}$ and with $\frac{P_{v, 1}}{1 - \kappa_{v, 1}} + Q_{v, 1}$ when $v \in \mathcal{O} \setminus \{0\}$ in the integrand, to get
		\begin{align*}
			\mathsf{Q}_{t}(f,\omega)(x,u) ={}&\frac{1}{2\pi}\int_{\LieA}e^{-i\langle v, t\mathsf{v} + u \rangle_\psi}\widehat{\omega}(v)\rho_\mathcal{O}(v)\frac{P_{v, 1}(f)(x)}{1 - \kappa_{v, 1}} \, dv 
			\\
			{}&+ \frac{1}{2\pi}\int_{\LieA}e^{-i\langle v, t\mathsf{v} + u \rangle_\psi}\widehat{\omega}(v)\rho_\mathcal{O}(v)Q_{v, 1}(f)(x) \, dv 
			\\
			{}&+ \frac{1}{2\pi}\int_{\LieA}e^{-i\langle v, t\mathsf{v} + u \rangle_\psi}\widehat{\omega}(v)(1-\rho_\mathcal{O}(v))L_1(v)(f)(x) \, dv
		\end{align*}
		for all $t \geq 0$ and $(x, u) \in \Sigma^+ \times \ker\psi$. By analyticity of the provided maps, our choice of $\rho_\mathcal{O}$, and our assumption $\widehat{\omega} \in C^N_\mathrm{c}(\LieA)$, we can integrate by parts to show that the second and third terms above are $O\bigl(\|t\mathsf{v} + u\|_\psi^{-N}\bigr)$ as $t\mathsf{v}+u \to +\infty$ uniformly in $x \in \Sigma^+$. Recalling the bound for $N$, we can combine this with \cref{lem:ThirionEstimate} and replace $t$ with $t-s$ and $u$ with $r(t)\mathsf{u} - u$ to deduce that
		\begin{multline*}
			\mathsf{Q}_{t-s}(f,\omega)(x,r(t)\mathsf{u} - u) = \frac{\widehat{\omega}(0)\nu_{\Sigma^+}(f)}{2\pi}\int_{\LieA} \frac{e^{-i\langle v, (t-s)\mathsf{v} + r(t)\mathsf{u} - u \rangle_\psi}\rho_\mathcal{O}(v)}{-i\nu_{\Sigma^+}(\tau)\langle v, \mathsf{v} \rangle_\psi + \|v\|_*^2} \, dv 
			\\
			+ o\left(\|(t-s)\mathsf{v}+r(t)\mathsf{u}-u\|_\psi^{-\frac{\rank-1}{2}}\right)
		\end{multline*}
		as $(t-s)\mathsf{v}+r(t)\mathsf{u}-u \to +\infty$ uniformly in $x \in \Sigma^+$. By \cref{lem:ThirionLimit}, we have
		\begin{align*}
			& \lim\limits_{t \to +\infty}t^{\frac{\rank-1}{2}}\mathsf{Q}_{t-s}(f,\omega)(x,r(t)\mathsf{u} - u)
			\\
			={}&\frac{\widehat{\omega}(0)\nu_{\Sigma^+}(f)}{2\pi}\lim\limits_{t \to +\infty}t^{\frac{\rank-1}{2}}\int_{\LieA}\frac{e^{-i\langle v, (t-s)\mathsf{v}+r(t)\mathsf{u}-u \rangle_\psi} \rho_\mathcal{O}(v)}{-i\nu_{\Sigma^+}(\tau)\langle v, \mathsf{v} \rangle_\psi + \|v\|_*^2} \, dv \\
			={}&\widehat{\omega}(0)\nu_{\Sigma^+}(f)\frac{Ce^{-\ell I(\mathsf{u})}}{\nu_{\Sigma^+}(\tau)}
		\end{align*}
		for all $(x, s, u) \in \Sigma^+ \times \R \times \ker\psi$, where the constant $C$ is from the same lemma and the convergence is uniform on compact subsets of $\Sigma^+ \times \R \times \ker\psi$. In the case that $\mathsf{u} = 0$ or $r = 0$ so that $\ell I(\mathsf{u}) = 0$, we can alternatively repeat the analysis in \cite[Theorem 3.19]{OP19} which is based on \cite[Appendix]{LS06}. In fact, doing this shows that the constant from \cref{lem:ThirionLimit} is explicitly $C = (2\pi)^{\frac{\rank - 1}{2}} \mathfrak{c}^{-\frac{1}{2}}$.
		
		Now we prove \cref{itm:QProperties3}. Suppose $\mu \neq 1$. We use \cref{itm:Nontrivial_mu_GeometricSeriesTransferOperatorAnalytic} in \cref{thm:AnalyticityOf_Id-TransferOperator}, which provides an analytic map $L_\mu: \LieA \to \mathcal{B}(L(\Sigma^+, V_\mu))$ and allows replacing $\sum_{k = 0}^\infty\mathcal{L}_{v, \mu}^k$ with $L_\mu(v)$ in the integrand, to get
		\begin{align*}
			&\langle w_1, \mathsf{Q}_{\mu, w_2, t - s}(f,\omega)(x, r(t)\mathsf{u} - u) \rangle \\
			={}&\frac{1}{2\pi}\int_{\LieA} e^{-i\langle v, (t - s)\mathsf{v} + r(t)\mathsf{u} - u \rangle_\psi}\widehat{\omega}(v) \langle w_1, L_\mu(v)(f w_2)(x)\rangle \, dv
		\end{align*}
		for all $t \geq 0$ and $(x, s, u) \in \Sigma^+ \times \R \times \ker\psi$. Since the integrand is of class $C^N_\mathrm{c}(\LieA)$, we can integrate by parts to show that the above is $O\bigl(\|(t - s)\mathsf{v} + r(t)\mathsf{u} - u\|_\psi^{-N}\bigr)$ as $(t - s)\mathsf{v} + r(t)\mathsf{u} - u \to +\infty$ uniformly in $x \in \Sigma^+$. Recalling the bound for $N$, this implies \cref{itm:QProperties3}.
	\end{proof}
	
	\subsection{Correlation function}
	Define $\tilde{\Sigma} = \Sigma \times \ker\psi \times M_\Gamma \times \R$ equipped with the Borel measure $\mathsf{M}$ given by
	\begin{align*}
		d\mathsf{M} = \frac{1}{\nu_{\Sigma^+}(\tau)} \, d\nu_\Sigma \, du \, dm \, ds
	\end{align*}
	where $du$ and $ds$ denote the Lebesgue measures on $\ker\psi$ and $\R$ compatible with the inner product $\langle \cdot,\cdot \rangle_\psi$ on $\LieA$, and $dm$ is the Haar measure on $M_\Gamma$ induced by the probability Haar measure on $M$. Define the map $\varsigma: \tilde{\Sigma} \to \tilde{\Sigma}$ by
	\begin{align*}
		\varsigma(x,u,m,s) = (\sigma(x), u - \widehat{\mathsf{K}}(x),\vartheta(x)^{-1}m, s - \tau(x)) \qquad \text{for all $(x, u, m, s) \in \tilde{\Sigma}$}.
	\end{align*}
	Define the suspension space $\Sigma^\varsigma = \tilde{\Sigma}/{\sim}$ where the equivalence relation $\sim$ on $\tilde{\Sigma}$ is defined by $(x,u,m,s) \sim \varsigma(x,u,m,s)$ for all $(x, u, m, s) \in \tilde{\Sigma}$. Equip $\Sigma^\varsigma$ with the Borel measure $\mathsf{M}^\varsigma$ induced by $\mathsf{M}$. The suspension flow $\Sigma^\varsigma \times \mathbb R \to \Sigma^\varsigma$ is defined by $([(x, u, m, s)], t) \mapsto [(x, u, m, s + t)]$. Then the surjective Lipschitz map $\Sigma^\varsigma \to \Omega_0$ defined by $(x,u,m,s) \mapsto F(\eta(x))ma_{u + s\mathsf{v}}$ is a semiconjugacy between the one-parameter diagonal flow $\mathcal{W}_{\mathsf{v}}$ on $\Omega_0$ and the suspension flow on $\Sigma^\varsigma$. It is important to note that the fixed $A$-ergodic component $\mathsf{m}$ on $\Omega_0$ of the BMS measure is simply the pushforward of $\mathsf{M}^\varsigma$ by this map (see \cref{sec:AnosovSubgroups,subsec:Thermodynamics}). In a similar fashion, we also define $\tilde{\Sigma}^+ = \Sigma^+ \times \ker\psi \times M_\Gamma \times \R$ equipped with a Borel measure $\mathsf{M}^+$, $\varsigma: \tilde{\Sigma}^+ \to \tilde{\Sigma}^+$ by abuse of notation, and the suspension space $\Sigma^{+,\varsigma}$ equipped with the Borel measure $\mathsf{M}^{+,\varsigma}$ induced by $\mathsf{M}^+$.
	
	Define the family of functions $\mathfrak{F}$ which consists of functions $\Psi \in C_{\mathrm{c}}(\tilde{\Sigma}^+)$ of the form $\Psi(x,u,m,s) = f(x)\omega(s\mathsf{v} + u)\langle \mu(m)w_1,w_2\rangle$ for all $(x,u,m,s) \in \tilde{\Sigma}^+$ for any $f \in L(\Sigma^+)$, $\omega \in C_\mathrm{c}(\LieA)$, $\mu \in \widehat{M}_\Gamma$, and $w_1, w_2 \in V_\mu$ with $\|w_1\|_2 = \|w_2\|_2 = 1$. We will simply say $\Psi = f(\cdot)\omega(\cdot)\langle \mu(\cdot)w_1,w_2\rangle \in \mathfrak{F}$ to specify functions of the above form. For all $\Psi_1,\Psi_2 \in \mathfrak{F}$, define the \emph{correlation function} $J(\Psi_1,\Psi_2): \mathbb R_{\geq 0} \to \mathbb R$ by
	\begin{equation*}
		J_t(\Psi_1,\Psi_2) = \sum_{k=0}^\infty\int_{\tilde{\Sigma}^+} (\Psi_1 \circ \varsigma^k)(x,u + r(t)\mathsf{u},m,s+t) \Psi_2(x,u,m,s) \, d\mathsf{M}^+(x, u, m, s)
	\end{equation*}
	for all $t \geq 0$. The following lemma gives the key relation between the correlation function and the transfer operators with holonomy. It can be proven as in \cite[Lemma 4.2]{OP19} by doing a change of variables and using $\mathcal{L}_0^*(\nu_{\Sigma^+}) = \nu_{\Sigma^+}$, unitarity of $\mu \in \widehat{M}_\Gamma$, the Fourier inversion formula, and \cref{itm:QProperties1} in \cref{thm:QProperties}.
	
	\begin{lemma}\label{lem:IntegralIFormulaUsingQ}
		For all $\Psi_1 \in C_\mathrm{c}(\tilde{\Sigma}^+)$, $\Psi_2 = f(\cdot)\omega(\cdot)\langle \mu(\cdot)w_1,w_2\rangle \in \mathfrak{F}$, and $t \geq 0$, we have
		\begin{align*}
			&(2\pi)^{\rank-1}J_t(\Psi_1,\Psi_2) 
			\\
			={}&\int_{\tilde{\Sigma}^+} \Psi_1 (x,u,m,s)\langle \mu(m)w_1, \mathsf{Q}_{\mu, w_2, t - s}(f,\omega)(x, r(t)\mathsf{u} - u) \rangle \, d\mathsf{M}^+(x, u, m, s).
		\end{align*}
	\end{lemma}
	
	\begin{proposition}\label{pro:JtLimit}
		For all $\Psi_1,\Psi_2 \in C_\mathrm{c}(\tilde{\Sigma}^+)$, we have
		\begin{equation*}
			\lim\limits_{t \to +\infty} t^{\frac{\rank-1}{2}}J_t(\Psi_1,\Psi_2) = (2\pi)^{-\frac{\rank - 1}{2}} \mathfrak{c}^{-\frac{1}{2}} e^{-\ell I(\mathsf{u})} \mathsf{M}^+(\Psi_1)\mathsf{M}^+(\Psi_2).
		\end{equation*}
	\end{proposition}
	
	\begin{proof}
		Let $\Psi_1 \in C_\mathrm{c}(\tilde{\Sigma}^+)$. First, consider any $\Psi_2 = f(\cdot)\omega(\cdot) \langle \mu(\cdot)w_1,w_2\rangle \in \mathfrak{F}$. For all $t \geq 0$, define $F_t: \tilde{\Sigma}^+ \to \mathbb R$ by
		\begin{equation*}
			F_t(x,u,m,s) = \frac{t^{\frac{\rank-1}{2}}}{(2\pi)^{\rank-1}}\Psi_1 (x,u,m,s)\langle \mu(m)w_1, \mathsf{Q}_{\mu, w_2, t - s}(f,\omega)(x, r(t)\mathsf{u} - u) \rangle
		\end{equation*}
		for all $(x,u,m,s) \in \tilde{\Sigma}^+$ so that by \cref{lem:IntegralIFormulaUsingQ}, we have
		\begin{align*}
			t^{\frac{\rank-1}{2}}J_t(\Psi_1,\Psi_2) = \int_{\tilde{\Sigma}^+} F_t(x,u,m,s) \, d\mathsf{M}^+(x,u,m,s).
		\end{align*}
		Suppose $\mu = 1$. We can assume $\langle w_1, w_2\rangle = 1$. By \cref{itm:QProperties2} in \cref{thm:QProperties}, $F_t$ is dominated by a constant multiple of $\Psi_1$ for sufficiently large $t > 0$ and converges pointwise to
		\begin{align*}
			\frac{(2\pi)^{\frac{\rank - 1}{2}}\widehat{\omega}(0)e^{-\ell I(\mathsf{u})}\nu_{\Sigma^+}(f)}{(2\pi)^{\rank-1}\sqrt{\mathfrak{c}} \cdot \nu_{\Sigma^+}(\tau)}\Psi_1 = (2\pi)^{-\frac{\rank - 1}{2}}\mathfrak{c}^{-\frac{1}{2}}e^{-\ell I(\mathsf{u})}\mathsf{M}^+(\Psi_2)\Psi_1.
		\end{align*}
		Thus, the proposition follows in this case by the dominated convergence theorem. Now suppose $\mu \neq 1$. By \cref{itm:QProperties3} in \cref{thm:QProperties}, $F_t$ is dominated by $\Psi_1$ for sufficiently large $t > 0$ and converges pointwise to $0$. Again by the dominated convergence theorem, $\lim\limits_{t \to +\infty}t^{\frac{\rank-1}{2}}J_t(\Psi_1,\Psi_2) = 0$. The proposition follows in this case since $\mathsf{M}^+(\Psi_2) = 0$ as $\int_{M_\Gamma} \langle \mu(m)w_1,w_2\rangle \, dm = 0$ due to $\mu \neq 1$. Using linearity of $J_t$, the proposition now extends for all $\Psi_2 \in \Span\mathfrak{F}$. 
		
		Now consider any $\Psi_2 \in C_\mathrm{c}(\tilde{\Sigma}^+)$. Let $\varepsilon > 0$. The Peter--Weyl theorem and the Stone--Weierstrass theorem gives $\Psi_3, \Psi_4 \in \Span\mathfrak{F}$ such that $|\Psi_2-\Psi_3| < \varepsilon$, $0 \le \Psi_4 \le 2$, $\Psi_4|_{\supp(\Psi_2-\Psi_3)} \ge 1$, and $\mathsf{M}^+(\supp(\Psi_3)), \mathsf{M}^+(\supp(\Psi_4)) \leq C\mathsf{M}^+(\supp(\Psi_2))$ for some $C > 0$ independent of $\varepsilon$. Then $|\Psi_2-\Psi_3| < \varepsilon\Psi_4$ and hence
		\begin{align*}
			&\limsup_{t \to +\infty} \left|t^{\frac{\rank-1}{2}}(J_t(\Psi_1,\Psi_2)-J_t(\Psi_1,\Psi_3))\right| \\
			\le{}&\lim_{t \to +\infty} t^{\frac{\rank-1}{2}}J_t(|\Psi_1|,\varepsilon\Psi_4) = (2\pi)^{-\frac{\rank - 1}{2}}\mathfrak{c}^{-\frac{1}{2}}e^{-\ell I(\mathsf{u})}\mathsf{M}^+(|\Psi_1|)\mathsf{M}^+(\Psi_4)\varepsilon \\
			\le{}&2C(2\pi)^{-\frac{\rank - 1}{2}}\mathfrak{c}^{-\frac{1}{2}}e^{-\ell I(\mathsf{u})}\mathsf{M}^+(|\Psi_1|)\mathsf{M}^+(\supp(\Psi_2))\varepsilon = C'\varepsilon
		\end{align*}
		where $C' > 0$ is a constant independent of $\varepsilon$. It follows that
		\begin{align*}
			\lim\limits_{t \to +\infty}t^{\frac{\rank-1}{2}}J_t(\Psi_1,\Psi_2) &= \lim\limits_{t \to +\infty}t^{\frac{\rank-1}{2}}J_t(\Psi_1,\Psi_3) + O(\varepsilon)
			\\
			&= (2\pi)^{-\frac{\rank - 1}{2}}\mathfrak{c}^{-\frac{1}{2}}e^{-\ell I(\mathsf{u})}\mathsf{M}^+(\Psi_1)\mathsf{M}^+(\Psi_2) + O(\varepsilon)
		\end{align*}
		since $\mathsf{M}^+(\Psi_3) = \mathsf{M}^+(\Psi_2) + O(\varepsilon)$. This completes the proof since $\varepsilon > 0$ is arbitrary.
	\end{proof}
	
	Finally, we prove the main theorem.
	
	\begin{proof}[Proof of \cref{thm:LocalMixingForErgodicComponent}]
		Let $\phi_1, \phi_2 \in C_{\mathrm{c}}(\Gamma \backslash G)$. We can assume $\supp(\phi_1), \supp(\phi_2) \subset \Omega_0$ and use the aforementioned semiconjugacy to view them as $\phi_1, \phi_2 \in C_{\mathrm{c}}(\Sigma^\varsigma)$. Let $\Phi_1, \Phi_2 \in C_\mathrm{c}(\tilde{\Sigma})$ be their lifts, i.e., we have
		\begin{equation*}
			\phi_j([(x,u,m,s)]) = \sum_{k \in \Z} (\Phi_j \circ \varsigma^k)(x,u,m,s)
		\end{equation*}
		for all $(x,u,m,s) \in \tilde{\Sigma}$ and $j \in \{1, 2\}$. Using a partition of unity, we can assume $\supp(\phi_2)$ is sufficiently small such that if $(x,u,m,s) \in \supp(\Phi_2)$, then $\varsigma^k(x,u,m,s) \allowbreak \notin \supp(\Phi_2)$ for all integers $k \ne 0$. Then we have
		\begin{align}
			\label{eqn:MixingIntegral}
			\begin{aligned}
				&\int_{\Gamma \backslash G} \phi_1([g] a_{t\mathsf{v} + r(t)\mathsf{u}}) \phi_2([g]) \, d\mathsf{m}([g]) 
				\\
				={}&\sum_{k \in \Z} \int_{\tilde{\Sigma}} (\Phi_1 \circ \varsigma^k)(x,u + r(t)\mathsf{u},m,s+t)\Phi_2(x,u,m,s) \, d\mathsf{M}(x,u,m,s)
			\end{aligned}
		\end{align}
		for all $t \geq 0$. In fact, the sum in \cref{eqn:MixingIntegral} is only over $k \in \mathbb Z_{\geq 0}$ because for all $(x, u, m, s) \in \supp(\Phi_2)$ and $t > 0$ sufficiently large, we have
		\begin{align*}
			&\varsigma^{-k}(x,u+r(t)\mathsf{u},m,s+t) 
			\\
			={}&(\sigma^{-k}(x), u + r(t)\mathsf{u} + \widehat{\mathsf{K}}_k(\sigma^{-k}(x)),\vartheta^k(\sigma^{-k}(x))m, s + t + \tau_k(\sigma^{-k}(x))) \notin \supp(\Phi_1)
		\end{align*}
		since $\tau_k(\sigma^{-k}(x)) \geq 0$ and $\supp(\Phi_1)$ is compact, for all $k \in \Z_{\geq 0}$. We now approximate $\Phi_1$ and $\Phi_2$ by $\tilde{\Phi}_1, \tilde{\Phi}_2 \in C_\mathrm{c}(\tilde{\Sigma})$ which are independent of past coordinates. Let $\varepsilon > 0$. By uniform continuity of $\Phi_1$ and $\Phi_2$ on their supports, there exists $N \in \N$ sufficiently large such that $|(\Phi_j\circ \varsigma^N)(x,u,m,s)-(\Phi_j\circ\varsigma^N)(y,u,m,s)| < \varepsilon$ for all $(x,u,m,s), (y,u,m,s) \in \tilde{\Sigma}$ and $j \in \{1, 2\}$ whenever $x_k = y_k$ for all $k \in \mathbb Z_{\geq 0}$. Let $\Sigma^+ \to \Sigma$ denoted by $x \mapsto \overline{x}$ be any continuous section of $\proj_{\Sigma^+}$. Define $\tilde{\Phi}_j \in C_\mathrm{c}(\tilde{\Sigma})$ and $\Psi_j \in C_\mathrm{c}(\tilde{\Sigma}^+)$ by $\tilde{\Phi}_j(x,u,m,s) = \Psi_j(\proj_{\Sigma^+}(x),u,m,s) = (\Phi_j\circ\sigma^N)(\overline{\proj_{\Sigma^+}(x)},u,m,s)$ for all $(x,u,m,s) \in \tilde{\Sigma}$ and $j \in \{1, 2\}$. Then, there exist bump functions $\Phi_3, \Phi_4 \in C_\mathrm{c}(\tilde{\Sigma})$ such that $|(\Phi_j \circ \varsigma^N) - \tilde{\Phi}_j| < \varepsilon\Phi_{j + 2}$ and so applying a similar argument as in \cref{pro:JtLimit} and the proposition itself gives
		\begin{align*}
			&\lim\limits_{t \to +\infty}t^{\frac{\rank-1}{2}}\int_{\Gamma \backslash G} \phi_1([g] a_{t\mathsf{v} + r(t)\mathsf{u}}) \phi_2([g]) \, d\mathsf{m}([g]) 
			\\
			={}&\lim\limits_{t \to +\infty}t^{\frac{\rank-1}{2}}J_t(\Psi_1, \Psi_2) + O(\varepsilon)
			\\
			={}&(2\pi)^{-\frac{\rank - 1}{2}} \mathfrak{c}^{-\frac{1}{2}} e^{-\ell I(\mathsf{u})}\mathsf{m}(\phi_1)\mathsf{m}(\phi_2) + O(\varepsilon)
		\end{align*}
		since $\mathsf{M}^+(\Psi_j) = \mathsf{M}^+(\Phi_j) + O(\varepsilon)$ and $\mathsf{M}^+(\Phi_j) = \mathsf{m}(\phi_j)$ for all $j \in \{1, 2\}$. This completes the proof since $\varepsilon > 0$ is arbitrary.
	\end{proof}
	
	\section{Topological mixing}
	\label{sec:TopologicalMixing}
	In this section, we prove the following theorem regarding topological mixing as an application of \cref{thm:LocalMixingNoPsi}. Recall the notation in \cref{subsec:ErgodicDecomposition}.
	
	\begin{theorem}
		\label{thm:TopologicalMixing}
		Let $G$ be a connected semisimple real algebraic group, $\Gamma < G$ be a Zariski dense discrete subgroup, and $\mathsf{v} \in \interior(\limitcone)$. Then the one-parameter diagonal flow $\{a_{t\mathsf{v}}: t \in \mathbb R\}$ on $\Gamma \backslash G$ is topologically mixing on $\Omega_{[m]}$ for all $[m] \in M_\Gamma \backslash M$.
	\end{theorem}
	
	\begin{proof}
		Let $\Gamma < G$ and $\mathsf{v} \in \interior(\limitcone)$ as in the theorem. Define $\Omega_\Gamma = \{[g] \in \Gamma \backslash G: g^\pm \in \limitset\}$ and let $\tilde{\Omega}_\Gamma \subset G$ be the $\Gamma$-invariant lift of $\Omega_\Gamma$. Define $\mathfrak{Z}_\Gamma = \{\Omega_{[m]}:[m] \in M_\Gamma \backslash M\}$ and $\tilde{\mathfrak{Z}}_\Gamma = \{\tilde{Z}: Z \in \mathfrak{Z}_\Gamma\}$ where $\tilde{Z} \subset G$ denotes the $\Gamma$-invariant lift of $Z \in \mathfrak{Z}_\Gamma$. Then $\mathfrak{Z}_\Gamma$ and $\tilde{\mathfrak{Z}}_\Gamma$ give partitions of $\Omega_\Gamma$ and $\tilde{\Omega}_\Gamma$ respectively. We use similar notations for other discrete subgroups of $G$. Fix $Z \in \mathfrak{Z}_\Gamma$ and $\Gamma$-invariant open sets $\mathcal{O}_1,\mathcal{O}_2 \subset G$ such that $\Gamma \backslash \mathcal{O}_1 \cap Z$ and $\Gamma \backslash \mathcal{O}_2 \cap Z$ are nonempty. For all loxodromic $\gamma \in \Gamma$, we denote its attracting and repelling fixed points by $\gamma_+ \in \limitset$ and $\gamma_- \in \limitset$ respectively. By \cite[Proposition 3.6]{Ben97}, the subset $\{(\gamma_+,\gamma_-): \gamma \in \Gamma \text{ is loxodromic}\} \subset \limitset \times \limitset$ is dense. Hence, for all $j \in \{1, 2\}$ there exist a loxodromic $\gamma_j \in \Gamma$ and $[g_j] \in \Gamma \backslash \mathcal{O}_j \cap Z$ such that $g_j^\pm = (\gamma_j)_\pm$. By \cite[Proposition 4.3]{Ben97}, there exists a Zariski dense Schottky subgroup $\Gamma' < \Gamma$ generated by $\{\gamma_1',\gamma_2',\dotsc,\gamma_k'\} \subset \Gamma$ such that $\mathsf{v} \in \interior(\mathcal{L}_{\Gamma'})$. Moreover, we can choose $\gamma_1'', \gamma_2'', \dotsc, \gamma_l'' \in \Gamma$ such that $\{\lambda^M(\gamma_j''): 1 \leq j \leq l\}$ intersects all connected components of $M_\Gamma$ by \cref{pro:SubgroupOfGeneralizedLengthSpectrumOfGamma_0DenseInAM_Gamma_0}. Taking appropriate sufficiently large powers of $\gamma_1,\gamma_2, \gamma_1',\gamma_2',\dotsc,\gamma_k',\gamma_1'',\gamma_2'',\dotsc,\gamma_l''$, we obtain a Zariski dense Schottky subgroup $\Gamma_0 < \Gamma$ such that $\mathsf{v} \in \interior(\mathcal{L}_{\Gamma_0})$, $\mathcal{O}_1$ and $\mathcal{O}_2$ intersect $\tilde{\Omega}_{\Gamma_0}$, and $M_{\Gamma_0} = M_\Gamma$ again by \cref{pro:SubgroupOfGeneralizedLengthSpectrumOfGamma_0DenseInAM_Gamma_0}. Then, there is a unique $Z_0 \in \mathfrak{Z}_{\Gamma_0}$ such that $\tilde{Z}_0 \subset \tilde{Z}$. By construction, $\mathcal{O}_1$ and $\mathcal{O}_2$ intersect $\tilde{Z}_0$. By \cite[Lemma 7.2]{ELO20}, $\Gamma_0 < G$ is a Zariski dense Anosov subgroup and \cref{thm:LocalMixingNoPsi} implies that the one-parameter diagonal flow $\{a_{t\mathsf{v}}: t \in \mathbb R\}$ on $\Gamma_0 \backslash G$ is topologically mixing on $Z_0 \in \mathfrak{Z}_{\Gamma_0}$. Thus, we have $\mathcal{O}_1 a_{t\mathsf{v}} \cap \mathcal{O}_2 \neq \varnothing$ for all sufficiently large $t$.
	\end{proof}
	
	Dang proved \cref{thm:TopologicalMixing} in the case that $M^\circ$ is abelian \cite[Theorem 1.4(d)]{Dan20}. She also proved the following.
	
	\begin{theorem}[{\citealp[Theorem 1.4(c)]{Dan20}}]
		Let $G$ be a connected semisimple real algebraic group, $\Gamma < G$ be a Zariski dense discrete subgroup, $\mathsf{v} \in \interior(\LieA^+)$, and $[m] \in M_\Gamma \backslash M$. If the one-parameter diagonal flow $\{a_{t\mathsf{v}}: t \in \mathbb R\}$ is topologically mixing on $\Omega_{[m]}$, then $\mathsf{v} \in \interior(\limitcone)$.
	\end{theorem}
	
	Combining Dang's result with \cref{thm:TopologicalMixing} yields the following theorem.
	
	\begin{theorem}
		Let $G$ be a connected semisimple real algebraic group, $\Gamma < G$ be a Zariski dense discrete subgroup, $\mathsf{v} \in \interior(\LieA^+)$, and $[m] \in M_\Gamma \backslash M$. Then the one-parameter diagonal flow $\{a_{t\mathsf{v}}: t \in \mathbb R\}$ on $\Gamma \backslash G$ is topologically mixing on $\Omega_{[m]}$ if and only if $\mathsf{v} \in \interior(\limitcone)$.
	\end{theorem}
	
	\nocite{*}
	\renewcommand*{\bibfont}{\footnotesize}
	\bibliographystyle{alpha_name-year-title}
	\bibliography{References}
\end{document}